\documentclass[a4paper, 11pt,reqno]{amsart}

\title{Cobordism and Concordance of Surfaces in 4-Manifolds}
\author{Simeon Hellsten}
\address{School of Mathematics and Statistics, University of Glasgow, United Kingdom}
\email{simeon.hellsten@glasgow.ac.uk}

\def\subjclassname{\textup{2020} Mathematics Subject Classification}
\expandafter\let\csname subjclassname@1991\endcsname=\subjclassname

\usepackage[margin=25mm]{geometry}
\usepackage{amssymb, amsthm, amsmath}
\usepackage{thmtools, stmaryrd, mathrsfs}

\usepackage[nodisplayskipstretch]{setspace}

\usepackage{enumitem}
\usepackage{mathtools}
\usepackage{wrapfig,caption}

\usepackage{tikz-cd}
\usepackage{tikz-3dplot}

\usepackage{xcolor}
\usepackage{microtype}

\usepackage{hyperref}
\hypersetup{hidelinks}
\usepackage[nobysame, alphabetic]{amsrefs}
\usepackage{zref-clever}
\zcsetup{cap=true,comp=false,sort=false}
\newcommand{\cref}{\zcref[S]}

\newcommand{\stimes}{\mathord{\times}}
\newcommand{\cyc}[1]{{\Z/{#1}}}
\renewcommand{\iff}{\Leftrightarrow}
\newcommand{\inv}{^{-1}}

\newcommand{\restr}[3][]{\mathchoice{%
\restriction{#2}{#3}{\displaystyle}{#1}}{%
\restriction{#2}{#3}{\textstyle}{#1}}{%
\restriction{#2}{#3}{\scriptstyle}{#1}}{%
\restriction{#2}{#3}{\scriptscriptstyle}{#1}}}

\newlength{\totbarheight}
\newlength{\bardepth}

\renewcommand{\restriction}[4]{
\settodepth{\bardepth}{\(#3 #1\)} 
\settoheight{\totbarheight}{\(#3 #1\)}
\addtolength{\totbarheight}{\bardepth}
\addtolength{\totbarheight}{1ex}
\addtolength{\bardepth}{0.7ex}
{#1 
\rule[-\bardepth]{0ex}{\totbarheight} 
\mkern 2mu \vrule width 0.14ex \mkern 2.5mu}
{\addtolength{\bardepth}{-0.2ex}
\addtolength{\totbarheight}{-0.2ex}
\rule[-\bardepth]{0ex}{\totbarheight}}^{#4}_{#2}
}

\DeclareMathOperator{\pr}{pr}
\DeclareMathOperator{\pt}{pt}

\newcommand{\C}{\mathbb{C}}
\newcommand{\R}{\mathbb{R}}
\renewcommand{\P}{\mathbb{P}}
\newcommand{\Z}{\mathbb{Z}}

\renewcommand{\S}{\Sigma}

 \newcommand{\cI}{\mathcal{I}} 
 \newcommand{\cP}{\mathcal{P}} 

\newcommand{\del}{\partial}

\DeclareMathOperator{\fr}{fr}  
\DeclareMathOperator{\self}{self}

\DeclareMathOperator{\BO}{BO}
\DeclareMathOperator{\BSO}{BSO}
\DeclareMathOperator{\BSpin}{BSpin}
\DeclareMathOperator{\BPin}{BPin}

\DeclareMathOperator{\lk}{lk}
\DeclareMathOperator{\PD}{PD}

\newcommand{\immerse}{\looparrowright}

\DeclareMathOperator{\Aut}{Aut}
\DeclareMathOperator*{\colim}{colim}
\DeclareMathOperator{\Hom}{Hom}
\DeclareMathOperator{\Map}{Map}
\DeclareMathOperator{\id}{id}
\DeclareMathOperator{\im}{im}
\DeclareMathOperator{\rank}{rank}

\renewcommand{\O}{\mathrm{O}}
\DeclareMathOperator{\SO}{SO}
\DeclareMathOperator{\Spin}{Spin}
\DeclareMathOperator{\Pin}{Pin}

\newcommand{\floor}[1]{\left\lfloor #1 \right\rfloor}

\newcommand{\abs}[1]{\left\lvert #1 \right\rvert}

\newcommand{\wh}{\widehat}
\renewcommand{\bar}{\overline}

\theoremstyle{plain}
\declaretheorem[name=Theorem,sibling=theorem,numberwithin=section]{theorem}
\declaretheorem[name=Theorem]{mainthm}
\declaretheorem[name=Corollary,sibling=mainthm]{maincor}
\declaretheorem[name=Proposition,sibling=theorem]{proposition}
\declaretheorem[name=Lemma,sibling=theorem]{lemma}
\declaretheorem[name=Corollary,sibling=theorem]{corollary}

\theoremstyle{definition}
\declaretheorem[name=Definition, sibling=theorem]{definition}
\declaretheorem[name=Remark,sibling=theorem]{remark}
\declaretheorem[name=Notation, sibling=theorem]{notation}

\newtheoremstyle{named}{}{}{\itshape}{}{\bfseries}{.}{.5em}{\thmnote{#3}}
\theoremstyle{named}
\newtheorem*{namedtheorem}{Theorem}

\begin{document}

\begin{abstract}
    We show that two properly embedded compact surfaces in an orientable 4-manifold are cobordant if and only if they are $\mathbb{Z}/2$-homologous and either the 4-manifold has boundary or the surfaces have the same normal Euler number. 
    If the 4-manifold is simply-connected and the surfaces are closed, non-orientable, and cobordant, we show that they are in fact concordant. This completes the classification of closed surfaces in simply-connected 4-manifolds up to concordance.
    Our methods give new constructions of cobordisms with prescribed boundaries, and completely determine when a given cobordism between the boundaries extends to a cobordism or concordance between the surfaces. We obtain our concordance results by extending Sunukjian's method of ambient surgery to the unoriented case using Pin$^-$-structures. 
    We also discuss conditions for an arbitrary codimension 2 properly embedded submanifold to admit an unoriented spanning manifold with prescribed boundary. All results hold in both the smooth and topological categories.
\end{abstract}
\maketitle
\thispagestyle{empty}

\vspace{-1em}
 \section{Introduction}\label{sec: introduction}
{\setstretch{1.09}
Let $X$ be a $4$-manifold and let $\S_0, \S_1 \subset X$ be properly embedded compact surfaces. A \textit{cobordism} from $\S_0$ to $\S_1$ is a properly embedded compact 3-manifold with corners $Y \subset X \times I$ such that $Y \cap (X \times \{i\}) = \S_i \times \{i\}$ for $i=0,1$. We say that a cobordism $Y$ is a \textit{concordance} if $\S_0 \cong \S_1$ and $Y \cong \S_0 \times I$. 
If $\del \S_0 = \del \S_1$ and  $Y \cap (\del X \times I) = \del \S_0 \times I$, then we call $Y$ a \textit{cobordism} (resp.\ \textit{concordance}) \textit{rel.\ boundary}.

Kervaire began the study of concordance of surfaces in 4-manifolds, using surgery techniques to show that if $X = S^4$ and $\S_0 \cong \S_1 \cong S^2$, then $\S_0$ and $\S_1$ are concordant \cite[Th\'eor\`eme III.6]{Kervaire}. 
This was extended to all pairs of closed connected surfaces in $X=S^4$ by Blanl{\oe}il--Saeki, who showed that $\S_0, \S_1\subset S^4$ are concordant if and only if $\S_0 \cong \S_1$ and they have the same normal Euler number \cite[Corollaire 3.1]{BlanloeilSaeki}. Their approach used $\Pin^-$-structures to study when a closed connected surface in $S^4 = \del D^5$ bounds a handlebody in $D^5$.
Sunukjian adapted Kervaire's techniques, using $\Spin$-structures to prove that for any closed simply-connected 4-manifold $X$, two closed connected orientable surfaces $\S_0, \S_1 \subset X$ are concordant if and only if $\S_0 \cong \S_1$ and they are $\Z$-homologous after some choice of orientations \cite[Theorem 6.1]{Sunukjian}. 

In this paper, we extend these results to arbitrary simply-connected 4-manifolds $X$ and arbitrary compact connected surfaces, possibly non-orientable and with non-empty boundary.

\begin{mainthm}\label{mainthm: concordance rel boundary}
    Let $X$ be a simply-connected 4-manifold, and let $\S_0, \S_1 \subset X$ be properly embedded compact connected surfaces. Suppose that $\del \S_0 = \del \S_1$. Then $\S_0$ and $\S_1$ are concordant rel.\ boundary if and only if $\S_0 \cong \S_1$ and either:
    \begin{enumerate}[label=(\roman*)]
        \item $\S_0$ and $\S_1$ are orientable, and $[\S_0 \cup \S_1] = 0 \in H_2(X;\Z)$ after a choice of orientation; or
        \item $\S_0$ and $\S_1$ are non-orientable, $[\S_0 \cup \S_1] = 0 \in H_2(X;\cyc{2})$, and $e(\S_0,s)  = e(\S_1,s)$ after an arbitrary choice of framing $s$ of $\del \S_0$.
    \end{enumerate}
\end{mainthm}
Here $e(\S_i,s)$ is the normal Euler number of $\S_i$ relative to the framing~$s$ of the link $\del \S_i \subset \del X$. 
Although one must choose an orientation of $X$ to compute $e(\S_0,s)$ and $e(\S_1,s)$, reversing the orientation changes the sign of both quantities. Similarly, changing the framing $s$ changes both quantities by the same amount, hence the equality $e(\S_0,s) = e(\S_1,s)$ is either true or false independent of these choices.
Theorem~\ref{mainthm: concordance rel boundary}, as well as all other results in this paper, holds in both the smooth and topological categories, and does not require $X$ to be compact.

By applying Theorem~\ref{mainthm: concordance rel boundary} with $\S_0$ and $\S_1$ closed, we complete the classification of closed connected surfaces in simply-connected 4-manifolds up to concordance.
}

\begin{maincor}\label{maincor: closed concordance}
    Let $X$ be a simply-connected 4-manifold and let $\S_0, \S_1 \subset X$ be embedded closed connected surfaces. Then $\S_0$ and $\S_1$ are concordant if and only if $\S_0 \cong \S_1$ and either:
    \begin{enumerate}[label=(\roman*)]
        \item $\S_0$ and $\S_1$ are orientable and $[\S_0] = [\S_1] \in H_2(X;\Z)$ with some choice of orientations;~or
        \item $\S_0$ and $\S_1$ are non-orientable, $[\S_0] = [\S_1] \in H_2(X;\cyc{2})$, and $e(\S_0) = e(\S_1)$.
    \end{enumerate}
\end{maincor}

As in \cites{Kervaire,BlanloeilSaeki,Sunukjian}, the proof of Theorem~\ref{mainthm: concordance rel boundary} proceeds in two steps. The first step is to reduce the existence of a concordance rel.\ boundary to the existence of a cobordism rel.\ boundary. More generally, we say that a cobordism $Z \subset \del X \times I$ from $\del \S_0$ to $\del \S_1$ \textit{extends to} a cobordism $Y \subset X \times I$ from $\S_0$ to $\S_1$ if $Z = Y \cap (\del X \times I)$; equivalently, we say that $Y$ \textit{extends} $Z$. We show that if a concordance from $\del \S_0$ to $\del \S_1$ extends to a cobordism from $\S_0$ to $\S_1$ which is orientable if $\S_0$ and $\S_1$ are, then it in fact extends to a concordance.

\begin{mainthm}\label{mainthm: general concordance}
    Let $X$ be a simply-connected 4-manifold and let $\S_0,\S_1 \subset X$ be properly embedded compact connected surfaces. Let $Z \subset \del X \times I$ be a concordance from $\del \S_0$ to $\del \S_1$. Then $Z$ extends to a concordance from $\S_0$ to $\S_1$ if and only if $\S_0 \cong \S_1$ and either:
    \begin{enumerate}[label=(\roman*)]
        \item $\S_0$ and $\S_1$ are orientable, and $Z$ extends to an orientable cobordism from $\S_0$ to $\S_1$; or
        \item $\S_0$ and $\S_1$ are non-orientable, and $Z$ extends to a cobordism from $\S_0$ to $\S_1$.
    \end{enumerate}
\end{mainthm}

We prove Theorem~\ref{mainthm: general concordance} by showing that when $X$ is simply-connected, any cobordism extending $Z$ can be ambiently surgered into a concordance in $X \times I$. We do this by adapting Sunukjian's methods of $\Spin$-surgery to the unoriented setting, by replacing $\Spin$-structures with their unoriented analogues, $\Pin^\pm$-structures. See \cref{sec: intro concordance} below for more details. 

\begin{wrapfigure}{r}{0.4\textwidth} 
    \vspace{-0.4cm}
    \begin{center}
    \begin{tikzpicture}
\begin{scope}[rotate around y=-7]

\path (1,0,0);
\pgfgetlastxy{\cylxx}{\cylxy}
\path (0,1,0);
\pgfgetlastxy{\cylyx}{\cylyy}
\path (0,0,1);
\pgfgetlastxy{\cylzx}{\cylzy}
\pgfmathsetmacro{\cylt}{(\cylzy * \cylyx - \cylzx * \cylyy)/ (\cylzy * \cylxx - \cylzx * \cylxy)}
\pgfmathsetmacro{\ang}{atan(\cylt)}
\pgfmathsetmacro{\ct}{1/sqrt(1 + (\cylt)^2)}
\pgfmathsetmacro{\st}{\cylt * \ct}

\pgfmathsetmacro{\myl}{5}

\fill[cyan!50!blue,opacity=0.1] (\ct,\st,0) -- ++(0,0,-\myl) arc[start angle=\ang,delta angle=180,radius=1] -- ++(0,0,\myl) arc[start angle=\ang+180,delta angle=-180,radius=1];

\fill[cyan!50!blue, fill opacity=0.1] (0,0,0) circle[radius=1];

\draw[green!50!black, thick]
      plot[variable=\t,domain=-1:1,samples=50,smooth]
      ({sqrt(0.25)*\t},
       {sqrt(1-0.25*\t*\t)},
       {-2*cos(\t*90)});

\draw[green!50!black, thick, dashed]
      plot[variable=\t,domain=180:240,samples=50,smooth]
      ({cos(\t)},
       {sin(\t)},
        {\myl*\t/60-4*\myl});

 \draw[green!50!black, thick]
      plot[variable=\t,domain=180:230,samples=50,smooth]
      ({-cos(\t)},
       {sin(\t)},
       {\myl*\t/60-4*\myl});

\draw[green!50!black, thick,dashed]
      plot[variable=\t,domain=230:240,samples=50,smooth]
      ({-cos(\t)},
       {sin(\t)},
       {\myl*\t/60-4*\myl});

\draw[red,thick] ({sqrt(0.25)},{sqrt(0.75)},0)--({sqrt(0.25)},-{sqrt(0.75)},0);
\draw[red,thick] (-{sqrt(0.25)},{sqrt(0.75)},0)--(-{sqrt(0.25)},-{sqrt(0.75)},0);

\draw[red,dashed,thick] (-1,0,-\myl)--(1,0,-\myl);

\draw (0,0,0) circle[radius=1];

\draw (\ct,\st,0) -- ++(0,0,-\myl);
\draw (-\ct,-\st,0) -- ++(0,0,-\myl);
\draw (\ct,\st,-\myl) arc[start angle=\ang,delta angle=180,radius=1];
\draw[dashed] (\ct,\st,-\myl) arc[start angle=\ang,delta angle=-180,radius=1];

\node[left] at (-1,1.2,-\myl/2) {$X \times I$};

\node[red!50!black] at (0,0,0) {$\S_0$};
\node[red!50!black] at (0,0.3,-\myl) {$\S_1$};

\node[green!30!black] at (0.2,-0,-\myl/2) {$Z$};
\draw[green!30!black,->] (0.4,-0.05,-\myl/2) -- ++ (0.3,-0.3,0);
\draw[green!30!black,->] (0,-0.1,-\myl/2) -- ++ (-0.45,-0.05,0);
\draw[green!30!black,->] (0.05,0.15,-\myl/2) -- ++ (-0.35,0.3,0);

\end{scope}
\end{tikzpicture}
    \caption{Example of input data for Theorem \ref{mainthm: general cobordism}, with reduced dimensions. Here $X=D^2$, $\S_0 \cong I \sqcup I$, $\S_1 \cong I$, and $Z \cong I \sqcup I \sqcup I$. In this example, $Z$ extends to a cobordism from $\S_0$ to $\S_1$ (not shown).}
    \label{fig: excob}
    \end{center}
    \vspace{-1em}
\end{wrapfigure}
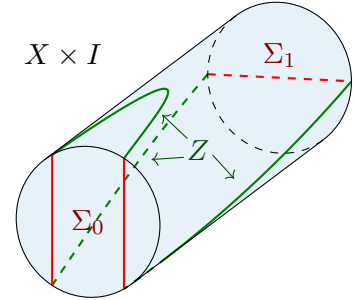

The second step in the proof of Theorem~\ref{mainthm: concordance rel boundary} is to classify when a given cobordism from $\del \S_0$ to $\del \S_1$ extends to some cobordism between the surfaces. We do this for all orientable 4-manifolds $X$; see Theorem~\ref{mainthm: general cobordism} below. 

Perhaps surprisingly, this is the more technically difficult step, since the classical approaches in e.g.\ \cite{ThomQuelqueProprietes} cannot be easily adapted from the oriented setting to the unoriented setting. 
Instead, we introduce a novel method for constructing unoriented cobordisms between properly embedded surfaces. Our method relies on a new notion of a spanning 3-manifold for an immersed surface in a 4-manifold, which is inspired by ideas in \cite{BaykurSunukjian}. See \cref{sec: intro cobordisms} for more details on this construction. Below, we write $\pr_X \colon X \times I \to X$ for the projection map. See \cref{fig: excob} for an example of the set-up of Theorem~\ref{mainthm: general cobordism} in reduced dimensions.

\begin{mainthm}\label{mainthm: general cobordism}
    Let $X$ be a connected orientable 4-manifold, and let $\S_0, \S_1 \subset X$ be properly embedded compact surfaces. 
    Let $Z \subset \del X \times I$ be a cobordism from $\del \S_0$ to $\del \S_1$.
    Then $Z$ extends to a cobordism from $\S_0$ to $\S_1$ if and only if $[\S_0 \cup \pr_X(Z) \cup \S_1] = 0 \in H_2(X;\cyc{2})$ and $e(\S_1,s_1) = e(\S_0,s_0) + e(Z,s_0\cup s_1)$ after an arbitrary choice of framing $s_i$ of $\del \S_i$ for $i=0,1$.
\end{mainthm}

As a consequence of Theorem~\ref{mainthm: general cobordism}, we complete the classification of compact surfaces in orientable 4-manifolds up to cobordism. This extends results of Carter--Kamada--Saito--Satoh in the case that $\S_0$ and $\S_1$ are connected and $X = \R^4$ \cite{CKSSBordism}.

\begin{mainthm}\label{mainthm: cobordism}
    Let $X$ be a connected orientable 4-manifold and let $\S_0, \S_1 \subset X$ be properly~embedded compact surfaces. Then $\S_0$ and $\S_1$ are cobordant if and only if $[\S_0] = [\S_1] \in H_2(X, \del X;\cyc{2})$ and either $\del X \neq \varnothing$ or $e(\S_0) = e(\S_1)$.
\end{mainthm}

We deduce Theorem~\ref{mainthm: cobordism} from Theorem~\ref{mainthm: general cobordism} and a result of Whitney stating that for any $n \in \Z$, there is an embedded compact surface $F \subset D^4$ with $e(F) = 2n$ \cite{WhitneyDiffMan}. This allows us to modify a fixed initial cobordism on the boundary to satisfy the conditions in Theorem~\ref{mainthm: general cobordism}, by introducing a suitable disjoint closed surface in $\del X \times I$ of the correct normal Euler number; see \cref{sec: theorem E}.

Theorem~\ref{mainthm: general cobordism} also lets us complete the classification of compact surfaces in orientable 4-manifolds up to cobordism rel.\ boundary; see \cref{sec: cobordisms rel boundary}

\begin{maincor}\label{maincor: cobordism rel boundary}
    Let $X$ be an orientable 4-manifold and let $\S_0,\S_1 \subset X$ be properly embedded compact surfaces such that $\del \S_0 = \del \S_1$. Then $\S_0$ and $\S_1$ are  cobordant rel.\ boundary if and only if $[\S_0 \cup \S_1] = 0 \in H_2(X;\cyc{2})$ and $e(\S_0,s)  = e(\S_1,s)$ after an arbitrary choice of framing~$s$~of~$\del \S_0$.
\end{maincor}

In order to complete the proof of Theorem~\ref{mainthm: concordance rel boundary}, we also need to complete the classification of orientable surfaces up to orientable cobordisms rel.\ boundary.
This is due to Thom in the case that $X$ is closed \cite[Th\'eor\`eme IV.6]{ThomQuelqueProprietes}. We extend these methods, using similar classical techniques, to arbitrary $X$ and arbitrary orientable cobordisms on the boundary.

\begin{theorem}[\cite{ThomQuelqueProprietes}]\label{thm: oriented cobordism}
    Let $X$ be a connected oriented $(n+2)$-manifold and let $\S_0, \S_1 \subset X$ be oriented properly embedded compact $n$-manifolds. Then there is an oriented cobordism from $\S_0$ to $\S_1$ if and only if $[\S_0] = [\S_1] \in H_2(X, \del X;\Z)$.
\end{theorem}

\begin{theorem}\label{thm: oriented general cobordism}
        Let $X$ and $\S_0,\S_1 \subset X$ be as in \cref{thm: oriented cobordism}.
        Let $Z \subset \del X \times I$ be an oriented cobordism from $\del \S_0$ to $\del \S_1$. Then $Z$ extends to an oriented cobordism from $\S_0$ to $\S_1$ if and only if $[-\S_0 \cup -\pr_{X}(Z) \cup \S_1] = 0 \in H_n(X;\Z).$
\end{theorem}
Here an oriented cobordism from $\S_0$ to $\S_1$ is a properly embedded compact oriented $(n+1)$-manifold $Y \subset X \times I$ such that $\del Y \cap (X \times \{0\} ) = -\S_0$ and $\del Y \cap (X \times \{1\}) = \S_1$ as oriented manifolds. 
Both Theorems~\ref{thm: oriented cobordism} and \ref{thm: oriented general cobordism} hold for codimension 2 proper embeddings in all ambient dimensions.
See \cref{sec: intro oriented} for an outline of the proof of Theorems~\ref{thm: oriented cobordism} and \ref{thm: oriented general cobordism}, and a discussion about why these classical techniques do not adapt to the unoriented setting.

We can now complete the proof of Theorem~\ref{mainthm: concordance rel boundary}.

\begin{proof}[Proof of Theorem~\ref{mainthm: concordance rel boundary}]
    Let $X$ be a simply-connected 4-manifold, and let $\S_0, \S_1 \subset X$ be properly embedded compact connected surfaces with $\del \S_0 = \del \S_1$. 
    By Theorem~\ref{mainthm: general concordance}, they are concordant rel.\ boundary if and only if $\S_0 \cong \S_1$ and there is a cobordism rel.\ boundary from $\S_0$ to $\S_1$, which is orientable if $\S_0$ and $\S_1$ are. Corollary~\ref{maincor: cobordism rel boundary} says that they are cobordant rel.\ boundary if and only if they are $\cyc{2}$-homologous and $e(\S_0,s) = e(\S_1,s)$ after a choice of framing $s$ of $\del \S_0$; \cref{thm: oriented general cobordism} says that this cobordism can be chosen to be orientable if and only if $[\S_0 \cup \S_1] = 0 \in H_2(X;\Z)$ after a choice of orientation. These two facts prove Theorem~\ref{mainthm: concordance rel boundary} in the cases that $\S_0$ and $\S_1$ are non-orientable and orientable respectively.
\end{proof}

\subsection{$\Pin^\pm$-surgery and concordance}\label{sec: intro concordance}
We now discuss the proof of Theorem~\ref{mainthm: general concordance} in more detail, and explain why a cobordism between compact connected surfaces in a simply-connected 4-manifold can be surgered into a concordance. We first consider Sunukjian's approach in the oriented setting \cite{Sunukjian}, before discussing our adaptations to the unoriented setting.

Given a compact connected oriented 3-manifold $Y$, properly embedded in a simply-connected 5-manifold $X$, Sunukjian showed that $\Spin$-structures on $Y$ can be used to control which integral Dehn surgeries on $Y$ can be performed ambiently in $X$.
We understand this as follows. An abstract surgery on a knot $K \subset Y$ is specified by a framing of $N_YK$. Whether this surgery can be realised ambiently in $X$ depends only on the stable equivalence class of this framing. 
A $\Spin$-structure can be interpreted as a compatible choice of stable framing of $N_YK$ for each $K \subset Y$, so one can hope to construct a $\Spin$-structure on $Y$ given by stable framings which yield ambient surgery. By performing weak internal stabilisations (i.e.\ ambient self-connected sums), we can always arrange that $Y$ admits such a $\Spin$-structure, and so all abstract integral Dehn surgeries which preserve this $\Spin$-structure can be performed ambiently in $X$. 

If $Y$ is an oriented cobordism between diffeomorphic surfaces $\S_0 \cong \S_1$, which extends a concordance from $\del \S_0$ to $\del \S_1$, the Lickorish--Wallace theorem says that there is a sequence of abstract integral Dehn surgeries taking $Y$ to $\S_0 \times I$. Kaplan showed that that these surgeries can be chosen to preserve any fixed $\Spin$-structure on $Y$ \cite{Kaplan}. In particular, with the $\Spin$-structure on $Y$ described above, they can be realised ambiently in $X$, and hence $Y$ can be surgered into a concordance. This proves Theorem~\ref{mainthm: general concordance} in the case that $\S_0$ and $\S_1$ are orientable.

We adapt these ideas to the unoriented setting by replacing $\Spin$-structures with their unoriented analogues, $\Pin^\pm$-structures. If $Y$ is a 3-manifold, not necessarily orientable, properly embedded in an orientable 5-manifold $X$, a $\Pin^-$-structure on $Y$ can be interpreted as a compatible choice of framing of $\restr{TX}{K}$ for each knot $K \subset Y$. We use this perspective to show that many important results from \cite{Sunukjian} carry over to the unoriented setting. In particular, if $X$ is simply-connected, we can perform weak internal stabilisations to arrange that $Y$ admits a $\Pin^-$-structure such that all abstract surgeries on $Y$ which preserve the $\Pin^-$-structure can be realised ambiently. Our proof of this corrects an error in the statement and proof of Proposition 5.1 of \cite{Sunukjian} (see \cref{rmk: proof of 5.1}). We also give a non-orientable analogue of Kaplan's theorem, and hence prove Theorem~\ref{mainthm: general concordance} in the case that $\S_0$ and $\S_1$ are non-orientable.

We do not consider concordance when $\pi_1(X) \neq 1$, since the $\Pin^\pm$-surgery techniques described here do not yield very interesting results. Instead, we refer the reader to \cite{FreedmanQuinn}, \cite{Stong}, and \cites{KlugMiller1, KlugMiller2} for results in these cases.

\subsection{Construction of cobordisms}\label{sec: intro cobordisms} In this subsection, we describe our method of constructing unoriented cobordisms between surfaces, and give an outline of the proof of Theorem~\ref{mainthm: general cobordism}.

Let $X$ be an orientable 4-manifold, and let $\S_0, \S_1 \subset X$ be properly embedded compact surfaces. For simplicity, first assume that $\del X = \varnothing$, and that $\S_0$ and $\S_1$ are closed. 
Write $\wh{X} \subseteq X$ for the 4-manifold given by puncturing $X$ at the points of intersection between $\S_0$ and $\S_1$.
Then $\wh{\S} \coloneq (\S_0 \cup \S_1) \cap \wh{X}$ is a compact surface properly embedded in $\wh{X}$. Note that $\del \wh{X}$ is a disjoint union of one copy of $S^3$ for each intersection point in $\S_0\cap \S_1$, and $\wh{\S}$ meets each component of $\del \wh{X}$ in a Hopf link, where one component comes from $\S_0$ and the other from $\S_1$. 

Fix a compact surface $\wh{Z} \subset \del \wh{X}$ which is a disjoint union of annuli and satisfies $\del \wh{Z} = \del \wh{\S}$; that is, $\wh{Z}$ is a union of an annular Seifert surfaces for the Hopf links which make up $\del \wh{\S}$. 
Now suppose that there exists a 3-manifold with corners $\wh{Y} \subset \wh{X}$ such that $\del \wh{Y} = \wh{\S} \cup \wh{Z}$. We call $\wh{Y}$ a \textit{spanning 3-manifold} for $\wh{\S}$. We can then ``stretch out" $\wh{Y}$ into a cobordism from $\S_0 \cap \wh{X}$ to $\S_1 \cap \wh{X}$ as follows. For any map $f \colon \wh{Y} \to I$ such that $f\inv(\{i\}) = \S_i \cap \wh{X}$ for $i=0,1$, the submanifold $\{(y, f(y)) \mid y \in \wh{Y}\} \subset \wh{X} \times I$ is such a cobordism. The boundary of this cobordism is the trace of an isotopy from $\del \wh{\S}_0$ to $\del \wh{\S}_1$, so it can be extended around the punctures to give a cobordism $Y \subset X \times I$ from $\S_0$ to $\S_1$.

The existence of a cobordism from $\S_0$ to $\S_1$ thus reduces to the existence of such a spanning 3-manifold $\wh{Y}$. We break this question into two parts: first, when does a fixed choice of annuli $\wh{Z}$ extend to a spanning 3-manifold $\wh{Y}$; and secondly, when can $\S_0$ and $\S_1$ be modified by isotopies to ensure that such a choice $\wh{Z}$ exists.

We answer the first question via a thorough study on the existence of spanning manifolds of codimension 2 proper embeddings; see \cref{sec: intro spanning manifolds} for a summary of our results, or \cref{sec: Seifert manifolds} for our full discussion. We show that $\wh{Z}$ extends to a spanning 3-manifold if and only if $[\S_0] = [\S_1] \in H_2(X; \cyc{2})$ and the normal Euler number $e(S,s^Z) = 0$ for each component $S$ of $\S_0$ or $\S_1$, where $s^Z$ is the framing of $\del \wh{\S}_0 \cup \del \wh{\S}_1 \subset \del \wh{X}$ given by the normal direction into $\wh{Z}$. 

Then, by using a combinatorial argument, we show that the normal Euler number condition can be satisfied after isotopies of $\S_0$ and $\S_1$ if and only if $e(\S_0) = e(\S_1)$. See \cref{sec: new surfaces sec} for details of the argument, and more general results. We use carefully chosen finger moves between $\S_0$ and $\S_1$ to find suitable isotopies, and introduce a new type of diagram to help keep track of finger moves and their effects on the normal Euler number. We do this in the more general setting of a proper immersion of a surface in an orientable 4-manifold, then at the end reduce to the case of an immersion given by the union of two embeddings.

Taken together, this shows that $\S_0$ and $\S_1$ are cobordant whenever they are $\cyc{2}$-homologous and $e(\S_0)=e(\S_1)$. These conditions are also necessary, proving Theorem~\ref{mainthm: general cobordism} whenever $\del X = \varnothing$.

Now consider the general case, where $\del X$, $\del \S_0$, and $\del \S_1$ may all be non-empty. Fix a cobordism $Z$ from $\del \S_0$ to $\del \S_1$, and let $ \S' \coloneq \S_0 \stimes \{0\} \cup Z \cup \S_1 \stimes \{1\} \subset \del (X \times I)$.
Then $\S'$ is a closed surface embedded in the closed orientable 4-manifold $\del (X \times I)$. We use the closed case outlined above and a homological argument to show that there exists a properly embedded compact surface $M \subset X$ such that $\S'$ is cobordant to $\del (M \times I)$ in $\del (X \times I)$. By embedding this cobordism in a collar of $\del (X \times I)$, and filling it in with $M \times I$, we build a cobordism from $\S_0$ to $\S_1$ which extends $Z$. This concludes the proof of Theorem~\ref{mainthm: general cobordism}.

\subsection{Existence of spanning manifolds}\label{sec: intro spanning manifolds} 
As mentioned in \cref{sec: intro cobordisms}, our existence statements for cobordisms rely on existence statements for spanning manifolds of properly embedded surfaces in 4-manifolds. To this end, \cref{sec: Seifert manifolds} is a detailed study of the existence of unoriented spanning manifolds of codimension 2 embeddings. Although we are primarily interested in the case of ambient dimension 4, we work with arbitrary ambient dimension. 

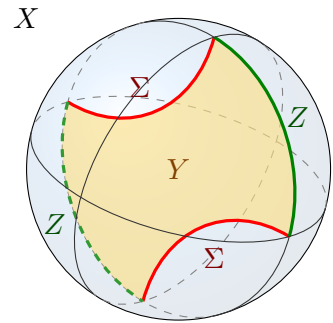
\begin{wrapfigure}{r}{0.35\textwidth}
    \vspace{-1em}
    \begin{center}
    \def\myr{2}
\tdplotsetmaincoords{65}{125}

\begin{tikzpicture}[scale=1]

    \draw[ball color=cyan!50!blue, opacity=0.1, draw opacity=1] (0,0) circle[radius={\myr}];

    \node at (-\myr,\myr) {$X$};

    \begin{scope}[tdplot_main_coords, rotate=-15]
    \begin{scope}[canvas is xz plane at y=0]
            \draw[dashed,black!50] (\tdplotmainphi:\myr) 
                arc[start angle={\tdplotmainphi}, 
                    end angle={\tdplotmainphi+180}, 
                    radius={\myr}];
        \end{scope}
        
        \begin{scope}[canvas is xy plane at z=0] 
            \draw[dashed,black!50] (\tdplotmainphi:\myr) 
                arc[start angle={\tdplotmainphi}, 
                    end angle={\tdplotmainphi+180}, 
                    radius={\myr}];
        \end{scope}

        \begin{scope}[canvas is yz plane at x=0]
            \draw[dashed,black!50] (\tdplotmainphi:\myr) 
                arc[start angle={\tdplotmainphi}, 
                    end angle={\tdplotmainphi+180}, 
                    radius={\myr}];
        \end{scope}

        \begin{scope}[canvas is yz plane at x=0]
            \draw[black!80] (\tdplotmainphi:\myr) 
                arc[start angle={\tdplotmainphi}, 
                end angle={\tdplotmainphi-180}, 
                radius={\myr}];
        \end{scope}

        \begin{scope}[canvas is yz plane at x=0]

            \fill[fill=orange!50!yellow!50, fill opacity=0.6]
        (90:\myr) arc[start angle={0}, end angle={-90}, radius={\myr}] arc[start angle={180}, end angle={270}, radius={\myr}] arc[start angle={180}, end angle={90}, radius={\myr}] arc[start angle={0}, end angle={90}, radius={\myr}];

        \draw[very thick,red] (90:\myr) 
                arc[start angle={0}, 
                end angle={-90}, 
                radius={\myr}];
        \draw[very thick,red] (0:\myr) 
                arc[start angle={90}, 
                end angle={180}, 
                radius={\myr}];
        \draw[very thick,green!50!black!80,dashed] (180:\myr) 
                arc[start angle={180}, end angle={270}, radius={\myr}];
        \draw[very thick,green!50!black] (0:\myr) 
                arc[start angle={0}, end angle={90}, radius={\myr}];

        \node[orange!50!black] at (0,0) {$Y$};
        \node[green!30!black] at (45:\myr*1.15) {$Z$};
        \node[green!30!black] at (225:\myr*1.2) {$Z$};

        \node[red!50!black,right] at (-60:\myr*0.6) {$\Sigma$};
        \node[red!50!black,right] at (155:\myr*0.7) {$\Sigma$};
        
        \end{scope}

        \begin{scope}[canvas is xz plane at y=0]
            \draw[black!80] (\tdplotmainphi:\myr) 
                arc[start angle={\tdplotmainphi}, 
                    end angle={\tdplotmainphi-180}, 
                    radius={\myr}];
        \end{scope}
        
        \begin{scope}[canvas is xy plane at z=0]
            \draw[black!80] (\tdplotmainphi:\myr) 
                arc[start angle={\tdplotmainphi}, 
                    end angle={\tdplotmainphi-180}, 
                    radius={\myr}];
        \end{scope}
    
    \end{scope}

\end{tikzpicture}
    \caption{A proper embedding with $X = D^3$ and $\S \cong I \sqcup I$ (red), and spanning surface $Y$ (yellow) extending $Z \cong I \sqcup I$ (green).}
    \label{fig: exspan}
    \end{center}
    \vspace{-1em}
\end{wrapfigure}

Let $X$ be an $(n+2)$-manifold and let $\S \subset X$ be a properly embedded compact $n$-manifold. A \textit{spanning manifold} for $\S$ is a compact $(n+1)$-manifold $Y$ with corners, embedded in $X$, such that $\del Y = \S \cup Z$, where $Z \subset \del X$ is an embedded $n$-manifold. We say that $Z$ \textit{extends to} $Y$, or equivalently that $Y$ \textit{extends} $Z$. See \cref{fig: exspan}. These definitions are adapted from \cite{BaykurSunukjian}; they may differ from those used elsewhere.

If $\S$ is oriented, we call $Y$ an \textit{oriented spanning manifold} if $Y$ and $Z$ are oriented so that $\del Y = \S \cup Z$ as oriented manifolds. Note this implies that $\del Z = -\del \S$. If $\S$ and $X$ are closed, this recovers the usual definition of a spanning manifold, and it is well-known in this case that $\S$ admits a spanning manifold if and only if $[\S] = 0 \in H_n(X;\Z)$ \cite[Section  VIII]{Kirby}, \cite[Section  XXI]{RanickiKnots}. If $\del X \neq \varnothing$, the same argument shows that $\S$ admits an oriented spanning manifold if and only if $[\S]=0 \in H_n(X, \del X;\Z)$ and the normal bundle $N_X\S$ is trivial (see \cref{prop: oriented spanning manifolds}). 

Fewer results about unoriented spanning manifolds are recorded in the literature. Gordon--Litherland discussed the existence of unoriented spanning surfaces for links in $S^3$ \cite{GordonLitherland}. In the proof of Theorem 6 of \cite{BaykurSunukjian}, Baykur--Sunukjian showed that certain non-orientable surfaces properly embedded in an orientable 4-manifold admit spanning 3-manifolds. Their approach requires that the boundary of the 4-manifold be non-empty and that the surface meet the boundary in a fibred link. 
We expand on their methods to give more general conditions for the existence of a spanning manifold.

\begin{theorem}\label{thm: spanning manifolds}
    Let $X$ be an $(n+2)$-manifold and let $\S \subset X$ be a properly embedded compact $n$-manifold. Let $w \colon \pi_1(\S) \to \Aut(\Z)$ be the orientation character of the normal bundle $N_X\S$, and suppose that $H^2(\S ;\Z^{w} )$ has no 2-torsion. Then $\S$ admits a spanning manifold if and only if $[\S] = 0 \in H_n(X, \del X;\cyc{2} )$ and $N_X\S$ admits a nowhere-vanishing section.
\end{theorem}

The proof of \cref{thm: spanning manifolds} follows the outline of the usual proof in the oriented setting, where a spanning manifold is first shown to exist in the exterior of $\S$ by transversality, then extended through a tubular neighbourhood of $\S$. To this end, we introduce the notion of a \textit{Seifert section} of the circle normal bundle $SN_X\S$, which is a section whose image is null-homologous in the exterior of $\S$. See \cref{sec: seifert sections} for a more precise definition. We show that $\S$ admits a spanning manifold if and only if $SN_X\S$ admits a Seifert section. The forwards direction of this implication is clear, since the normal direction to $\S$ into a spanning manifold determines a Seifert section, which we term the Seifert section \textit{associated to} that spanning manifold. The reverse direction involves a study of sections of $S^1$-bundles, and reducing the well-known homotopy-theoretic obstructions to the existence of sections of bundles to obstructions in $\cyc{2}$-homology.

These methods also allow us to determine when a spanning manifold $Z \subset \del X$ for $\del \S$ extends to a spanning manifold for $\S$. This is the result which is used in the proof of Theorem~\ref{mainthm: general cobordism}.

\begin{theorem}\label{thm: relative spanning manifolds}
    Let $X$, $\S \subset X$, and $w$ be as in \cref{thm: spanning manifolds}.
    Let $Z \subset \del X$ be a spanning manifold for $\del \S$, and let $s^Z$ be the Seifert section associated to $Z$.
    Assume that $H^2(\S, \del \S;\Z^{w})$ has no 2-torsion. 
    Then $Z$ extends to a spanning manifold for $\S$ if and only if there is a section of $SN_X\S$ extending $s^Z$ and $[Z \cup \S] = 0 \in H_n(X;\cyc{2}).$
\end{theorem}

We will apply Theorems~\ref{thm: spanning manifolds} and \ref{thm: relative spanning manifolds} respectively in the cases $n=1$ and $n=2$, corresponding to finding spanning surfaces for the boundaries of surfaces, and spanning 3-manifolds for the surfaces themselves. In these low dimensions, the 2-torsion condition simplifies as follows.
If $n=1$, then $H^2(\S;\Z^w) = 0$, so the 2-torsion condition is automatically satisfied. If $n=2$, then $H^2(\S, \del \S;\Z^w) \cong H_0(\S;\Z^{w'})$, where $w' \colon \pi_1(\S) \to \Aut(\Z)$ is a character which is trivial when the normal bundle $N_X\S$ has orientable total space. Therefore the 2-torsion condition is automatically satisfied whenever the 4-manifold $X$ is orientable. The condition that $s^Z$ extends to a section of $SN_X\S$ is equivalent to saying that $e(S, s^Z) = 0$ for all components $S \subset \S$. This implies the result mentioned in \cref{sec: intro cobordisms} about when the collection $\wh{Z}$ of annuli extends to a spanning 3-manifold of the punctured surface $\wh{\S}$.

\subsection{Failure of classical techniques}\label{sec: intro oriented}
In this subsection, we provide motivation for the introduction of new techniques for constructing cobordisms. We outline a proof of \cref{thm: oriented cobordism} in the case $n=2$, based on the argument in \cite{ThomQuelqueProprietes}, to highlight the need for a different approach in the unoriented setting. See \cref{sec: cs homotopy} and \cref{sec: oriented cobordisms} for more detail, as well as a slightly different perspective which affords more control over the boundary by using ideas similar to those described by Baykur--Sunukjian in Section 2.2 of \cite{BaykurSunukjian}.

Let $X$ be an oriented connected 4-manifold, and let $\S_0, \S_1 \subset X$ be oriented properly embedded compact surfaces.
Since $\C\P^\infty \simeq \BSO(2)$, the Thom--Pontryagin construction and cellular approximation give maps $f_0,f_1 \colon X \to \C\P^2$ such that $f_i\inv(\C\P^1) = \S_i$ for $i=0,1$. Since $\S_0$ and $\S_1$ are compact, $f_0$ and $f_1$ can be chosen to have compact support, in the sense that they agree and are constant outside of a compact set. Since $\C\P^
\infty$ is a $K(\Z,2)$-space, we show in \cref{sec: cs homotopy} that there are bijections
\begin{equation*}
 H_2(X, \del X;\Z) \cong [X,K(\Z,2)]_c \cong [X,\C\P^2]_c,
\end{equation*}
where $[A,B]_c$ is the set of maps $A \to B$ with compact support up to homotopy with compact support. This follows from the usual bijections between cohomology groups and sets of homotopy classes of maps to Eilenberg--MacLane spaces and a general form of Poincar\'e duality. Then if $[\S_0] = [\S_1] \in H_2(X, \del X;\Z)$, there is a homotopy $H \colon X \times I \to \C\P^2$ from $f_0$ to $f_1$ which is constant outside of a compact neighbourhood of $(\S_0 \cup \S_1) \times I$. The preimage $H\inv(\C\P^1) \subset X \times I$ is an oriented cobordism from $\S_0$ to $\S_1$. Hence compact oriented surfaces which are $\Z$-homologous rel.\ boundary are cobordant, and the converse is clear.

This strategy does not work easily in the unoriented case, since it would require the existence of a $K(\cyc{2},2)$-space and a space homotopy equivalent to $\BO(2)$ with a common 5-skeleton. This is not possible, since $\pi_1(\BO(2)) = \cyc{2} \neq \{1\}$. An alternative homotopy-theoretic approach may be viable, but we do not pursue it here. 

Throughout this paper, we provide proofs of both the unoriented statements that we are primarily interested in, and the analogous oriented statements. This is partially because we believe the comparison to be instructive, and partially because our methods give slightly more general results than those currently in the literature, even in the oriented case.

\subsection*{Topological vs smooth categories}
All results in this paper hold in both the topological and smooth categories. In each case, we only prove the smooth statement, then deduce each topological result from the corresponding smooth one. One direction will follow since all of our obstructions are topological. To deduce the reverse direction of statements with ambient dimension $n=4$, recall that by Section 8.2 of \cite{FreedmanQuinn}, we can choose a smooth structure on the 4-manifold $X$ in the complement of a point which extends the smooth structure on the normal bundles of the embedded surfaces $\S_0$ and $\S_1$. Since we never assume that $X$ is compact, we may instead assume that $\S_0$ and $\S_1$ are smoothly properly embedded in $X \setminus \{\pt\}$, and apply our smooth results there. This proves the topological version of all Theorems \ref{mainthm: concordance rel boundary}--\ref{maincor: cobordism rel boundary}. For statements with ambient dimension $n \neq 4$, we rely on topological transversality; see \cite{QuinnTransversality} and \cite[Essay III\textsection1]{KirbySiebenmann}, or Section 10 of \cite{FNOPTopological} for a survey.
Henceforth we work only in the smooth category except where otherwise specified, and assume all manifolds and maps between them are smooth. 

\subsection*{Immersions and embeddings}

Let $f \colon S \immerse X$ be an immersion of manifolds. We say $f$ is \textit{proper} if $\restr{f}{\del S}$ is an embedding, $f\inv(\del X) = \del S$, and $f$ is transverse to $\del X$. If $S$ is a compact $n$-manifold and $X$ is a $2n$-manifold, $n \geq 1$, we assume that $f$ is generic, meaning that the only self-intersections of $f$ are isolated transverse double points.

We say $\S \subset X$ is \textit{properly embedded} if the inclusion map is a proper embedding. 
A \textit{tubular neighbourhood} of $\S$ is a pair $(\nu \S, \varphi)$ where $\nu\S \varsubsetneq X$ is a closed neighbourhood of $\S$ and $\varphi \colon DN_X\S \to X$ is an embedding of the normal disc bundle such that $\varphi(DN_X\S) = \nu\S$ and if $s_0\colon \S \to DN_X\S$ is the zero-section, then $\varphi s_0 = \id_\S$. 
We write $S\nu\S \coloneq \varphi(SN_X\S) \subseteq \del (\nu \S)$ for the image of the normal sphere bundle.
If $A \subseteq \del X$ is a subset, a \textit{collar} of $A$ is a pair $(C,\varphi)$ where $C \varsubsetneq X$  and
$\varphi \colon A \times I \to X$ is an embedding such that $\varphi(A \times I) = C$, $\varphi(x,0) = x$ for $x \in A$, and $\varphi\inv(\del X) = A \times \{0\}$. 
If the identification $\varphi$ is not relevant, we may simply refer to $\nu\S$ (resp.\ $C$) as a tubular neighbourhood of $\S$ (resp.\ a collar of $A$).

\subsection*{Orientation conventions}
We make no assumptions about orientations or orientability unless otherwise specified.
In particular, we use \textit{unoriented manifold} to refer either to a non-orientable manifold or to an orientable manifold for which a choice of orientation has not been made.
We say two properly immersed $R$-oriented $n$-manifolds $A,B \subseteq X$ are \textit{$R$-homologous} if $[A] = [B] \in H_n(X, \del X;R)$. When the choice of $R$ is clear from context, we may simply say that $A$ and $B$ are \textit{homologous}. 

If $M$ is an oriented manifold, $-M$ refers to $M$ but with the reversed orientation.
We orient $N_M \del M$ in the direction of the outward normal. We orient $\del M$ such that for all $p \in \del M$, the two orientations on $T_p M$ coming from $M$ and from the decomposition $\restr{N_M \del M}{p} \oplus T_p \del M$ agree. We orient $M \times I$ such that the orientations on $T_{(p,t)} (M \times I)$ and $T_{t}I \oplus T_p M$ agree. This is to fit with the usual convention that if $M$ is closed and oriented, $\del (M \times I) = -M \times \{0\} \sqcup M \times \{1\}$.

\subsection*{Organisation}
Sections~\ref{sec: cs homotopy}--\ref{sec: euler number theory} contain most of the technical results required to prove our main theorems which work in arbitrary ambient dimension. 
In \cref{sec: cs homotopy}, we prove results about representing relative homology classes by properly embedded submanifolds with prescribed boundary.
In \cref{sec: Seifert manifolds}, we use these to study spanning manifolds of codimension 2 embeddings and prove Theorems~\ref{thm: spanning manifolds} and \ref{thm: relative spanning manifolds}.
\cref{sec: euler number theory} is primarily a review of the relative normal Euler number of immersed submanifolds.

We then specialise to the case of surfaces in 4-manifolds.
In \cref{sec: new surfaces sec}, we study spanning manifolds of properly embedded surfaces given by puncturing immersions around double points.
In \cref{sec: cobordisms}, we prove Theorems~\ref{mainthm: general cobordism} and \ref{mainthm: cobordism}, and Corollary~\ref{maincor: cobordism rel boundary}.

\cref{sec: Ambient surgery} is primarily a review of facts about ambient surgery featured in \cite{Sunukjian} and \cite{KlugMiller1}.
The theory of $\Pin^\pm$-surgery is developed in \cref{sec: Spin}. In \cref{sec: final concordance}, we prove Theorem~\ref{mainthm: general concordance}.

\subsection*{Acknowledgements}
The author is grateful to their advisors Mark Powell and Brendan Owens for their patience and many useful insights, as well as to Anthony Conway for a useful conversation concerning the material in \cref{sec: existence of seifert sections}. The author also thanks Nathan Sunukjian and Anthony Conway for helpful conversations about an earlier draft. This work was supported by the Engineering and Physical Sciences Research Council through the EPSRC Centre for Doctoral Training in Algebra, Geometry and Quantum Fields (AGQ) (EP/Y035232/1).
This work was also supported by the Additional Funding Programme for Mathematical Sciences, delivered by EPSRC (EP/V521917/1) and the Heilbronn Institute for Mathematical Research. 
 \section{Representing homology classes by properly embedded submanifolds}\label{sec: cs homotopy}

For spaces $A$ and $B$, we write $[A,B]$ for the set of continuous maps $A \to B$ up to homotopy.
It is a standard result in algebraic topology that for any abelian group $G$ and any $n \geq 1$, there is a fundamental class $\alpha \in H^n(K(G,n);G)$ such that for any CW-complex $X$, the map $[X,K(G,n)] \to H^n(X;G)$ given by $[f] \mapsto f^*\alpha$ is a bijection \cite[Chapter V.4]{Whitehead}. One consequence of this and Poincar\'e duality, originally due to Thom, is that all $\cyc{2}$-homology classes of codimension 1 in closed manifolds are represented by embedded closed manifolds. Similarly, all $\Z$-homology classes of codimension 1 or 2 in closed oriented manifolds are represented by embedded closed oriented manifolds \cite[\textsection II.11]{ThomQuelqueProprietes}. 

In this section, we prove generalisations of these realisation results arbitrary manifolds, possibly non-compact and with non-empty boundary. In particular, we show in \cref{sec: representing} that the submanifolds representing given relative homology classes of small enough codimension can be taken to have any prescribed boundary satisfying the necessary homological conditions. We will use these results to prove the existence of spanning manifolds (see \cref{sec: Seifert manifolds}) and oriented cobordisms (see \cref{sec: oriented cobordisms}) with prescribed boundaries.

\subsection{Homotopy with compact support}\label{sec: homotopy}

For a space $A$ and a pointed space $(B,*)$, let $\Map_c (A,B)$ be the maps $A \to B$ with compact support, i.e.
    \begin{equation*}
        \Map_c (A,B) \coloneq \big\{ f \colon A \to B \mid \exists K \subseteq A \text{ compact s.t. } f(A \setminus K) \subseteq \{*\}\big\}.
    \end{equation*}
For $f,g \in \Map_c(A,B)$, we write $f \sim_c g$ if there is a homotopy from $f$ to $g$ with compact support; that is, if there is a compact set $K \subseteq A$ such that $f \simeq g$ rel.\ $A \setminus K$ and $f(A \setminus K) = g(A \setminus K) \subseteq \{*\}$. 
Then $\sim_c$ is an equivalence relation on $ \Map_c(A,B)$, and we write
\begin{equation*}
    [A,B]_c \coloneq \frac{\Map_c(A,B)}{\sim_c}
\end{equation*}
for the set of equivalence classes.

\begin{remark}
    Although both $ \Map_c(A,B)$ and $[A,B]_c$ do in general depend on the choice of basepoint $* \in B$, this dependence is trivial when $B$ is a connected manifold. This is the only case that we will consider, so we suppress basepoint-dependence in our notation. If $A$ is compact, then $ \Map_c(A,B) = \Map(A,B)$ and $[A,B]_c = [A,B]$ as would be expected.
\end{remark}

The set $[A,B]_c$ also admits a description as a colimit of sets.

\begin{lemma}\label{lmm: colimit definition of homotopy classes}
    For any space $A$ and pointed space $(B,*)$,
    \begin{equation*}
        [A,B]_c = \colim_{K \subseteq A} [(A, A \setminus K), (B,*)],
    \end{equation*}
    where the colimit is taken over all compact subsets $K \subseteq A$ ordered by inclusion.
\end{lemma}
\begin{proof}
    By the construction of a colimit of sets, an element of $\colim [(A,A\setminus K), (B,*)]$ is represented by a homotopy class of maps $[f] \in [(A, A \setminus K), (B,*)]$ for some $K \subseteq A$ compact. Then $[f]$ represents the same class in the colimit as $[g] \in [(A, A \setminus K'), (B,*)]$ if and only if there is a compact subset $K'' \supseteq K \cup K'$ such that 
    \begin{equation*}
        [f] = [g] \in [(A, A \setminus K''), (B,*)].
    \end{equation*}
    But this is exactly saying that there is a homotopy from $f$ to $g$ with compact support.
\end{proof}

We now give our bijection between homology classes rel.\ boundary and compactly supported homotopy classes of maps to Eilenberg--MacLane spaces. This is mostly a formal consequence of \cref{lmm: colimit definition of homotopy classes} and the Poincar\'e duality between homology rel.\ boundary and cohomology with compact support.

\begin{lemma}\label{prop: compact homotopy classes}
    Let $M$ be an $n$-manifold.
    \begin{enumerate}[label=(\roman*)]
        \item Let $k \geq n$ and fix a basepoint $* \in \R\P^{k+1} \setminus \R\P^k$.  There is a bijection
        \begin{equation*}
            [M,\R\P^{k+1}]_c \xrightarrow{\cong}  H_{n-1}(M,\del M;\cyc{2})
        \end{equation*}
        given by $[f] \mapsto [f\inv (\R\P^k) ]$ for any representative $f$ such that $f \pitchfork \R\P^k$.
        \end{enumerate}
    Now suppose that $M$ is oriented.
    \begin{enumerate}[label=(\roman*)]
     \setcounter{enumi}{1}
        \item Fix a basepoint $* \in S^1 \setminus \{1\}$. There is a bijection 
        \begin{equation*}
            [M,S^1]_c \xrightarrow{\cong}  H_{n-1}(M,\del M;\Z)
        \end{equation*}
        given by $[f] \mapsto [f\inv (\{1\}) ]$ for any representative $f$ such that $f \pitchfork \{1\}$, i.e.\ such that $1$ is a regular value of $f$.
        \item Let $k \geq \floor{(n-1)/2}$ and fix a basepoint $* \in \C\P^{k+1} \setminus \C\P^k$. There is a bijection 
        \begin{equation*}
            [M,\C\P^{k+1}]_c \xrightarrow{\cong}  H_{n-2}(M,\del M;\Z)
        \end{equation*}
        given by $[f] \mapsto [f\inv (\C\P^k) ]$ for any representative $f$ such that $f \pitchfork \C\P^k$.
    \end{enumerate}
\end{lemma}

\begin{proof}
    We prove (i). The arguments for (ii) and (iii) are almost identical, using $S^1$ as a $K(\Z,1)$-space and $\C\P^\infty$ as a $K(\Z,2)$-space. 

    By \cref{lmm: colimit definition of homotopy classes} and cellular approximation, we get bijections
    \begin{equation*}
        [M,\R\P^{k+1}]_c = \colim_{K \subseteq M} [(M, M \setminus K), (\R\P^{k+1}, *)] \xrightarrow{\cong} \colim_{K \subseteq M} [(M, M \setminus K), (\R\P^\infty, *)],
    \end{equation*}
    where both colimits are colimits of sets taken over compact subsets $K \subseteq M$.
    Since $\R\P^\infty$ is a $K(\cyc{2},1)$-space, there is a bijection
\begin{equation*}
    \colim_{K \subseteq M} [(M, M \setminus K), (\R\P^\infty, *)] \xrightarrow{\cong} \colim_{K \subseteq M} H^1(M, M \setminus K;\cyc{2}) = H^1_c(M;\cyc{2}),
\end{equation*}
where $H^*_c(M;\cyc{2})$ is cohomology with compact support. 
By the non-compact version of Poincar\'e duality \cite[\textsection6]{Spanier}, we get a bijection
\begin{equation*} \PD \colon H^1_c(M;\cyc{2}) \xrightarrow{\cong}  H_{n-1}(M,\del M;\cyc{2}). \end{equation*}
Composing all of the above, we get the required bijection.

To see the explicit form, note that the fundamental class $\alpha \in H^1(\R\P^\infty;\cyc{2}) \cong \cyc{2}$ is the generator, and that the restriction $\alpha' \in H^1(\R\P^{k+1};\cyc{2}) \cong \cyc{2}$ is Poincar\'e dual to the generator $[\R\P^k] \in H_n(\R\P^{k+1};\cyc{2}) \cong \cyc{2}$. Then since transverse preimages are Poincar\'e dual to pullbacks (see e.g.\ \cite[Theorem 11.16]{Bredon}),
\begin{equation*}
    \PD(f^*\alpha') = [f\inv(\R\P^k)] \in H_{n-1}(M, \del M;\cyc{2}),
\end{equation*}
whenever $f \pitchfork \R\P^k$. Finally, recall that such $f$ are generic \cite[IV\textsection4]{GolubitskyGuillemin}.
\end{proof}

\begin{remark}
    The assumption that $M$ is oriented in (ii) and (iii) is not in fact required. In general, the proof gives bijections
    \begin{equation*}
        [M,S^1]_c \xrightarrow{\cong} H_{n-1}(M, \del M;\Z^{w_1}) \quad \text{and} \quad [M,\C\P^{k+1}]_c \xrightarrow{\cong} H_{n-2}(M, \del M;\Z^{w_1}),
    \end{equation*}
    where $w_1 \colon \pi_1(M) \to \Aut(\Z)$ is the orientation character of $M$.
\end{remark}

\begin{remark}\label{rem: topological category}
    \cref{prop: compact homotopy classes} also holds in the topological category, since all transversality statements also hold for topological transversality. See the proof of Theorem 10.11 of \cite{FNOPTopological} for the case of $M$ oriented and compact; the same modifications as given in the proof above for $M$ unoriented and non-compact go through in that setting. In turn, this implies that all later results in this section as well as all direct applications of them also go through in the topological category without any major modifications. We therefore omit any future remarks on proofs in the topological category.
\end{remark}

\subsection{Representing homology classes with prescribed boundaries}\label{sec: representing}

We can use \cref{prop: compact homotopy classes} to represent relative homology classes by compact embedded manifolds with prescribed boundaries. Let $M$ be an $n$-manifold and let $A \subseteq \del M$ be an embedded $(n-1)$-manifold which is closed as a subspace of $M$. Define the maps on homology (with any choices of coefficients and degree)
\begin{equation*}
        \del_A \colon  H_{*}(M, \del M) \xrightarrow{\del}  H_{*-1}( \del M, \del M \setminus A) \xrightarrow{\cong} H_{*-1}(A, \del A),
\end{equation*}
where the first map is the boundary map in the long exact sequence of the triple $(M, \del M, \del M \setminus A)$, and the second is the inverse of an excision isomorphism. This is given explicitly on a relative cycle $\sigma$ in $(M, \del M)$ by $\del_A[\sigma] = [\sigma \cap A]$.

We can use \cref{prop: compact homotopy classes}(i) to represent any $\cyc{2}$-homology class of codimension 1 by an unoriented submanifold with appropriate prescribed boundary. The methods we use here to control the boundary of the representing submanifolds are adapted and expanded from \cite{BaykurSunukjian}.

\begin{proposition}\label{prop: realising homology classes}
    Let $B \subset A$ be a properly embedded compact $(n-2)$-manifold. 
    For any $\beta \in H_{n-1}(M, \del M;\cyc{2})$ with $\del_A \beta = [B] \in H_{n-2}(A, \del A;\cyc{2})$,
    there exists a properly embedded compact $(n-1)$-manifold $Y \subset M$ such that $Y \cap A = B$ and $[Y] = \beta \in H_{n-1}(M, \del M;\cyc{2})$.
\end{proposition}

\begin{proof}
Fix a basepoint  $* \in \R\P^{n+1} \setminus \R\P^{n}$. Let $f \colon M \to \R\P^{n+1}$ be a representative of the class corresponding to $\beta$ under the bijection
\begin{equation*}
    [M,\R\P^{n+1}]_c \xrightarrow{\cong} H_{n-1}(M,\del M;\cyc{2})
\end{equation*}in \cref{prop: compact homotopy classes}(i). We may assume that $f \pitchfork \R\P^n$, so that $Y \coloneq f\inv(\R\P^{n}) \subset M$
is a properly embedded compact $(n-1)$-manifold with $[Y] = \beta \in H_{n-1}(M, \del M;\cyc{2})$. 

We now show that if there exists some map $g \colon A \to \R\P^{n+1}$ with compact support such that $g\inv(\R\P^n) = B$, then we can choose $f$ so that $Y \cap A = \restr{f}{A}\inv(\R\P^n) = B$.
Indeed, suppose that such a $g$ exists. Since both $B$ and $Y \cap A$ represent the class $\del_A \beta \in H_{n-2}(A, \del A;\cyc{2})$, they are homologous in $A$ rel.\ $\del A$. Then by applying \cref{prop: compact homotopy classes}(i) to $A$ with $k=n$, we see that $\restr{f}{A} \simeq_c g$. That is, there is a homotopy $H \colon A \times I \to \R\P^{n+1}$ from $\restr{f}{A}$ to $g$, and a compact set $K \subseteq A$ such that $H\big((A \setminus K) \times I\big) \subseteq \{*\}$. Let $N \subseteq M$ be a compact neighbourhood of $K$. Then $H$ extends to a homotopy $H' \colon (A \cup \bar{M \setminus N}) \times I\to \R\P^{n+1} $ by
\begin{equation*}
    H'(x,t) = \begin{cases} H(x,t) & x \in A, \\ * & x \in \bar{M \setminus N}. \end{cases} 
\end{equation*}
So $H'$ is a homotopy from $\restr{f}{A \cup \bar{M \setminus N} }$ to $g \cup \restr{f}{\bar{M \setminus N}}$.
Since $A\cup\bar{M \setminus N}$ is a closed subspace of $M$, the inclusion $A \cup \bar{M \setminus N} \hookrightarrow M$ is a cofibration \cite[Chapter VII.1]{Bredon}, and hence $H'$ can be extended over $M$ to a homotopy from $f$ to some $f'\colon M \to \R\P^{n+1}$ with 
\begin{equation*}
    \restr{f'}{A \cup \bar{M \setminus N}} = g \cup \restr{f}{\bar{M \setminus N}}.
\end{equation*}
This homotopy is supported in $N$, so has compact support. Hence $f \simeq_c f'$. So by replacing $f$ by $f'$, we may assume that $Y \cap A = \restr{f'}{A}\inv(\R\P^n) = B$, and the result follows.

It remains to construct such a map $g\colon A \to \R\P^{n+1}$. Let $\nu B \subset A$ be a tubular neighbourhood for $B$ in $A$, and hence a $D^1$-bundle over $B$. The universal $D^1$-bundle is the disc bundle of the tautological line bundle $\gamma \colon E \to \R\P^\infty$, so after cellular approximation we obtain a bundle morphism
\[\begin{tikzcd}
	{\nu B} & {DE^n} \\
	B & {\R\P^n},
	\arrow["{g_0}", from=1-1, to=1-2]
	\arrow[from=1-1, to=2-1]
	\arrow["{D\gamma^n}", from=1-2, to=2-2]
	\arrow[from=2-1, to=2-2]
\end{tikzcd}\]
where $\gamma^n \colon E^n \to \R\P^n$ is the tautological line bundle \cite[\textsection5.1]{MilnorStasheff}. Since the Thom space ${DE^n/SE^n}$ is diffeomorphic to $\R\P^{n+1}$, we can extend $g_0$ to a map
\begin{equation*}
    g \colon A \xrightarrow{\text{collapse}} \nu B / S \nu B \xrightarrow{\bar{g_0}} DE^n/ SE^n \xrightarrow{\cong} \R\P^{n+1}.
\end{equation*}
This map satisfies $g\inv(\R\P^{n}) = B$ by construction, so $g$ is the required map. In the case that $A$ is non-compact, we must ensure that the chosen basepoint is the point $g(A \setminus \nu B)$. This can be arranged by choosing an appropriate diffeomorphism $DE^n/ SE^n \xrightarrow{\cong} \R\P^{n+1}$.
\end{proof}

Using \cref{prop: compact homotopy classes}(ii), we can show an analogous result when all manifolds are taken to be oriented and homology is taken with $\Z$-coefficients. 

\begin{proposition}\label{prop: realising oriented homology classes}
    Suppose that $M$ and $A$ are oriented. Let $B \subset A$ be an oriented properly embedded compact $(n-2)$-manifold. 
    For any $\beta \in H_{n-1}(M, \del M;\Z)$ with $\del_A \beta = [B] \in H_{n-2}(A, \del A;\Z)$,
    there exists an oriented properly embedded compact $(n-1)$-manifold $Y \subset M$ such that $Y \cap A = B$ and $[Y] = \beta \in H_{n-1}(M, \del M;\Z)$.
\end{proposition}

\begin{proof}
    The proof is almost identical to the proof of \cref{prop: realising homology classes}. Let $f\colon M \to S^1$ be a representative of the class corresponding to $\beta$ under the bijection
    \begin{equation*}
        [M,S^1]_c \xrightarrow{\cong} H_{n-1}(M, \del M;\Z)
    \end{equation*}
    in \cref{prop: compact homotopy classes}(ii), and assume $1$ is a regular value of $f$. Then $Y \coloneq f\inv(\{1\}) \subset M$ is a oriented properly embedded compact manifold with $[Y] = \beta \in H_{n-1}(M, \del M;\Z)$.    
    As in the unoriented case, to arrange that $Y \cap A = B$, it suffices to show that there is a map $g \colon A \to S^1$ with compact support such that $g\inv(\{1\}) = B$.

    To construct the map $g\colon A \to S^1$, let $\nu B \subset A$ be a tubular neighbourhood of $B$. Then since $A$ and $B$ are oriented, $\nu B$ is the trivial $D^1$-bundle over $B$. Hence there is projection map $\nu B \to D^1$, and the Thom construction gives the required map
    \begin{equation*}
        g \colon A \xrightarrow{\text{collapse}} \nu B / S\nu B \to D^1 / \del D^1 \cong S^1
    \end{equation*}
    with $g\inv(\{1\}) = B$.
\end{proof}

By \cref{prop: compact homotopy classes}(iii), the same construction works for $\Z$-homology classes of codimension 2.

\begin{proposition}\label{prop: oriented codimension 2}
    Suppose that $M$ and $A$ are oriented. Let $B \subset A$ an oriented properly embedded compact $(n-3)$-manifold.
    Then for any $\beta \in H_{n-2}(M, \del M;\Z)$ with $\del_A \beta = [B] \in H_{n-3}(A, \del A;\Z)$,
    there exists an oriented properly embedded compact $(n-2)$-manifold $Y \subset M$ such that $Y \cap A = B$ and $[Y] = \beta \in H_{n-2}(M, \del M;\Z)$.
\end{proposition}
\begin{proof}
    Let $k = \floor{(n-1)/2}$, and let $f\colon M \to \C\P^{k+1}$ be a representative of the class corresponding to $\beta$ under the bijection
    \begin{equation*}
        [M,\C\P^{k+1}]_c \xrightarrow{\cong} H_{n-2}(M, \del M;\Z)
    \end{equation*}
    in \cref{prop: compact homotopy classes}(iii), and assume $f \pitchfork \C\P^k$. Then $Y \coloneq f\inv(\C\P^k) \subset M$ is a oriented properly embedded compact manifold with $[Y] = \beta \in H_{n-2}(M, \del M;\Z)$.    
    As before, to arrange that $Y \cap A = B$, it suffices to show that there is a map $g \colon A \to \C\P^{k+1}$ with compact support such that $g\inv(\C\P^k) = B$. 

    The construction of such a map is nearly identical to the unoriented codimension 1 case. Let $\nu B \subset A$ be a tubular neighbourhood, which is an oriented $D^2$-bundle over $B$ since $A$ and $B$ are oriented. 
    The universal oriented $D^2$-bundle is the disc bundle of the tautological oriented rank 2 bundle $\gamma \colon E \to \C\P^\infty$, so after cellular approximation, $\nu B$ is classified by a bundle morphism
\[\begin{tikzcd}
	{\nu B} & {DE^k} \\
	B & {\C\P^{k},}
	\arrow["{g_0}",from=1-1, to=1-2]
	\arrow[from=1-1, to=2-1]
	\arrow["{D\gamma^k}", from=1-2, to=2-2]
	\arrow[from=2-1, to=2-2]
\end{tikzcd}\]
where $\gamma^k \colon E^k \to \C\P^k$ is the tautological oriented rank 2 bundle. The Thom space $DE^k/SE^k$ is diffeomorphic to $\C\P^{k+1}$, so the Thom construction gives the required map
\begin{equation*}
    g \colon A \xrightarrow{\text{collapse}} \nu B / S \nu B \xrightarrow{\bar{g_0}} DE^k/ SE^k \xrightarrow{\cong} \C\P^{k+1},
\end{equation*}
with $g\inv(\C\P^k) = B$.
\end{proof}

There is no complete analogue to \cref{prop: oriented codimension 2} allowing us to represent codimension 2 classes in $\cyc{2}$-homology by properly embedded compact submanifolds. However we can make some progress in the case of $n=4$; see e.g.\ Chapter 2 of \cite{Kirby} or Remark 1.2.4 of \cite{GompfStipsicz}.

\begin{proposition}\label{lmm: embedding surfaces in 4manifolds}
    Let $X$ be a 4-manifold, and fix $\alpha \in H_2(X, \del X; \cyc{2})$. Then there exists a properly embedded compact surface $\S \subset X$ such that $[\S] = \alpha$. 
\end{proposition}

 The same amount of control over the boundary as in \cref{prop: oriented codimension 2} can also be obtained in this case, although this uses different techniques and will not be necessary for our applications.

\section{Spanning manifolds of codimension 2 embeddings}\label{sec: Seifert manifolds}
This section contains a discussion of spanning manifolds of codimension 2 proper embeddings. Recall that if $X$ is an $(n+2)$-manifold and $\S \subset X$ is a properly embedded compact $n$-manifold, then a spanning manifold for $\S$ is a compact $(n+1)$-manifold $Y$ with corners, embedded in $X$, such that $\del Y = \S \cup Z$, where $Z \subset \del X$ is an embedded $n$-manifold.

In \cref{sec: seifert sections}, we define the notion of a Seifert section of a codimension 2 properly embedded submanifold $\S \subset X$. This is a section of the normal circle bundle $SN_X\S$, which is a generalisation of a Seifert framing of a knot or link in $S^3$.
In \cref{sec: oriented seifert manifolds}, we give exact conditions for an oriented proper codimension 2 embedding in an oriented manifold to admit a spanning manifold.
In \cref{sec: section determines spanning}, we show that an unoriented properly embedded submanifold admits a spanning manifold if and only if its normal circle bundle admits a Seifert section. 
Sections~\ref{sec: existence of seifert sections} and \ref{sec: actual seifert stuff}, prove Theorems~\ref{thm: spanning manifolds} and \ref{thm: relative spanning manifolds} by giving sufficient conditions for the existence of a Seifert section.

\begin{notation}\label{notation: cobordism}
    For the rest of this section, fix an $(n+2)$-manifold $X$ and  a properly embedded compact $n$-manifold $\S \subset X$, with $n \geq 1$. 
    Fix a tubular neighbourhood $(\nu \S, \varphi)$ for $\S$, and recall that we write
    \begin{equation*}
        S\nu \S \coloneq \varphi(SN_X\S) \subseteq \del (\nu \S)
    \end{equation*}
    for the embedding of the normal circle bundle around $\S$.
    Write $\pi \colon S\nu\S \to \S$ for the $S^1$-bundle structure such that $\restr{\pi\varphi}{SN_X\S} \colon SN_X\S \to \S$ is the usual projection map.
    Let $E \coloneq \bar{X \setminus \nu \S} \subset X$ be the exterior of $\S$. Note that $E$ is a manifold with corners, and that
    \begin{equation*}
        \del E = (\del X \cap \del E) \mathop{\cup}_{\del (S\nu \S)} S\nu \S.
    \end{equation*}
    We also define the maps on homology (with any choice of coefficient ring)
\begin{equation*}
        \del_{S} \colon H_{n+1}(E, \del E) \xrightarrow{\del} H_n(\del E, \del X \cap \del E) \xrightarrow{\cong} H_n(S\nu\S, \del( S\nu\S))
\end{equation*}
and
\begin{equation*}
        \del_X \colon H_{n+1}(E, \del E) \xrightarrow{\cong} H_{n+1}(X, \del X \cup \S) \xrightarrow{\del} H_n(\del X \cup \S, \S) \xrightarrow{\cong} H_n(\del X, \del \S),
\end{equation*}
where all marked isomorphisms are excision isomorphisms or their inverses, and all indicated boundary maps come from the appropriate long exact sequences of triples. If $\sigma$ is a relative $(n+1)$-cycle in $(E, \del E)$, then $\del_S[\sigma] = [\del \sigma \cap S\nu \S]$ and $\del_X[\sigma] = [\del \sigma \cap \del X]$.
\end{notation}

\begin{remark}\label{rmk: on boundary components away from Sigma}
    If $\del_0 X \subseteq \del X$ is a boundary component disjoint from $\S$, then we may replace the ambient manifold $X$ with $X \setminus \del_0 X$. In this way, we can assume that all constructions occur away from boundary components of $X$ that do not meet $\S$. In particular, if $\S$ is closed, then all results in this section still hold when considering $H_n(X;\cyc{2})$ instead of $H_n(X, \del X;\cyc{2})$ and assuming that all constructions occur in the interior of $X$.
\end{remark}

\subsection{Seifert sections}\label{sec: seifert sections}

Suppose that $\S$ has a spanning manifold $Y$.
By taking the tubular neighbourhood $\nu \S$ small enough, we may assume that $Y\cap \nu\S$ is a collar of $\S$ in $Y$. Then $Y \cap S\nu \S$ is a push-off of $\S$, and hence is the image of a section $\S \to S\nu\S$ of $\pi$. In this way, $Y$ determines a section $s \colon \S \to SN_X\S$ via $\varphi$, given by the direction into $Y$. This is analogous to the Seifert framing of a knot in $S^3$, though in the general case the normal bundle $N_X\S$ need not be trivial. We formalise this notion as follows, and refer the reader to \cref{fig: example embedding}.

\begin{figure}[t]
    \begin{center}
    \def\myr{2}
\tdplotsetmaincoords{65}{125}

\begin{tikzpicture}[scale=1]

    \draw[ball color=cyan!50!blue, opacity=0.1, draw opacity=1] (0,0) circle[radius={\myr}];

    \node at (-\myr*0.9,\myr*0.9) {$X$};

    \begin{scope}[tdplot_main_coords, rotate=-15]
    \begin{scope}[canvas is xz plane at y=0]
            \draw[dashed,black!50] (\tdplotmainphi:\myr) 
                arc[start angle={\tdplotmainphi}, 
                    end angle={\tdplotmainphi+180}, 
                    radius={\myr}];
        \end{scope}
        
        \begin{scope}[canvas is xy plane at z=0] 
            \draw[dashed,black!50] (\tdplotmainphi:\myr) 
                arc[start angle={\tdplotmainphi}, 
                    end angle={\tdplotmainphi+180}, 
                    radius={\myr}];
        \end{scope}

        \begin{scope}[canvas is yz plane at x=0]
            \draw[dashed,black!50] (\tdplotmainphi:\myr) 
                arc[start angle={\tdplotmainphi}, 
                    end angle={\tdplotmainphi+180}, 
                    radius={\myr}];
        \end{scope}

        \begin{scope}[canvas is yz plane at x=0]

            \fill[fill=orange!50!yellow!50, fill opacity=0.6]
        (90:\myr) arc[start angle={0}, end angle={-90}, radius={\myr}] arc[start angle={180}, end angle={270}, radius={\myr}] arc[start angle={180}, end angle={90}, radius={\myr}] arc[start angle={0}, end angle={90}, radius={\myr}];

        \draw[very thick,red] (90:\myr) 
                arc[start angle={0}, 
                end angle={-90}, 
                radius={\myr}];
            \draw[very thick,red] (0:\myr) 
                arc[start angle={90}, 
                end angle={180}, 
                radius={\myr}];

        \draw[very thick,red!80,densely dotted] (80:\myr) 
                arc[start angle={-10}, 
                end angle={-80}, 
                radius={\myr*1.42}];

        \draw[very thick,red!80,densely dotted] (10:\myr) 
                arc[start angle={100}, 
                end angle={170}, 
                radius={\myr*1.42}];

        \node[orange!50!black] at (0,0) {$Y$};
        \node[red!50!black,right] at (-60:\myr*0.5) {$\Sigma$};
        \node[red!50!black,right] at (155:\myr*0.65) {$\Sigma$};
        
        \end{scope}

        \begin{scope}[canvas is xz plane at y=0]
            \draw[black!80] (\tdplotmainphi:\myr) 
                arc[start angle={\tdplotmainphi}, 
                    end angle={\tdplotmainphi-180}, 
                    radius={\myr}];
        \end{scope}
        
        \begin{scope}[canvas is xy plane at z=0]
            \draw[black!80] (\tdplotmainphi:\myr) 
                arc[start angle={\tdplotmainphi}, 
                    end angle={\tdplotmainphi-180}, 
                    radius={\myr}];
        \end{scope}

        \begin{scope}[canvas is yz plane at x=0]
            \draw[black!80] (\tdplotmainphi:\myr) 
                arc[start angle={\tdplotmainphi}, 
                end angle={\tdplotmainphi-180}, 
                radius={\myr}];
        \end{scope}
    
    \end{scope}

\end{tikzpicture}
    \caption{A proper codimension 2 embedding with $X = D^3$ and $\S = I \sqcup I$ (red). A spanning surface $Y$ (yellow) is given. The push-off $\S^s_+$ of $\S$ determined by the Seifert section associated to $Y$ is shown (dashed red).}
    \label{fig: example embedding}
    \end{center}
\end{figure}
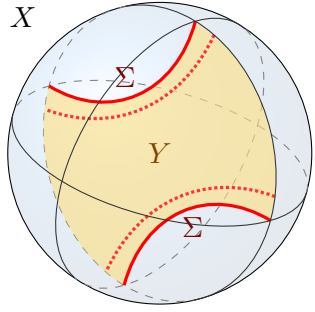

\begin{definition}\label{def: Seifert section}
    Let $s \colon \S \to SN_X\S$ be a section. We write $\S^s_+ \coloneq \varphi s (\S) \subset S\nu \S$ for the push-off of $\S$ in the direction of $s$. 
    We call $s$ a  \textit{Seifert section} if \begin{equation*} [\S^s_+] = (\varphi s)_*[\S] \in \im \del_{S} \subseteq H_n(S\nu\S, \del(S\nu\S);\cyc{2}).\end{equation*}  
    Given a spanning manifold $Y$ for $\S$, suppose that $\nu \S$ is small enough that $Y \cap \nu \S$ is a collar neighbourhood for $\S$ in $Y$. The \textit{Seifert section associated to $Y$} is the unique section $s \colon \S \to SN_X\S$ such that $\S^s_+ = Y\cap S\nu\S$.
\end{definition}

Note that whether a section of $SN_X\S$ is a Seifert section or not is independent of the choice of tubular neighbourhood $(\nu\S, \varphi)$. Moreover, the choice of tubular neighbourhood only affects the Seifert section associated to spanning manifold $Y$ by an isotopy. If $Z \subset \del X$ is a spanning manifold for $\del \S$ with associated Seifert section $s^Z$, and $Z$ extends to a spanning manifold $Y$ for $\S$ with associated Seifert section $s$, then $s^Z = \restr{s}{\del \S}$ after identifying $SN_{\del X}{\del \S} = \restr{SN_X\S}{\del \S}$.

\begin{remark}\label{rem: null-homologous seifert}
    It is interesting and important to remark that $SN_X\S$ can admit a Seifert section only if $\S$ is null-homologous. To see this, consider the diagram
\[\begin{tikzcd}
	{H_{n+1}(X, \del X \cup \S)} & {H_n(\del X \cup \S, \del X)} & {H_n(X, \del X)} \\
	{H_{n+1}(E, \del E)} & {H_n(\S, \del \S)} & {H_n(X, \del X),}
	\arrow["\del", from=1-1, to=1-2]
	\arrow[from=1-2, to=1-3]
	\arrow["\cong", from=2-1, to=1-1]
	\arrow["{\pi_* \del _S}", from=2-1, to=2-2]
	\arrow["\cong"', from=2-2, to=1-2]
	\arrow[from=2-2, to=2-3]
	\arrow["{=}"', from=2-3, to=1-3]
\end{tikzcd}\]
where marked isomorphisms are excision isomorphisms, unlabelled arrows are induced by inclusion, and $\cyc{2}$-coefficients are omitted.
    The top row is exact by the long exact sequence of the triple $(X, \del X \cup \S, \del X)$; the left square commutes by the definition of $\del_S$; the right square commutes since all maps are induced by inclusion. Hence the bottom row is exact, and if $s \colon \S \to SN_X\S$ is a Seifert section, then
\begin{equation*}
    [\S] = \pi_*[\S^s_+] \in \im \big( H_n(E, \del E) \xrightarrow{\pi_* \del_S } H_n(\S, \del \S) \big) = \ker \big(H_n(\S, \del \S) \to H_n( X, \del X) \big),
\end{equation*}
and so $[\S]$ is null-homologous.
\end{remark}

\subsection{Existence of oriented spanning manifolds}\label{sec: oriented seifert manifolds}

We first prove an oriented version of Theorem~\ref{thm: spanning manifolds}; that is, we prove that a codimension 2 proper embedding admits a spanning manifold if and only if it is null-homotopic and has a trivial normal bundle. We feel this is instructive even though it is very similar to standard results in e.g.\ Chapter VIII of \cite{Kirby}, since the proof follows many of the the same steps as the proof of Theorems~\ref{thm: spanning manifolds} and \ref{thm: relative spanning manifolds}. Recall that we call $Y$ an oriented spanning manifold for $\S$ extending $Z \subset \del X$ if $Y$ and $Z$ are oriented, and $\del Y = \S \cup Z$ as oriented manifolds. 

The proof strategy, both in the oriented and unoriented case, is as follows. First, impose assumptions on the normal circle bundle $SN_X\S$ that guarantee that it admits a Seifert section $s$. Then consider the inclusions $\S^s_+ \subset S\nu\S \subseteq \del E$. By applying \cref{prop: realising oriented homology classes} (or in the unoriented case, \cref{prop: realising homology classes}) with $M = E$, $A= S\nu\S$, and $B = \S^s_+$, we find a spanning manifold $\wh{Y} \subset E$ for $\S^s_+$ which is contained in the exterior of $\S$. This can then be extended linearly through $\nu\S$ to a spanning manifold $Y$ for $\S$.

\begin{proposition}\label{prop: oriented spanning manifolds}
    Suppose that $X$ and $\S$ are oriented. Then $\S$ admits an oriented spanning manifold if and only if $N_X\S$ is trivial and $[\S] = 0 \in H_n(X, \del X;\Z)$.
\end{proposition}
If $X$ is closed, it is enough to assume that $[\S]=0 \in  H_n(X;\Z)$, since this implies that $N_X\S$ is trivial \cite[Theorem VIII.2]{Kirby}. 
\begin{proof}
    We first prove the forward direction, so suppose that $Y$ is an oriented spanning manifold for $\S$. Then $\S$ must be null-homologous, and $SN_X\S$ must admit a Seifert section, namely the Seifert section associated to $Y$. Since $SN_X\S$ is an oriented $S^1$-bundle over an oriented base which admits a section, it must be trivial. Then $N_X\S$ is also trivial, proving the forward direction.

    For the reverse direction, we use the argument outlined above. Suppose that $N_X\S$ is trivial and that $[\S] = 0 \in H_n(X, \del X;\Z)$. Repeating the argument in \cref{rem: null-homologous seifert} with $\Z$-coefficients, exactness of the bottom row shows that there is a class $\beta \in H_{n+1}(E, \del E;\Z)$ such that
    \begin{equation*}
        \pi_* \del_S \beta = [\S] \in H_n(\S, \del \S;\Z).
    \end{equation*}
    Since $N_X\S$ is trivial, so is $SN_X\S$. Thus for each class $\alpha \in H_n(S\nu \S, \del (S\nu \S);\Z)$ such that $\pi_*\alpha = [\S]$, there is a section $s\colon \S \to S\nu\S$ of $\pi$ such that $s_*[\S] = \alpha$. In particular, there is a Seifert section $s \colon \S \to SN_X\S$ such that
    \begin{equation*}
        [\S_+^s] = \del_S \beta \in H_n(S\nu\S, \del (S\nu\S);\Z).
    \end{equation*}

    After smoothing the corners of $E$, we can apply \cref{prop: realising oriented homology classes} with $M=E$, $A = S\nu\S$, and $B = \S_+^s$ to find a properly embedded compact $(n+1)$-manifold $\wh{Y} \subset E$ such that $\wh{Y} \cap S\nu\S = \S_+^s$ and $[\wh{Y}] = \beta \in H_{n+1}(E,\del E;\Z)$. Finally, we can extend $\wh{Y}$ through $\nu\S$ by setting
    \begin{equation*}
        Y \coloneq \varphi\big(\{ \lambda \cdot  s(x) \mid \lambda \in [0,1],\, x \in \S \}\big) \cup \wh{Y} \subset X.
    \end{equation*}
    Then $Y$ is an oriented spanning manifold for $\S$ as required.
\end{proof}

We can similarly prove an oriented version of \cref{thm: relative spanning manifolds}, classifying when an oriented spanning manifold for $\del \S$ extends to one of $\S$. Recall that we say a spanning manifold $Z \subset \del X$ for $\del \S$ extends to a spanning manifold $Y \subset X$ for $\S$ if $\del Y = \S \cup Z$.

\begin{proposition}\label{prop: relative oriented spanning manifolds}
    Suppose that $X$ and $\S$ are oriented. Let $Z \subset \del X$ be an oriented spanning manifold for $\del \S$ with associated Seifert section $s^Z$. Then $\S$ admits an oriented spanning manifold extending $Z$ if and only if $N_X\S$ is trivial and $[\S \cup Z] = 0 \in H_n(X;\Z)$.
\end{proposition}
\begin{proof}
    The forward direction follows quickly by the same argument as in \cref{prop: oriented spanning manifolds}. So we prove the reverse. If $[\S \cup Z] = 0 \in H_n(X;\Z)$, then there is an $(n+1)$-chain $\beta$ in $X$ with boundary $\del \beta = \S \cup Z$. Then $\beta$ represents a class
    \begin{equation*}
         [\beta] \in H_{n+1}(E, \del E;\Z) \cong H_{n+1}(X, \del X \cup \S;\Z),
    \end{equation*}
    where the isomorphism follows by excision.
    So again $\pi_* \del_S [\beta] = [\S] \in H_n(\S, \del \S;\Z)$. Since $SN_X\S$ is trivial, there is a section $s$ of $SN_X\S$ such that $[\S_+^s] = \del_S [\beta]$. 

    Since both $Z \cap E$ and $\S^s_+$ meet $\del (S\nu\S)$ transversely, we can smooth the corners of $E$ so that
    \begin{equation*}
        B \coloneq (Z \cap E) \cup \S_+^s \subset \del E
    \end{equation*}
    is a smoothly embedded submanifold. Note that $\del [\beta] = [B] \in H_n(\del E;\Z)$. We can thus apply \cref{prop: realising oriented homology classes} with $M=E$ and $A = \del E$ to find a properly embedded compact $(n+1)$-manifold $\wh{Y} \subset E$ such that $\wh{Y} \cap \del E= \S_+^s \cup Z$. 
    As $\wh{Y} \cap S\nu\S = \S^s_+$ is the image of a section of $\pi \colon S\nu \S \to \S$, we may extend $\wh{Y}$ through $\nu \S$ to obtain a spanning manifold $Y \subset X$ for $\S$ as in the proof of \cref{prop: oriented spanning manifolds}.
    Then $Y \cap \del X$ is at least isotopic to $Z$ rel.\ $\del \S$, so we can perform an isotopy to arrange that $Y$ extends $Z$.
\end{proof}

\subsection{Existence of a spanning manifold given a Seifert section}\label{sec: section determines spanning}
We now return to the unoriented situation.
In this subsection, we use the proof strategy outlined in \cref{sec: oriented seifert manifolds} to show that every Seifert section of $s\colon \S \to SN_X\S$ arises as the Seifert section associated to some spanning manifold. We do this in two steps. The first step is to show that if the restriction $\restr{s}{\del \S}$ is the Seifert section associated to some spanning manifold of $\del \S \subset \del X$, then $s$ is the Seifert section associated to some spanning manifold of $\S$.

\begin{lemma}\label{lmm: extending spanning manifolds with specified sections}
    Let $s \colon \S \to SN_X\S$ be a Seifert section and let $Z \subset \del X$ be a spanning manifold for $\del \S$ whose associated Seifert section is $\restr{s}{\del \S}$. Then $Z$ extends to a spanning manifold $Y$ for $\S$ with associated Seifert section $s$ if and only if
    \begin{equation*}
        [(Z \cap E) \cup \S^s_+] = 0 \in H_n(E;\cyc{2}).
    \end{equation*}
    In this case, for any $\beta \in H_{n+1}(E, \del E;\cyc{2})$ such that $\del_S\beta = [\S^s_+]$ and $\del_X \beta = [Z]$, we can choose $Y$ such that $[Y \cap E] = \beta$.
\end{lemma}
\begin{proof}
    Write $B \coloneq (Z \cap E) \cup \S^s_+$. The forward direction follows quickly, since $B = \del (Y \cap E)$, so must be null-homologous in $E$.
    
    We now prove the reverse direction. Suppose that $[B] = 0 \in  H_n(E;\cyc{2})$, so that there exists some $(n+1)$-chain in $E$ with boundary $B$. In particular, we can find some $\beta \in H_{n+1}(E, \del E;\cyc{2})$ such that $\del_S \beta = [\S_+^s]$ and $\del_X\beta = [Z]$. After smoothing the corners of $E$ so that $B \subset \del E$ is a smoothly embedded submanifold, we apply \cref{prop: realising homology classes} with $M=E$ and $A = \del E$ to find a compact properly embedded $(n+1)$-manifold $\wh{Y} \subset E$ with $\wh{Y} \cap A = B$ and $[\wh{Y}] = \beta \in H_{n+1}(E, \del E;\cyc{2})$.
    
    As $\wh{Y} \cap S\nu\S = \S^s_+$ is the image of a section of $\pi \colon S\nu \S \to \S$, we may extend $\wh{Y}$ through $\nu \S$ to obtain a spanning manifold $Y \subset X$ for $\S$ as in the proof of \cref{prop: oriented spanning manifolds}.
    Then $Y \cap \del X$ is at least isotopic to $Z$ rel.\ $\del \S$, so we can perform an isotopy to arrange that $Y$ extends $Z$.
\end{proof}

We can now apply \cref{lmm: extending spanning manifolds with specified sections} to $\del \S \subset \del X$, and use the fact that $\del X$ has empty boundary to argue that $\restr{s}{\del \S}$ must always be the Seifert section associated to some spanning manifold of $\del \S$. Hence any Seifert section $s$ is associated to some spanning manifold for $\S$. 

\begin{proposition}\label{prop: existence of Seifert manifold}
     Let $s \colon \S \to SN_X\S$ be a Seifert section. Then $\S$ admits a spanning manifold $Y$ whose associated Seifert section is $s$.
\end{proposition}
\begin{proof}
    Since $s$ is a Seifert section, we can find an $n$-chain $\alpha$ in $\del X \cap \del E$ and an $(n+1)$-chain $\beta$ in $E$ such that $\del \beta = \alpha \cup \S^s_+$ and $\del \alpha =  \del \S^s_+$. Since $\del X$ has empty boundary, we can apply \cref{lmm: extending spanning manifolds with specified sections} with $\del \S \subset \del X$ in place of $\S \subset X$ and $\restr{s}{\del \S}$ in place of $s$. This gives a spanning manifold $Z \subset \del X$ for $\del \S$ with associated Seifert section $\restr{s}{\del \S}$ and such that $[Z] = [\alpha] \in H_n(\del X, \del \S;\cyc{2})$. Then 
    \begin{equation*}
        [(Z \cap E) \cup \S^s_+] = [\alpha \cup \S^s_+] = [\del \beta] = 0 \in H_n(E;\cyc{2}).
    \end{equation*}
    Finally we can apply \cref{lmm: extending spanning manifolds with specified sections} to $\S$, $s$, and $Z$, to find a spanning manifold for $\S$ with associated Seifert section $s$.
\end{proof}

\begin{corollary}
    The submanifold $\S \subset X$ admits a spanning manifold if and only if $SN_X\S$ admits a Seifert section.
\end{corollary}

Hence the question of whether $\S$ admits a spanning manifold is equivalent to the question of whether $SN_X\S$ admits a Seifert section. The next section will focus on giving sufficient conditions for all possible $\cyc{2}$-homology classes of the total space of an arbitrary $S^1$-bundle to represented by the images of sections, which will give sufficient (but not necessary) conditions for $\S$ to admit a spanning manifold.

\subsection{Existence of Seifert sections}\label{sec: existence of seifert sections} 
We now consider sufficient conditions for $SN_X\S$ to admit a Seifert section, and prove \cref{thm: spanning manifolds}. As discussed in \cref{sec: oriented seifert manifolds}, when $X$ and $\S$ are orientable, $SN_X\S$ admits a Seifert section if and only if $\S$ is null-homologous and $SN_X\S$ admits any section at all, since this implies that $SN_X\S$ is trivial. This is not true without the assumption on the orientability of $X$ and $\S$. Instead, we will have to place other assumptions on the inclusion $\S \subset X$.

We recall some facts about homology with local coefficients, and refer the reader to Chapter VI of \cite{Whitehead} for further details. Let $B$ be a compact $n$-manifold.
Any group homomorphism
\begin{equation*}
    w \colon \pi_1(B)\to \Aut(\Z) \cong \{\pm 1\}
\end{equation*}
defines a left $\Z [\pi_1(B)]$-module with underlying abelian group $\Z$ and the action of $\pi_1(B)$ given by $w$, which we write $\Z^w$. This in turn defines a local coefficient system for $B$. 
The unique non-zero homomorphism $\Z^w \to \cyc{2}$ of $\Z [\pi_1(B)]$-modules given by reduction mod 2 induces change-of-coefficient homomorphisms
\begin{equation*}
    \rho \colon H^*(B; \Z^{w} ) \to H^*(B;\cyc{2}).
\end{equation*}
If $\xi \colon M \to B$ is a vector bundle or sphere bundle, we write $w_1(\xi) \colon \pi_1(B) \to \{\pm 1\}$ for the orientation character of $\xi$. Explicitly, for a loop $\gamma \colon S^1 \to B$, this is given by $w_1(\xi)([\gamma]) = +1$ if the monodromy of $\gamma^*\xi$ preserves the orientation of the fibres, and $w_1(\xi)([\gamma]) = -1$ otherwise. 

Given an $S^1$-bundle $\xi \colon M \to B$ which admits a section, we classify which relative $\cyc{2}$-homology classes in $M$ are represented by the images of sections of $\xi$.

\begin{proposition}\label{lmm: unoriented existence of sections}
    Let $B$ be a compact $n$-manifold and let $\xi \colon M \to B$ be an $S^1$-bundle which admits a section $s \colon B \to M$. Fix $\beta \in H_n(M, \del M;\cyc{2} )$.
    Then there is a section $s' \colon B \to M$ of $\xi$ with $s_*[B] + s'_*[B] = \beta$ if and only if
    \begin{equation*}
        \beta \in \im \Big( H^1(B;\Z^{w_1(\xi)}) \xrightarrow{\rho}  H^1(B;\cyc{2}) \xrightarrow{\xi^*} H^1(M;\cyc{2} ) \xrightarrow{\PD} H_n(M, \del M;\cyc{2} )\Big).
    \end{equation*}
\end{proposition}

\begin{proof}
    We explicitly construct a bijection between $\im(\PD \xi^* \rho) \subseteq H_n(M,\del M;\cyc{2})$ and sections of $\xi$ up to $\cyc{2}$-homology. Note that by working over each connected component separately, we may assume that $B$ is connected. 
    
    Since $\xi$ admits a section $s$, there is a bijection
    \begin{equation*}
        \delta(s,-) \colon \frac{ \{ s'\colon B \to M \text{ a section of } \xi\}}{ \text{homotopy}} \xrightarrow{\cong} H^1(B;\Z^{w_1(\xi)} )
    \end{equation*}
   given by the obstruction-theoretic primary difference \cite[Corollary VI.6.16]{Whitehead}. For a section $s' \colon B \to M$ of $\xi$, write $\bar{\delta}(s,s') \coloneq \rho{\delta}(s,s') \in H^1(B;\cyc{2} )$ for the reduction of the primary difference mod 2.

    We describe $\bar{\delta}(s,s')$ for two sections $s,s'$ of $\xi$. Identify $H^1(B;\cyc{2}) \cong \Hom(H_1(B;\cyc{2}), \cyc{2})$ via the map $\alpha \mapsto \left\langle \alpha,- \right\rangle$. 
    Choose a map $\gamma \colon S^1 \to B$, and consider the pullback bundle $\gamma^*\xi \colon \gamma^*M \to S^1$ with the pullback sections $\gamma^*s$ and $\gamma^*s'$ of $\gamma^*\xi$. 
    By the naturality of the primary difference,
    \begin{equation*}
        \gamma^* \bar\delta(s,s') = \bar\delta(\gamma^*s, \gamma^*s') \in H^1(S^1;\cyc{2}) \cong \cyc{2}.
    \end{equation*}
    Geometrically, this reduced primary difference between $\gamma^*s$ and $\gamma^*s'$ corresponds to the difference in how many times $\gamma^*s$ and $\gamma^*s'$ twist around the fibre, counted mod 2. Hence it is given by the algebraic intersection number of $\gamma^*s(S^1)$ and $\gamma^*s'(S^1)$ in $\gamma^*M$. Explicitly,
    \begin{equation}\label{eq: new dagger}
        \left\langle \bar{\delta}(s,s'), [\gamma(S^1)] \right\rangle = \left\langle \bar\delta(\gamma^*s, \gamma^*s'), [S^1] \right\rangle = [\gamma^*s (S^1)] \cdot [\gamma^*s'(S^1)] \in \cyc{2}. \tag{$\dagger$}
    \end{equation}
    Since $\gamma^*M$ has total space either a Klein bottle or a torus, for any other section $u\colon S^1 \to \gamma^*M$ of $\gamma^*\xi$, we have that
    \begin{equation*}
        [\gamma^*s (S^1)] \cdot [\gamma^*s'(S^1)] = [\gamma^*s(S^1)] \cdot [u(S^1)] + [\gamma^*s'(S^1)] \cdot [u(S^1)],
    \end{equation*}
    and hence by \eqref{eq: new dagger} that
    \begin{equation}\label{eq: other section relation}
        \left\langle \bar{\delta}(s,s'), [\gamma(S^1)] \right\rangle = [\gamma^*s(S^1)] \cdot [u(S^1)] +  [\gamma^*s'(S^1)] \cdot [u(S^1)]. \tag{$\ddagger$}
    \end{equation} 
    Any class in $H_1(B;\cyc{2})$ is represented by the image of a circle, so this fully describes $\bar\delta(s,s')$.

    Next, we describe $\xi^*\bar\delta(s,s') \in H^1(M;\cyc{2})$. Fix a map $\alpha \colon S^1 \to M$ and write $\gamma \coloneq \xi\alpha \colon S^1 \to B$. We can assume that $\alpha$ and $s$ are transverse and intersect away from any double points of $\alpha$.
    
    There is a section $u \colon S^1 \to \gamma^*M$ of $\gamma^*\xi$ given by
    \begin{equation*}
        u(\theta) \coloneq (\theta, \alpha(\theta) ) \in \gamma^*M = \big\{ (t,e) \in S^1 \times M \mid \gamma(t) = \xi(e) \big\}.
    \end{equation*}
    See \cref{fig: sections}. This has the property that $[\gamma^*s(S^1)] \cdot [u(S^1)] = [s(B)] \cdot [\alpha(S^1)] \in \cyc{2}$, where the first algebraic intersection is taken in $\gamma^*M$ and the final one is taken in $M$. This follows from the following computation:
\begin{figure}[t]
    \begin{center}
    \tdplotsetmaincoords{70}{60}

\begin{tikzpicture}[scale=0.8]
  \begin{scope}[tdplot_main_coords]

    \node at (-2.3,-2.3,2.3) {$M$};
    \node[red!60!black,above] at (-1,0,2) {$\alpha$};
    \node at (-2.3,-2.3,-3.7) {$B$};
    \node[above] at (0,1,-4) {$\gamma$};

    \node[green!40!black,left] at (-2,-2,1) {$s(B)$};
    \node[blue!50!black,left] at (-2,-2,-0.5) {$s'(B)$};

  \draw[black, very thin,opacity=0.5] (-2,-2,-2) -- (2,-2,-2) -- (2,2,-2) -- (-2,2,-2) -- cycle;

  \draw[black, very thin,opacity=0.5] (-2,2,-2) -- (-2,2,2);

    \draw[red, very thick]
      plot[variable=\t,domain=-1:-0.81903,samples=50,smooth]
      ({cos(deg(pi*\t))},
       {sin(deg(pi*\t))},
       {3*\t*\t*\t-\t*\t-\t+1});

    \fill[left color = blue!60, right color =blue, opacity=0.3] (-2,-2,-0.5) -- (2,-2,-0.5) -- (2,2,-0.5) -- (-2,2,-0.5) -- cycle;
    \fill[black] ({cos(-147.4254)},
       {sin(-147.4254)},
       {-0.5}) circle (2pt);

    \draw[red, very thick]
    plot[variable=\t,domain=-0.81903:-0.43426,samples=50,smooth]
      ({cos(deg(pi*\t))},
       {sin(deg(pi*\t))},
       {3*\t*\t*\t-\t*\t-\t+1});

    \draw[red, very thick]
    plot[variable=\t,domain=0:0.76759,samples=50,smooth]
      ({cos(deg(pi*\t))},
       {sin(deg(pi*\t))},
       {3*\t*\t*\t-\t*\t-\t+1});

    \fill[left color = green!60, right color =green!70!blue, opacity=0.4] (-2,-2,1) -- (2,-2,1) -- (2,2,1) -- (-2,2,1) -- cycle;

    \fill[black] ({cos(-78.1668)},
       {sin(-78.1668)},
       {1}) circle (2pt);
    \fill[black] ({1},
       {0},
       {1}) circle (2pt);
    \fill[black] ({cos(138.1662)},
       {sin(138.1662)},
       {1}) circle (2pt);

    \draw[red, very thick]
      plot[variable=\t,domain=-0.43426:0,samples=50,smooth]
      ({cos(deg(pi*\t))},
       {sin(deg(pi*\t))},
       {3*\t*\t*\t-\t*\t-\t+1});
       \draw[red, very thick]
      plot[variable=\t,domain=0.76759:1,samples=50,smooth]
      ({cos(deg(pi*\t))},
       {sin(deg(pi*\t))},
       {3*\t*\t*\t-\t*\t-\t+1});

    \draw[black, very thin,opacity=0.5] (-2,-2,2) -- (2,-2,2) -- (2,2,2) -- (-2,2,2) -- cycle;
    \draw[black, very thin,opacity=0.5] (2,2,-2) -- (2,2,2);
    \draw[black, very thin,opacity=0.5] (2,-2,-2) -- (2,-2,2);
    \draw[black, very thin,opacity=0.5] (-2,-2,-2) -- (-2,-2,2);

    \draw[black, very thin,opacity=0.5] (-2,-2,-4) -- (2,-2,-4) -- (2,2,-4) -- (-2,2,-4) -- cycle;
    \draw[black, very thin,densely dashed,opacity=0.5] (0,0,-4) circle (1);

  \end{scope}

  \begin{scope}[tdplot_main_coords, xshift =8cm]

  \draw[black, very thin,opacity=0.5] (0,0,-2) circle (1);
  
    \node at (-1.5,-1.5,2.3) {$\gamma^*M$};
    \node[red!60!black,above] at (-1,0,2) {$u$};
    \node at (-1.5,-1.5,-3.7) {$S^1$};

    \draw[red, very thick]
      plot[variable=\t,domain=-1:-0.81903,samples=50,smooth]
      ({cos(deg(pi*\t))},
       {sin(deg(pi*\t))},
       {3*\t*\t*\t-\t*\t-\t+1});

    \draw[blue!60, very thick] (0,0,-0.5) circle (1);
    \fill[black] ({cos(-147.4254)},
       {sin(-147.4254)},
       {-0.5}) circle (2pt);

    \draw[red, very thick]
    plot[variable=\t,domain=-0.81903:-0.43426,samples=50,smooth]
      ({cos(deg(pi*\t))},
       {sin(deg(pi*\t))},
       {3*\t*\t*\t-\t*\t-\t+1});

    \draw[red, very thick]
    plot[variable=\t,domain=0:0.76759,samples=50,smooth]
      ({cos(deg(pi*\t))},
       {sin(deg(pi*\t))},
       {3*\t*\t*\t-\t*\t-\t+1});

    \draw[green!60!blue!60, very thick] (0,0,1) circle (1);

    \fill[black] ({cos(-78.1668)},
       {sin(-78.1668)},
       {1}) circle (2pt);
    \fill[black] ({1},
       {0},
       {1}) circle (2pt);
    \fill[black] ({cos(138.1662)},
       {sin(138.1662)},
       {1}) circle (2pt);

    \draw[red, very thick]
      plot[variable=\t,domain=-0.43426:0,samples=50,smooth]
      ({cos(deg(pi*\t))},
       {sin(deg(pi*\t))},
       {3*\t*\t*\t-\t*\t-\t+1});
       \draw[red, very thick]
      plot[variable=\t,domain=0.76759:1,samples=50,smooth]
      ({cos(deg(pi*\t))},
       {sin(deg(pi*\t))},
       {3*\t*\t*\t-\t*\t-\t+1});

    \draw[black, very thin,opacity=0.5] (0,0,2) circle (1);
    \begin{scope}[tdplot_main_coords,rotate around z=60]
        \draw[black, very thin,opacity=0.5] (1,0,-2)--(1,0,2);
        \draw[black, very thin,opacity=0.5] (-1,0,-2)--(-1,0,2);

        \node[green!40!black,left] at (-1,0,1) {$\gamma^*s(S^1)$};
    \node[blue!50!black,left] at (-1,0,-0.5) {$\gamma^*s'(S^1)$};
    \end{scope}

    \draw[black, very thin,opacity=0.5] (0,0,-4) circle (1);

  \end{scope}
\end{tikzpicture}
    \caption{A simple model, demonstrating that $[\gamma^*s(S^1)] \cdot [u(S^1)] = [s(B)] \cdot [\alpha(S^1)]$. Left: the $S^1$-bundle $M$, above $B$; a curve $\alpha$ (red), and projection $\gamma$ (dashed); the images of two sections $s,s'\colon B \to M$ (green,blue). The top and bottom faces of $M$ are identified. Right: the $S^1$-bundle $\gamma^*M$, above $S^1$; the curve $u$ (red); the images of the two sections $\gamma^*s,\gamma^*s'$ (green, blue). Note that the intersections between $\alpha$ and $s$ (resp.\ $s'$) in $M$ correspond to the intersections between $u$ and $\gamma^*s$ (resp.\ $\gamma^*s'$) in $\gamma^*M$.}
    \label{fig: sections}
    \end{center}
\end{figure}
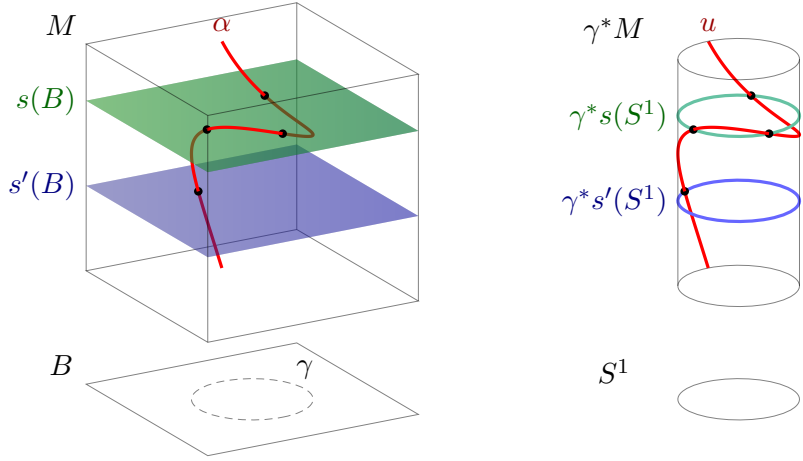
    \begin{align*}
        [\gamma^*s(S^1)] \cdot [u(S^1)] &= \#\{ \theta \in S^1 \mid s\gamma(\theta) = \alpha(\theta) \} \bmod{2} \\
        &= \#\{ (x,\theta) \in B \times S^1 \mid s(x) = \alpha(\theta) \} \bmod{2} \\
        &= [s(B)] \cdot [\alpha(S^1)] \in \cyc{2}.
    \end{align*}
    The second equality follows because the equation $s(x)=\alpha(\theta)$ implies $x=\gamma(\theta)$. Similarly, we get that
    \begin{equation*}
        [\gamma^*s'(S^1)] \cdot [u(S^1)] = [s'(B)] \cdot [\alpha(S^1)].
    \end{equation*}
    Applying \eqref{eq: other section relation} with this choice of $u$ gives that
    \begin{align*}
             \left\langle \bar\delta(s,s') , [\gamma(S^1)] \right\rangle &= [s(B)] \cdot [\alpha(S^1)] +  [s'(B)] \cdot [\alpha(S^1)]\\
             &= \left(s_*[B] + s'_*[B] \right) \cdot [\alpha(S^1)],
    \end{align*}
    and hence that
    \begin{align*}
        \left\langle \xi^*\bar\delta(s,s') , [\alpha(S^1)] \right\rangle & = \left\langle \bar\delta(s,s') , [\gamma(S^1)] \right\rangle\\
        &= \left(s_*[B] + s'_*[B] \right) \cdot [\alpha(S^1)].
    \end{align*}
    Thus $\xi^*\bar{\delta}(s,s')$ is Poincar\'e dual to $s_*[B]+s'_*[B] \in H_{n-1}(M, \del M;\cyc{2})$, or equivalently
    \begin{equation*}
        s_*[B] + s'_*[B] = \PD  \xi^* \rho \delta(s,s').
    \end{equation*}
    Thus the image of the realisation map
    \begin{equation*}
         \frac{ \{ s'\colon B \to M \text{ a section of } \xi\}}{ \text{homotopy}}\to H_n(M, \del M;\cyc{2})
    \end{equation*}
    given by $[s'] \mapsto s_*[B] +s'_*[B]$ is exactly the image of the composition
    \begin{equation*}
        \frac{ \{ s'\colon B \to M \text{ a section of } \xi\}}{ \text{homotopy}} \xrightarrow{\delta(s,-)} H^1(B;\Z^{w_1(\xi)}) \xrightarrow{\PD \xi^* \rho} H_n(M, \del M;\cyc{2}).
    \end{equation*}
    The final result follows since $\delta(s,-)$ is a bijection.
\end{proof}

Since $SN_X\S$ and $S\nu\S$ are isomorphic $S^1$-bundles, they both have the orientation character $w_1(\pi) \colon \pi_1(\S) \to \{\pm 1\}$. This allows us to classify when every possible codimension 1 relative $\cyc{2}$-homology class of $S\nu\S$ is realised by a push-off of $\S$.

\begin{lemma}\label{lemma: new enough sections}
Suppose that $SN_X\S$ admits a section. Then the following are equivalent.
\begin{enumerate}[label=(\roman*)]
        \item For each $\alpha \in H_n(S\nu \S, \del (S\nu\S);\cyc{2})$, there is a section $s\colon \S \to SN_X\S$ such that $[\S^s_+] = \alpha$ if and only if $\pi_*\alpha =[\S] \in H_n(\S,\del \S;\cyc{2})$.
        \item The 3-term sequence
        \[\begin{tikzcd}
	{H^1(\Sigma; \Z^{w_1(\pi)})} & {H_n(S\nu \Sigma, \partial(S \nu \Sigma);\Z/2)} & {H_n(\Sigma, \partial \Sigma;\Z/2)}
	\arrow["{\PD \pi^*\rho}", from=1-1, to=1-2]
	\arrow["{\pi_*}", from=1-2, to=1-3]
\end{tikzcd}\]
        is exact.
        \item The abelian group $H^2(\S;\Z^{w_1(\pi)})$ has no 2-torsion.
\end{enumerate}
\end{lemma}

\begin{proof}
    We first show that (i) $\iff$ (ii). Since $SN_X\S$ admits a section, we may fix a section $s \colon \S \to S\nu\S$ of $\pi$. Fix a class $\alpha \in H_n(S\nu \S, \del (S\nu\S);\cyc{2})$. By    \cref{lmm: unoriented existence of sections}, there is a section $s'$ of $\pi$ with $s'_*[\S]=\alpha$ if and only if $\alpha + s_*[\S] \in \im(\PD \pi^* \rho)$. But $\pi_*\alpha = [\S]$ if and only if $\alpha + s_*[\S] \in \ker \pi_*$. Thus (i) is equivalent to $\im(\PD \pi^* \rho) = \ker \pi_*$, i.e.\ that the sequence in (ii) is exact. So (i) $\iff$ (ii).
    
    We now show that (ii) $\iff$ (iii). This follows from two claims: first, that $\pi^*$ is injective and $\im (\PD  \pi^*) = \ker \pi_*$; and second, that $\rho$ is surjective if and only if $H^2(\S;\Z^{w_1(\pi)})$ has no 2-torsion. 
    
    For the first claim, consider the Gysin sequence of the fibration $\pi$, which begins
    \[\begin{tikzcd}[sep=small]
    	0 & {H_{n-1}(\S, \del \S;\cyc{2})} & {H_n(S\nu\S, \del (S\nu\S) ;\cyc{2})} & {H_n(\S, \del \S;\cyc{2}).}
    	\arrow[from=1-1, to=1-2]
    	\arrow["{\pi_!}", from=1-2, to=1-3]
    	\arrow["{\pi_*}", from=1-3, to=1-4]
    \end{tikzcd}\]
    Here, $\pi_!$ is the umkehr map defined by the composition
    \begin{equation*}
        H_{n-1}(\S, \del \S;\cyc{2}) \xrightarrow{\PD\inv} H^1(\S;\cyc{2}) \xrightarrow{\pi^*} H^1(S\nu\S;\cyc{2}) \xrightarrow{\PD} H_n(S\nu\S, \del (S\nu\S) ;\cyc{2}),
    \end{equation*}
    which is given explicitly on a relative $(n-1)$-cycle $\sigma$ in $(\S, \del \S)$ by $\pi_![\sigma] = [\pi\inv (\sigma)]$.
    Since Poincar\'e duality maps are isomorphisms and $\pi_!$ is injective by exactness, $\pi^*$ must be injective, and
    \begin{equation*}
        \im (\PD   \pi^*) = \im\pi_! = \ker \pi_*.
    \end{equation*}
    This proves the first claim.
    
    For the second claim, consider the  short exact sequence of $\Z[\pi_1(\S)]$-modules
    \[\begin{tikzcd}[sep=small]
	0 & {\Z^{w_1(\pi)}} & {\Z^{w_1(\pi)}} & {\cyc{2}} & 0.
	\arrow[from=1-1, to=1-2]
	\arrow["{\times 2}", from=1-2, to=1-3]
	\arrow[from=1-3, to=1-4]
	\arrow[from=1-4, to=1-5]
\end{tikzcd}\]
This induces the exact sequence
    \[\begin{tikzcd}[sep = small]
	{H^1(\S;\Z^{w_1(\pi)})} & {H^1(\S;\Z/2)} & {H^2(\S;\Z^{w_1(\pi)})} & {H^2(\S;\Z^{w_1(\pi)}),}
	\arrow["{\rho}", from=1-1, to=1-2]
	\arrow["\beta", from=1-2, to=1-3]
	\arrow["{\times 2}", from=1-3, to=1-4]
\end{tikzcd}\]
where $\beta$ is a Bockstein homomorphism. Then by exactness, $\rho$ is surjective if and only if $\beta = 0$, if and only if $H^2(\S;\Z^{w_1(\pi)})$ has no 2-torsion.
\end{proof}

Finally we deduce \cref{thm: spanning manifolds}, giving sufficient conditions for $\S$ to admit a spanning~manifold. 

\begin{namedtheorem}[\cref{thm: spanning manifolds}]
    Let $X$ be an $(n+2)$-manifold and let $\S \subset X$ be a properly embedded compact $n$-manifold. Let $w_1(\pi) \colon \pi_1(\S) \to \Aut(\Z) = \{\pm 1\}$ be the orientation character of $SN_X\S$, and suppose that $H^2(\S ;\Z^{w_1(\pi)} )$ has no 2-torsion. Then $\S$ admits a spanning manifold if and only if $SN_X\S$ admits a section and $[\S] = 0 \in H_n(X, \del X;\cyc{2} )$.
\end{namedtheorem}

\begin{proof}
    The forward direction is clear, and holds without the assumption that $H^2(\S ;\Z^{w_1(\pi)} )$ has no 2-torsion. If $\S$ admits a spanning manifold, then it must be null-homologous, since it bounds a relative chain in $(X, \del X)$.
    It must also admit a Seifert section, and so $SN_X\S$ admits a section.

    We now prove the reverse. By \cref{rem: null-homologous seifert}, if $\S$ is null-homologous then there is a class
    \begin{equation*}
        \alpha \in \im \del_S \subseteq H_n(S\nu \S, \del (S\nu\S) ;\cyc{2})
    \end{equation*}
    such that $\pi_* \alpha = [\S] \in H_n(\S, \del \S; \cyc{2})$.
    Since $H^2(\S ;\Z^{w_1(\pi)} )$ has no 2-torsion and $SN_X\S$ admits a section by assumption, \cref{lemma: new enough sections}(iii)$\Rightarrow$(i) says that there is a section $s \colon \S \to SN_X\S$ such that $[\S_+^s] = \alpha$. Hence $s$ is a Seifert section, and by \cref{prop: existence of Seifert manifold} it is the Seifert section associated to some spanning manifold for $\S$.
\end{proof}

If $H^2(\S,\Z^{w_1(\pi)})$ has non-trivial 2-torsion, more careful examination of the inclusion $\S \subset X$ is required to establish the existence of a spanning manifold. We do not investigate this further.

\subsection{Existence of spanning manifolds with specified boundary}\label{sec: actual seifert stuff}
We now prove \cref{thm: relative spanning manifolds}; that is, we give sufficient conditions for a spanning manifold $Z \subset \del X$ for $\del \S$ to extend to a spanning manifold for $\S$.

After fixing a section of $\restr{SN_X\S}{\del \S}$, the obstruction-theoretic results about extending this to a section of $SN_X\S$ give the following lemma, which is a relative version
of \cref{lemma: new enough sections}.

\begin{lemma}\label{lmm: enough sections rel boundary}
Let $s$ be a section of $SN_X\S$. The following statements are equivalent.
\begin{enumerate}[label=(\roman*)]
        \item For any $\beta \in H_n(S\nu \S;\cyc{2})$, there exists a section $s'\colon \S \to SN_X\S$ extending $\restr{s}{\del \S}$ such that $[\S^s_+ \cup \S^{s'}_+] = \beta$ if and only if $\pi_*\beta = 0 \in H_n(\S;\cyc{2})$.
        \item The sequence
        \[\begin{tikzcd}
	{H^1(\Sigma, \del \S; \Z^{w_1(\pi)})} & {H_n(S\nu \Sigma;\Z/2)} & {H_n(\Sigma;\Z/2)}
	\arrow["{\PD \pi^*\rho}", from=1-1, to=1-2]
	\arrow["{\pi_*}", from=1-2, to=1-3]
\end{tikzcd}\]
        is exact.
        \item The abelian group $H^2(\S,\del\S;\Z^{w_1(\pi)})$ has no 2-torsion.
\end{enumerate}
\end{lemma}
This can be proven by following all of the steps in the proofs of Lemmas~\ref{lmm: unoriented existence of sections} and \ref{lemma: new enough sections}, but working with relative cohomology instead of absolute cohomology throughout. We require the following corollary.

\begin{corollary}\label{lmm: new enough sections rel boundary}
Suppose $H^2(\S,\del\S;\Z^{w_1(\pi)})$ has no 2-torsion. Let $s$ be a section of ${SN_X\S}$, and fix a class $\alpha \in H_n(S\nu \S, \del (S\nu\S);\cyc{2})$.  Then there exists a section $s'\colon \S \to SN_X\S$ extending $\restr{s}{\del \S}$ such that $[\S^{s'}_+] = \alpha$ if and only if $\pi_*\alpha =[\S] \in H_n(\S, \del \S;\cyc{2})$ and $\del \alpha = [\del \S^{s}_+] \in H_{n-1}(\del (S\nu\S);\cyc{2})$.
\end{corollary}

\begin{proof}
    The forward direction is clear, so we prove the reverse.
    Consider the diagram below, which commutes and has exact rows by naturality of the long exact sequence of the pair:
\[\begin{tikzcd}
	& {H_n(S\nu \S;\cyc{2})} & {H_n(S\nu \S,\del (S\nu \S);\cyc{2})} & {H_{n-1}(\del (S\nu\S);\cyc{2})} \\
	0 & {H_n(\S;\cyc{2})} & {H_n(\S,\del \S;\cyc{2}).}
	\arrow["{i_*}", from=1-2, to=1-3]
	\arrow["{\pi_*}", from=1-2, to=2-2]
	\arrow["\del", from=1-3, to=1-4]
	\arrow["{\pi_*}", from=1-3, to=2-3]
	\arrow[from=2-1, to=2-2]
	\arrow["{j_*}", hook, from=2-2, to=2-3]
\end{tikzcd}\]
    Since $\del \alpha = \del [\S^s_+]$, exactness says that there is some $\alpha' \in H_n(S\nu\S;\cyc{2})$ such that $i_*\alpha' = \alpha + [\S^s_+]$. Then because $\pi_*\alpha = [\S]$, we see that $\pi_* i_* \alpha' = j_*\pi_*\alpha' = 0$. Finally, since $j_*$ is injective by exactness, $\pi_*\alpha' = 0$. 
    
    By \cref{lmm: enough sections rel boundary}(iii)$\Rightarrow$(i) and the assumption that $H^2(\S,\del\S;\Z^{w_1(\pi)})$ has no 2-torsion, there is a section $s'$ of $SN_X\S$ extending $\restr{s}{\del \S}$ such that $[\S^s_+ \cup \S^{s'}_+] = \alpha'$. So 
    \begin{equation*}
        i_*[\S^s_+ \cup \S^{s'}_+] = [\S^s_+] + [\S^{s'}_+] = \alpha + [\S^s_+] \in H_n(S\nu \S, \del (S\nu \S);\cyc{2}),
    \end{equation*}
    and $[\S^{s'}_+] = \alpha \in H_n(S\nu \S, \del (S\nu\S) ; \cyc{2})$ as required.
\end{proof}

This is enough to give sufficient conditions for which spanning manifolds of $\del \S$ extend to spanning manifolds $\S$. Note that the statement of \cref{thm: relative spanning manifolds} below is stronger than the one given in the introduction, since it includes the fact that spanning surfaces can be realised in any possible homology class.

\begin{namedtheorem}[\cref{thm: spanning manifolds}]
    Let $Z \subset \del X$ be a spanning manifold for $\del \S$ and let $s^Z$ be the associated Seifert section of $\restr{SN_X\S}{\del \S}$. 
    Let $w_1(\pi) \colon \pi_1(\S) \to \Aut(\Z) = \{\pm 1\}$ be the orientation character of $SN_X\S$, and assume that $H^2(\S, \del \S;\Z^{w_1(\pi)})$ has no 2-torsion. 
    Then $Z$ extends to a spanning manifold for $\S$ if and only if there is a section of $SN_X\S$ extending $s^Z$ and $[Z \cup \S] = 0 \in H_n(X;\cyc{2}).$
    
    In this case, for any $\beta \in H_{n+1}(E, \del E;\cyc{2})$ with $\pi_*\del_S \beta = [\S]$ and $\del_X \beta = [Z]$, there is a spanning manifold $Y$ for $\S$ with boundary $\del Y = Z \cup \S$ and $[Y\cap E] = \beta$.
\end{namedtheorem}
\begin{proof}
    The forward direction follows easily by the same argument as \cref{lmm: extending spanning manifolds with specified sections}, so we prove the reverse.
    Suppose that $[Z \cup \S] = 0$ and that $s \colon \S \to SN_X\S$ is a section with $\restr{s}{\del \S} = s^Z$. Then we can find an $(n+1)$-chain in $X$ with boundary $ Z \cup \S$. In particular, we can find some 
    \begin{equation*}
    \beta \in H_{n+1}(E, \del E;\cyc{2}) \cong H_{n+1}(X, \del X \cup \S;\cyc{2})
    \end{equation*}
    such that $\del_X \beta = [Z]$ and $\pi_*\del_S \beta = [\S]$. Note that any such $\beta$ must also satisfy
    \begin{equation*}
        \del (\del_S\beta) = [\del \S^{s'}_+] \in H_{n-1}(\del (S\nu\S); \cyc{2} ).
    \end{equation*}
    
    Since $\varphi s^Z$ extends to a section $\varphi s \colon \S \to S\nu\S$ of $\pi$, \cref{lmm: new enough sections rel boundary} says that we can choose $s$ such that $[\S^s_+] = \del_{S} \beta\in H_n(S\nu\S, \del (S\nu\S);\cyc{2})$. Then $s$ is a Seifert section of $\pi$, and by exactness of the long exact sequence of the pair $(E, \del E)$,
    \begin{equation*}
        [(Z\cup E) \cup \S^s_+] =  \del \beta = 0 \in H_n(E;\cyc{2}).
    \end{equation*}
    Hence we can apply \cref{lmm: extending spanning manifolds with specified sections} to find a spanning manifold $Y$ for $\S$, extending $Z$, such that $[Y \cap E] = \beta$.
\end{proof}

As outlined in \cref{sec: intro cobordisms}, the first step in our construction of cobordisms between surfaces will be to construct spanning 3-manifolds of embedded surfaces coming from puncturing immersed surfaces around their double points. \cref{thm: relative spanning manifolds} will allow us to do this in such a way that we can guarantee that these spanning manifolds extend annuli on the boundary 3-spheres introduced by puncturing at the double points.

In this case, the 2-torsion condition simplifies, and is automatically satisfied when $X$ is orientable. To see this, note that if $\S$ is a compact surface, then 
\begin{equation*}
    H^2(\S, \del \S; \Z^{w_1(\pi)} ) \cong H_0(\S; \Z^{w_1(T\S) + w_1(\pi)}) = H_0(\S;\Z^{w_1(\restr{TX}{\S})}).
\end{equation*}
So if $X$ is orientable so that $w_1(TX)$ is trivial, then $H^2(\S, \del \S; \Z^{w_1(\pi)} )$ is a free abelian group, so in particular has no 2-torsion. So \cref{thm: relative spanning manifolds} allows us to completely classify when a spanning surface for the boundary $\del \S$ extends to a spanning manifold for $\S$, without any additional assumptions of the embedding of $\S$ into $X$. 
\section{Relative normal Euler numbers}\label{sec: euler number theory}
In this section, we review the normal Euler number of an immersed submanifold, and prove several important technical lemmas. Much of this is well-known to experts, but is recorded for ease of reference and since proofs are not always easily found in the literature. We continue to only work explicitly in the smooth category, though all proofs can be modified to work in the topological category as well.

In \cref{sec: normal Euler number}, we recall the basic definitions in terms of counting intersection points. In \cref{sec: euler cob}, we discuss the normal Euler number as a cobordism obstruction. Section \ref{sec: main euler number section} gives a lemma explaining how to use the mod 2 reduction of the relative normal Euler number to compute the algebraic intersection number between two properly immersed submanifolds with disjoint boundaries. Finally, \cref{sec: knot framings} specialises to the case of surfaces with non-empty boundaries properly immersed in oriented 4-manifolds.

\subsection{Relative normal Euler numbers}\label{sec: normal Euler number}

We give the definition of the relative Euler number of a disc bundle with oriented total space in terms of counting intersections.

\begin{definition}\label{def: euler number}
    Let $B$ be a compact $n$-manifold, let $\xi \colon M \to B$ be a $D^n$-bundle with the total space $M$ oriented, and let $s$ be a section of the sphere bundle $\restr{S \xi}{\del B}$. Let $s'$ be any section of $\xi$ with $\restr{s'}{\del B} =s$. Identify $B$ with the zero section of $\xi$, and assume that $s'$ meets $B$ transversely. For each point $p \in B \cap s'(B)$, choose an arbitrary orientation of the tangent space $T_p B$. This induces an orientation of
    \begin{equation*}
        T_p B \oplus s'_*T_p B = T_p M.
    \end{equation*}
    Let $\varepsilon_p =+1$ if this orientation agrees with the orientation of $T_p M$ coming from the orientation of $M$, and let $\varepsilon_p=-1$ otherwise. The \textit{relative Euler number} $e(\xi, s)$ is the sum 
    \begin{equation*}
         \sum_{p \in B \cap s'(B)} \varepsilon_p \in \Z.
    \end{equation*}
    If $\del B = \varnothing$, we simply write $e(\xi)$. If $B_0$ is a union of connected components of $B$, we abuse notation and write $e(\restr{\xi}{B_0},s)$ instead of $e(\restr{\xi}{B_0},\restr{s}{B_0})$. 
\end{definition}

\begin{remark}\label{rem: normal euler algebraic}
    The relative Euler number $e(\xi,s)$ also has an algebraic formulation. The primary obstruction to the existence of a section of $S \xi$ extending $s$ is $\theta(s) \in H^n(B, \del B; \Z^{w_1(\xi)})$. Since $M$ is orientable, $w_1(\xi) + w_1(TB) \colon \pi_1(B) \to \Aut(\Z)$ is trivial. Then $e(\xi,s)$ is the image of $\theta(s)$ under the composition
    \begin{equation*}
        H^n(B, \del B; \Z^{w_1(\xi)}) \xrightarrow{\PD} H_0(B; \Z) \xrightarrow{\cong} \Z^{\pi_0(B)} \xrightarrow{\varepsilon} \Z,
    \end{equation*}
    where $\varepsilon(a_1,\ldots,a_k) = a_1 + \cdots + a_k$ is the augmentation map.
    See Chapters 9 \& 12 of \cite{MilnorStasheff} or Section 2 of \cite{ConwayOrsonPowell} for a more detailed treatment of this approach.
\end{remark}

Over a connected base space, the relative Euler number is the complete obstruction to the existence of a non-vanishing section of a disc bundle, which extends a given section on the boundary. This follows directly from the algebraic definition, but can also be shown from \cref{def: euler number} by cancelling intersection points of opposite sign.

\begin{lemma}\label{lmm: vanishing normal euler}
    With notation as in \cref{def: euler number}, $s$ extends to a section of $S \xi $ if and only if $e(\restr{\xi}{B_0},s) = 0$ for each connected component $B_0$ of $B$.
\end{lemma}

Now let $X$ be an oriented $2n$-manifold, let $S$ be a compact $n$-manifold, and let $f \colon S \immerse X$ be a proper immersion. Write $\S \coloneq f(S) \subset X$. Recall that any immersion $f$ admits a normal bundle $N(f) \to S$ which fits in the short exact sequence
\[\begin{tikzcd}
	0 & {TS} & {f^*TX} & {N(f)} & 0
	\arrow[from=1-1, to=1-2]
	\arrow[from=1-2, to=1-3]
	\arrow[from=1-3, to=1-4]
	\arrow[from=1-4, to=1-5]
\end{tikzcd}\]
of vector bundles over $S$.
Since $X$ is orientable, the orientation characters $w_1(TX)$ and $w_1(f^*TX)$ are trivial. 
Since orientation characters are additive under extension, $w_1(N(f)) = w_1(TS)$, which means that the total space of $N(f)$ is orientable. The choice of orientation on $X$ also determines one on the total space of $N(f)$ via $f$. 

If $f$ is proper, then $\restr{f}{\del S}$ is an embedding, and so there is a bundle isomorphism
\[\begin{tikzcd}
	{\restr{DN(f)}{\del S}} & {DN_{\del X}{\del \S}} \\
	{\del S} & {\del \S.}
	\arrow["f_*", from=1-1, to=1-2]
	\arrow[from=1-1, to=2-1]
	\arrow[from=1-2, to=2-2]
	\arrow["{\restr{f}{\del S}}", from=2-1, to=2-2]
\end{tikzcd}\]
Hence for any section $s \colon \del \S \to SN_{\del X}{\del \S}$, we can define the \textit{relative normal Euler number}
\begin{equation*}
e(\S,s)  \coloneq e\left(DN(f), f_*\inv\restr{s f}{\del S}\right).
\end{equation*}
The value of $e(\S,s)$ only depends on the isotopy class of $s$ and on the image $\S$, but not the choice of immersion $f$. It depends on the orientation of $X$ only up to sign, since reversing the orientation replaces $e(\S,s)$ with $-e(\S,s)$. Hence we can make sense of the statement $e(\S,s) = 0$ without fixing an orientation of $X$. Equally, if $\S_0, \S_1 \subset X$ are two compact properly immersed submanifolds and $s_i \colon \del \S_i \to SN_{\del X}{\del \S_i}$ is a section for $i=0,1$, we can make sense of the statement $e(\S_0,s_0) = e(\S_1,s_1)$ without fixing an orientation of $X$.

\subsection{Relative normal Euler numbers as cobordism invariants}\label{sec: euler cob}

Although it is not in general true that Euler characteristics are an obstruction to (abstract) cobordism, normal Euler numbers do give an obstruction to properly embedded submanifolds being cobordant. Specifically, if $X$ is an orientable $2n$-manifold with $\S_0, \S_1 \subset X$ properly embedded $n$-manifolds, and there is a cobordism $Y \subset X \times I$ extending $Z = Y \cap (\del X \times I)$, then the normal Euler number $e(\S_0 \stimes \{0\} \cup Z \cup \S_1 \stimes \{1\})$ must vanish. This was observed in the case that $X = \R^4$ in \cite{CKSSBordism}, but in general follows from the lemma below.

\begin{lemma}\label{lmm: euler numbers of boundaries}
    Let $M$ be an orientable $(2n+1)$-manifold and let $Y \subset M$ be a properly embedded compact $(n+1)$-manifold. Then $e(\del Y) = 0$, where $\del Y$ is viewed a closed $n$-manifold embedded in the orientable $2n$-manifold $\del M$.
\end{lemma}
\begin{proof}
    Fix any orientation on $M$.
    Let $(\nu Y, \varphi)$ be a tubular neighbourhood of $Y$ and let $s \colon Y \to DN_{M}Y$ be a section that meets the zero-section transversely. Write $Y^s_+\coloneqq \varphi s(Y)$ for the push-off of $Y$ in the direction of $s$. Then we may assume that $Y \cap Y^s_+$ consists of a collection of circles and arcs properly embedded in $Y$. 
    
    We show $e(\del Y)=0$ using $\restr{s}{\del Y}$ as in \cref{def: euler number}. Since all points of intersection in $\del Y \cap \del Y^s_+$ are endpoints of some properly embedded arc $\gamma \subseteq Y \cap Y^s_+$, it suffices to show that the two endpoints of each such arc contribute opposite signs to the sum in \cref{def: euler number}.

    Choose an arc $\gamma \subseteq Y \cap Y^s_+$ with endpoints $p,q \in \del Y$. Orient $\gamma$ from $p$ to $q$. Fix some orientation on $T_p\del Y$, and give $T_p \del Y^s_+ = (\varphi s)_* T_p Y$ the induced orientation. Without loss of generality, assume that the induced orientation on
    \begin{equation*}
        T_p M = T_p \gamma \oplus T_p \del Y \oplus T_p \del Y^s_+
    \end{equation*}
    agrees with the orientation coming from the fixed orientation on $\del M$, else consider the opposite orientation on $M$. Since $T_p \gamma$ is oriented in the direction of the inwards normal in $\restr{N_{M} \del M}{p}$, the induced orientation of $T_p \del M$ disagrees with the orientation coming from $T_p \del Y \oplus T_p \del Y^s_+$. In the notation of \cref{def: euler number}, this means $\varepsilon_p = -1$.

    However, parallel transporting the orientation of $T_p \del Y$ along $\gamma$ to $q$, we see that the orientation on $T_q M$ also agrees with the orientation coming from $T_q \gamma \oplus T_q \del Y \oplus T_q \del Y^s_+ $. However, since $T_q \gamma$ is oriented in the direction of the outward normal, the induced orientation of $T_q \del M$ agrees with the orientation coming from $T_q \del Y \oplus T_q \del Y^s_+$, so $\varepsilon_q = +1$. 

    Thus the two endpoints of $\gamma$ contribute opposite signs, and we are done.
\end{proof}

\subsection{Computing intersections between immersed submanifolds}\label{sec: main euler number section}

Let $X$ be an oriented $2n$-manifold, let $S$ be a compact $n$-manifold, and let $f \colon S \immerse X$ be a proper immersion. Write $\S \coloneq f(S) \subset X$.
If $\del S = \varnothing$, it is a standard result that $e(\S) \bmod{2}$ agrees with the algebraic intersection $[\S] \cdot [\S] \in \cyc{2}$. This is because both $e(\S) \bmod{2}$ and $[\S] \cdot [\S]$ are given by the mod 2 count of intersections between $\S$ and a push-off (or perturbation) of $\S$ which is chosen to intersect $\S$ transversely. 

The same argument works relative to a section $s \colon \del \S \to SN_{\del X}\del \S$ on the boundary, although one must take the algebraic intersection of $\S$ with a push-off in the direction of a section which extends $s$. The details of this appear to be missing from the literature, save a special case which appears as Lemma 13.8 in \cite{ConwayMiller}. In this subsection we give these details, and deduce a result that relates relative normal Euler numbers to the algebraic intersection of two proper immersions with disjoint boundaries.

First, recall that whenever $A,B \subseteq \del X$ are closed and have disjoint deformation retracts, there is an algebraic intersection form
\begin{equation*}
    - \cdot - \colon H_n(X, A;\cyc{2}) \times H_n(X, B;\cyc{2}) \to \cyc{2}.
\end{equation*}
This is realised by counting the intersections mod 2 of representing properly embedded $n$-manifolds which intersect transversely. In particular, it makes sense to discuss the algebraic intersection number of two properly immersed $n$-manifolds with disjoint boundaries, by taking $A$ to be a neighbourhood of the boundary of one, and taking $B \coloneq \bar{\del X \setminus A}$.

We now discuss how to formalise a push-off of the immersion $f$. Write $\pi \colon DN(f) \to S$ for the projection map. There exists an immersion $\vartheta \colon DN(f) \immerse X$ which restricts to $f \colon S \immerse X$ on the zero-section and is such that every point $x \in S$ has a neighbourhood $U \subseteq S$ such that $\restr{\vartheta}{\pi\inv(U)} \colon \pi\inv(U) \to X$ is an embedding. We may also assume that $\nu \del \S \coloneq \vartheta( \pi\inv(\del S) )$ is a tubular neighbourhood for $\del \S$ in $\del X$.
Then for any section $s'\colon S \to DN(f)$ of $\pi$ which is transverse to the zero section, write $\S_+^{s'} \coloneq \vartheta s'(S) \subset X$. This can be seen as a push-off of $\S$ in the direction of $s'$.

\begin{lemma}\label{cor: normal euler number}
    Fix a section $s\colon \del \S \to SN_{\del X}\del \S$, and any section $s'\colon S \to DN(f)$ with
    \begin{equation*}
        \restr{s'}{\del S} = f_*\inv \restr{sf}{\del S}\colon \del S \to \restr{SN(f)}{\del S}.
    \end{equation*}
    Then $[\S] \cdot [\S^{s'}_+] = e(\S,s) \bmod 2$.
\end{lemma}

\begin{proof}
    We may assume that $\S$ and $\S^{s'}_+$ intersect in isolated transverse double points, so this algebraic intersection number is just the count of intersections between $\S$ and $\S_+^{s'}$ mod 2. By definition, exactly $e(\S,s)$ such intersections occur away from the double points of $f$, counted with sign. Each double point of $f$ gives two more intersections between $\S$ and $\S_+^{s'}$, given by each sheet of $\S$ intersecting the opposite sheet of $\S^{s'}_+$. But these do not contribute to the count mod~2.
\end{proof}

\cref{cor: normal euler number} allows us compare the relative Euler numbers of two homologous properly immersed submanifolds.

\begin{lemma}\label{lmm: even euler number facts}
    Suppose that $S = S_0 \sqcup S_1$ is a disjoint union of two compact $n$-manifolds, and write $\S_i = f(S_i)$ for $i=0,1$ so that $\S = \S_0 \cup \S_1$. Let $Z \subset \del X$ be a spanning manifold for $\del \S$ with associated Seifert section $s^Z$, and suppose that $[Z \cup \S] = 0 \in H_n(X;\cyc{2})$. Then
    \begin{equation*}
        [\S_0] \cdot [\S_1] = e(\S_0,s^Z) \bmod 2 = e(\S_1,s^Z) \bmod 2.
    \end{equation*}
\end{lemma}
\begin{proof}
    Consider the tubular neighbourhood of $\del \S$ given by $\nu \del \S \coloneq \vartheta\big(\pi\inv(\del S) \big) \subset \del X,$
    and let $E_\del \coloneq \bar{\del X \setminus \nu \del \S} \subset \del X$ be the exterior of $\del \S$. Also let
    \begin{equation*}
        M \coloneq \S \cup (Z \cap \nu \del \S) \subset X
    \end{equation*}
    be the subspace given by extending $\S$ along its boundary by a collar of $Z$. Note that $M$ is the image of an immersion $S \immerse X$, homotopic to $f$, which is not a proper immersion. Finally, let $s'\colon S \to DN(f)$ be any section extending $s^Z$, and let $\S^{s'}_+ \coloneq \vartheta s'(S)$ be the push-off of $\S$ in the direction of $s'$. Then $\del M = \del \S^{s'}_+$, and in fact $\S^{s'}_+$ and $M$ are homotopic rel.\ $\del \S^{s'}_+$. So in $H_n(X, E_\del ;\cyc{2})$,
    \begin{align*}
        [\S^{s'}_+] = [M] = [\S \cup (Z \cap \nu \del \S)] &= [\S \cup (Z \setminus E_\del)] \\&= [Z \cup \S] - [Z \cap E_\del] \\&= 0\in H_n(X, E_\del ;\cyc{2}).
    \end{align*}
    Here the last line follows since $[Z \cup \S] = 0 \in H_n(X;\cyc{2})$ by assumption, and $Z \cap E_\del \subset E_\del$.

    Finally, let $f'\colon S \immerse X$ be some immersion, homotopic to $f$, with image $\S^{s'}_+$. Write $\S'_i \coloneq f'(S_i)$ for $i=0,1$, and note that $[\S'_0] = [\S'_1] \in H_n(X, E_\del;\cyc{2})$ since $[\S^{s'}_+] = [\S'_0 \cup \S'_1]=0$ by the previous paragraph. Hence
    \begin{equation*}
        [\S_0] \cdot [\S_1] = [\S_0] \cdot [\S'_1] = [\S_0] \cdot [\S'_0] \in \cyc{2},
    \end{equation*}
    where the first equality follows $\S'_1$ is homotopic to $\S_1$ rel.\ the exterior of $\del \S_0$.
    Finally, \cref{cor: normal euler number} gives that
    \begin{equation*}
         [\S_0] \cdot [\S'_0] = e(\S_0,s^Z) \bmod{2},
    \end{equation*}
    and the result follows.
\end{proof}

\begin{remark}
    By taking $\S = \S_0$ and $\S_1 = \varnothing$, repeating this argument with $\Z$-homology gives a proof that if $\S$ is closed, oriented, embedded, and $[\S] = 0 \in H_n(X;\Z)$, then $e(\S) = 0$. If $\S$ is closed and non-orientable, $e(\S)$ can be non-zero even if $\S$ is embedded and null-homologous. For example, results of Whitney and Massey show that if $\S \subset S^4$ is an embedded closed connected non-orientable surface of Euler characteristic $\chi$, then
    \begin{equation*}
        e(\S) \in \{2\chi - 4, 2\chi, 2\chi+4,\ldots,4-2\chi\},
    \end{equation*}
    and that all these normal Euler numbers are achieved \cite{WhitneyDiffMan}, \cite{Massey}.
\end{remark}

\subsection{Knot framings and relative normal Euler numbers of surfaces}\label{sec: knot framings}
Before discussing surfaces properly immersed in oriented 4-manifolds, we recall some facts about framings of 1-manifolds in oriented 3-manifolds.

Let $Y$ be an oriented 3-manifold. Let $L \subset Y$ be an embedded closed 1-manifold (i.e.\ a link in $Y$) and let $(\nu L,\varphi)$ be a tubular neighbourhood. Then the normal bundle $N_YL$ is trivial, and after fixing an orientation of $L$, two sections $s,s'\colon L \to SN_YL$ are isotopic if and only if $s_*[L]= s'_*[L] \in H_1(SN_YL;\Z)$. In this case, we refer to an isotopy class of sections of $SN_YL$ as a \textit{framing} of $L$. A \textit{Seifert framing} of $L$ is a framing that contains Seifert sections. We regularly abuse notation by using the same symbol to refer to a section of $SN_YL$ and the framing it represents.

Let $K \subset L$ be a connected component of $L$. If $Y= S^3$ with the usual orientation, there is a canonical identification
\begin{equation*}
    \fr_K \colon \{\text{framings of }K\} \xrightarrow{\cong} \Z
\end{equation*}
given by the linking number,
\begin{equation*}
\fr_K(s) \coloneq \lk(K,s(K)).
\end{equation*}

 For general oriented $Y$, the set of framings of $K$ still has a $\Z$-torsor structure, which we denote by $\star$. Let $\nu K$ be the component of $\nu L$ containing $K$. The orientation on $Y$ induces one on $\nu K$ and $S\nu K$, and hence one on $SN_YK$ via the identification $\varphi\colon DN_Y L \to \nu L$. Choose an orientation on $K$, which in turn induces one on the $S^1$-fibres of $SN_YK$. Then for a fixed framing $s$ and integer $n \in \Z$, there is a unique framing $s'$ of $K$ such that the algebraic intersection number $s_*[K] \cdot s'_*[K] = n \in  \Z$. This framing is independent of the choice of orientation of $K$, so we write $n \star s = s'$. Geometrically, $n$ acts by adding $n$ twists in the positive direction.
If $s$ is a framing of $L$, we abuse notation and write  $\fr_K(s)$ instead of $\fr_K(\restr{s}{K})$.

Now return to the situation where $X$ is an oriented 4-manifold, $S$ a compact surface, and $f \colon S \immerse X$ a proper immersion with image $\S \coloneq f(S)$. Then the $\Z$-action on the framings of $\del \S \subset \del X$ intertwines the relative normal Euler number of $\S$.

\begin{lemma}\label{lmm: twisting euler numbers}
    Fix a component $K$ of $\del \S$. Let $s, s'$ be framings of $\del \S$ which agree on $\del \S \setminus K$. If $\restr{s'}{K} = n \star \restr{s}{K}$, then $e(\S, s') = n+ e(\S,s)$.
\end{lemma}
This can be seen from the boundary twisting operation described in Section 1.3 of \cite{FreedmanQuinn}; we give a slightly different proof which is similar in spirit.
\begin{proof}
    It suffices to consider the case that $e(\S,s)=0$. Then we may compute $e(\S,s')$ using an extension of $s'$ to $S$ which is nowhere-vanishing outside of a collar of $K$, and which is the trace of an isotopy from $s'$ to $s$ in this collar. Then $e(\S,s')$ is the signed count of intersections of this isotopy with the zero-section in the collar. Thus it suffices to replace $\S$ with just this collar of $K$, that is to consider the case that $X = S^1 \times I \times D^2$, $\S = S^1 \times I \times \{0\}$.
    By arranging the intersection points to occur in different levels $S^1 \times \{\pt\} \subset \S$, it also suffices to consider that case $n = 1$.
    
    Finally, by viewing $S^1 \times I$ as a collar neighbourhood of the boundary of $D^2$, it in fact suffices to consider the case $X = D^2 \times D^2$, $\S = D^2 \times \{0\}$, $n=1$, and $e(\S,s)=0$. In this case, explicit representatives $s,s' \colon S^1 \times \{0\} \to S^1 \times S^1$ can be given by $s(x,0) = (x,\{1\})$ and $s'(x,0) = (x,x)$. Then $s' = 1 \star s$. By considering the extensions $D^2 \times \{0\} \to D^2 \times D^2$ given by the same formulae, we see that $e(\S,s) = 0$ and $e(\S,s') = 1$.
\end{proof}

We can also rewrite \ref{thm: relative spanning manifolds} as follows in the case $n=2$.

\begin{corollary}\label{cor: key surface spanning}
    Let $X$ be an orientable 4-manifold and let $\S \subset X$ be a properly embedded compact surface. 
    Let $Z \subset \del X$ be a spanning surface for $\del \S$ and let $s^Z$ be the associated Seifert section. 
    Then $\S$ admits a spanning manifold extending $Z$ if and only if $[Z \cup \S] = 0 \in H_n(X;\cyc{2})$ and $e(\S_0,s^Z) = 0$ for each connected component $\S_0 \subseteq \S$.
\end{corollary}

This follows from \cref{lmm: vanishing normal euler} and recalling that the 2-torsion condition is automatically satisfied when $X$ is orientable.

\section{Spanning manifolds of punctured immersed surfaces}\label{sec: new surfaces sec}

In order to construct cobordisms between properly embedded surfaces in 4-manifolds from spanning 3-manifolds, we need to develop a theory of spanning 3-manifolds for pairs of properly embedded surfaces, which may intersect. Their union is then an immersed surface, which is considered up to isotopy of the two surfaces. 

To this end, we develop a theory of spanning 3-manifolds for arbitrary properly immersed surfaces in oriented 4-manifolds. To do this, we puncture the ambient 4-manifold around the double points of the immersion to yield a proper embedding, and then only consider spanning 3-manifolds which extend annuli on the $S^3$ boundary components introduced by puncturing. We give conditions for such spanning 3-manifolds exist and have prescribed boundary when the immersion is allowed to be modified either by any regular homotopy, or only by a regular homotopy which decomposes as a pair of isotopies, at least one of which is constant on each connected component of the surface.

We now make this precise. Let $X$ be an oriented 4-manifold, let $S$ be a compact surface, let $f \colon S \immerse X$ be a proper immersion, and write $\S \coloneq f(S)$. 
Since we assume $f$ is generic, all self-intersections of $\S$ are transverse double points. Each such double point $p$ has a 4-ball neighbourhood $D \subset X$ such that the pair $(D, D \cap \S)$ is diffeomorphic to $(D^4, D^2 \stimes \{0,0\} \cup \{0,0\} \stimes D^2)$. The pair $(\del D,\del D \cap \S)$ a Hopf link in $S^3$ \cite[\textsection4.6]{GompfStipsicz}. We call the immersion
\begin{equation*}
    \restr{f}{f\inv(\bar{X \setminus D})} \colon f\inv(\bar{X \setminus D}) \immerse \bar{X \setminus D}
\end{equation*}
the \textit{immersion of $f$ punctured at $p$}.

\begin{definition}\label{def: almost-extendable}
The \textit{punctured embedding of $f$} is the embedding $\wh f \colon \wh S \to \wh X$ given by puncturing $f$ at all double points, so that $\wh \S \coloneq \wh f(\wh S)$ is properly embedded in $\wh X$. If $Z \subset \del X$ is a spanning surface for $\del \S$, we say that $Z$ is \textit{almost-extendable} over $\S$ if there is a spanning manifold $\wh Y$ for $\wh\S$ such that $\wh Y \cap \del X =Z$, and for each component $\del_0 \wh X \subseteq \del \wh X \setminus \del X$, the surface $\wh Y \cap \del_0 \wh X$ is an annulus.
\end{definition}

That is, $Z$ is almost-extendable over $\S$ if we can find annuli spanning the Hopf links in $\del\wh{\S} \subset \del \wh{X}$ introduced by puncturing, such that the union of $Z$ with these annuli extends to a spanning 3-manifold of $\wh{\S}$. Our main aim in this section is prove the following result, which will be an immediate corollary of our more general results.

\begin{proposition}\label{cor: puncturing embeddings}
    Let $X$ be an orientable 4-manifold, and let $\S_0,\S_1 \subset X$ be two properly embedded compact surfaces. Let $Z \subset \del X$ be a spanning surface for $\del \S_0 \cup \del \S_1$ with associated Seifert section $s^Z$. Then there are surfaces $\S'_0$ and $\S'_1$ isotopic to $\S_0$ and $\S_1$ respectively such that $Z$ is almost-extendable over $\S'_0 \cup \S'_1$ if and only if $[\S_0 \cup Z \cup \S_1] = 0 \in H_2(X;\cyc{2})$ and $e(\S_0,s^Z) = e(\S_1,s^Z)$. 
\end{proposition}

\begin{remark}\label{remark: BS}
    The notion of an almost-extendable spanning surface is inspired by the arguments in \cite{BaykurSunukjian} to show that surfaces $\S_0,\S_1$ in properly embedded in an orientable 4-manifold $X$ are weakly internally stably isotopic (that is, can be made isotopic after finitely many ambient self-connected sums) if and only if (i) they are orientable and $\Z$-homologous with some choice of orientations, or (ii) they are non-orientable, $\cyc{2}$-homologous, and have the same normal Euler number. 

    We summarise that argument here. After an isotopy, we may assume that each component of $\S_0$ intersects each component of $\S_1$ in at least one point. We can then resolve intersections to obtain a properly embedded connected surface $\S \subset X$, by replacing small neighbourhoods of the double points of $\S_0 \cup \S_1$ by annuli. Under the stated assumptions, \cref{cor: key surface spanning} says that $\S$ admits a spanning manifold $Y \subset X$. Choose a handle decomposition of $Y$ rel.\ $\S_0 \cap Y$. The 1-handles specify internal stabilisations of $\S_0$ and the 2-handles specify internal stabilisations of $\S_1$, which result in isotopic surfaces. In fact, the two resulting surfaces are equal away from the intersection points, and near an intersection point they are related by an isotopy along the chosen annulus.

    This method does not quite suffice for our purposes. In particular, it is important for us that the spanning 3-manifold $Y$ lies entirely outside of the neighbourhood of the double points of $\S_0 \cap \S_1$ whose boundary contains the annulus resolving the intersection. We ensure this by considering spanning manifolds of the punctured embedding of $\S_0 \cup \S_1$, rather than simply resolving intersections.
\end{remark}

\subsection{Regular and generic homotopies}\label{sec: regular homotopies}

A regular homotopy of $f$ is a homotopy rel.\ $\del S$ through immersions. A generic homotopy rel.\ $\del S$ is a composition of regular homotopies and cusp homotopies. A regular homotopy is generically a composition of isotopies, finger moves, and Whitney moves \cites{WhitneySingularities1, WhitneySingularities2}, \cite[\textsection III.3]{GolubitskyGuillemin}.
We omit precise definitions of these moves, and instead refer the reader to Chapter 1 of \cite{FreedmanQuinn} or Chapter XII of \cite{Kirby}.

We write $\self(\S) \in \Z_{\geq 0}$ for the number of self-intersections of $f$, not counted with sign since $S$ is unoriented. Note that $\self(\S)$ depends only on the image $\S$, justifying the notation. 
If $\S$ is obtained from $\S'$ by a finger move or from $\S''$ by a cusp homotopy, then 
\begin{equation*}
    \self(\S) = \self(\S')+ 2 = \self(\S'') + 1,
\end{equation*}
and
\begin{equation*}
    e(\S,s) = e(\S',s) = e(\S'',s) \pm 2
\end{equation*}
for any framing $s$ of $\del \S$. In particular, neither $e(\S,s)$ or $2\self(\S) \bmod{4}$ are affected by regular homotopies; a cusp homotopy changes both $e(\S,s) \bmod{4}$ and $2\self(\S) \bmod{4}$ by $2 \in \cyc{4}$. In particular, the quantity
\begin{equation*}
    e(\S,s) - 2 \self(\S) \mod{4}
\end{equation*}
is an invariant of the pair $(\S,s)$ under homotopy rel.\ $\del \S$.

\subsection{Conditions for a spanning surface to be almost-extendable}
In order to give conditions for a spanning surface $Z \subset \del X$ for $\del \S$ to be almost-extendable over $\S$, we need to consider the hypotheses of \cref{cor: key surface spanning} as applied to the punctured embedding $\wh{\S} \subset \wh{X}$. To this end, we check the effect of puncturing on homology and relative normal Euler numbers.

\begin{lemma}\label{lmm: finger puncture boundary lemma}
    Let $f' \colon S' \immerse X'$ be obtained by puncturing $f$ at a double point. Write $\S' \coloneq f'(S')$.
    \begin{enumerate}[label=(\roman*)]
        \item The class $[\S'] = 0 \in H_2(X', \del X';\cyc{2})$ if and only if $[\S] = 0 \in H_2(X,\del X;\cyc{2})$.
        \item Let $Z \subset \del X'$ be a spanning surface for $\del \S'$. Then $Z \cap \del X$ is a spanning surface for $\del \S$, and $[Z \cup \S'] = 0 \in H_2(X';\cyc{2})$ if and only if $[(Z \cap \del X) \cup \S] = 0 \in H_2(X;\cyc{2})$.
        \item Let $s$ be a framing of $\del \S'$ and let $\del \S' = \del \S \sqcup K_0 \sqcup K_1$. Then
        \begin{equation*}e(\S',s) = e(\S,s) + \fr_{K_0}(s) + \fr_{K_1}(s). \end{equation*}
    \end{enumerate}
\end{lemma}
\begin{proof}
We prove (i)--(iii) in order. Let $D \subset X$ be the 4-ball such that $X' = \bar{X \setminus D}$. Then the image of $[\S]$ under the composite isomorphism
\[\begin{tikzcd}[sep=small]
	{H_{2}(X, \del X;\cyc{2})} & {H_2(X, \del X \cup D;\cyc{2})} & {H_2(X',\del X';\cyc{2}),}
	\arrow["\cong", from=1-1, to=1-2]
	\arrow["\cong", from=1-2, to=1-3]
\end{tikzcd}\]
is exactly $[\S']$, proving (i).

For (ii), note that the map $H_2(X';\cyc{2}) \to H_2(X;\cyc{2})$ induced by inclusion is an isomorphism, so it suffices to show that $[Z \cup \S'] = [(Z \cap \del X) \cup \S] \in H_2(X;\cyc{2})$. This is true, since
\begin{equation*}
    [Z \cup \S'] + [(Z \cap \del X) \cup \S] = [(Z \cap D) \cup (\S \cap D)] \in \im( H_2(D;\cyc{2}) \to H_2(X;\cyc{2}) ),
\end{equation*}
so must be trivial.

For (iii), let $s'$ be the framing of $\del \S'$ such that $\restr{s'}{\del \S} = \restr{s}{\del \S}$, but $\fr_{K_0}(s') = \fr_{K_1}(s') = 0$. By \cref{lmm: twisting euler numbers}, it suffices to show that $e(\S',s') = e(\S,s)$. But this is true, since $e(\S,s)$ can be computed using the intersections with the zero-section of an extension of $s$ which is non-vanishing on $\S \cap D$.
\end{proof}

We can now give a combinatorial condition for when a given spanning surface $Z$ for $\del \S$ is almost-extendable over $\S$, depending on the specifics of the intersections between different components of $S$ under $f$. 

\begin{lemma}\label{lmm: combinatorial condition}
    Let $Z \subset \del X$ be a spanning surface for $\del \S$ with associated Seifert framing $s^Z$. Let $n = \self(\S)$, and let $p_1,\ldots,p_n \in \S$ be the double points of $f$. For each component $C \subset S$ and each $i=1,\ldots,n$, let
    \begin{equation*}
        \cP^C_i \coloneq \#\big(C \cap f\inv(\{p_i\})\big) \in \{0,1,2\}
    \end{equation*}
    be the number of preimages of $p_i$ in $C$.
    Then $Z$ is almost-extendable over  $\S$ if and only if $[Z \cup \S] = 0\in H_2(X;\cyc{2})$ and there exists a choice of $\varepsilon_1,\ldots,\varepsilon_n \in \{\pm 1\}$ such that for each component $C \subseteq S$,
    \begin{equation*}
    e(f(C), s^Z) =  \sum_{i=1}^n\cP^C_i\varepsilon_i.
\end{equation*}
\end{lemma}
\begin{proof}
    We will prove the equivalence directly by considering the hypotheses of \cref{cor: key surface spanning}.    
    Let $\wh f \colon \wh S \to \wh X$ be the punctured embedding of $f$, and write $\wh\S \coloneq \wh{f}(\wh{S})$. For each $i=1,\ldots,n$, let $D_i \subset X$ be the 4-ball neighbourhood of the double point $p_i$ such that $\wh X = \bar{X \setminus (D_1 \cup \cdots \cup D_n)}$. 
    
By \cref{cor: key surface spanning}, $Z$ is almost-extendable over $\S$ if and only if, for each $i = 1,\ldots,n$, there is an annulus $A_i$ spanning the Hopf link $\S \cap \del D_i$ with associated Seifert section $\wh{s}_i$ such that the following two conditions are met. 
\begin{itemize}
    \item The first condition is that
\begin{equation*}
    \Big[Z \cup \bigcup_i A_i \cup \wh{\S}\Big] = 0 \in H_2(\wh{X}; \cyc{2}),
\end{equation*}
which is equivalent to $[Z \cup \S] = 0\in H_2(X;\cyc{2})$ by \cref{lmm: finger puncture boundary lemma}(ii). 

    \item The second condition is that the section $s^Z \cup\wh{s}_1 \cup \cdots \cup \wh{s}_n$ of $SN_{\del \wh{X}} \del \wh{\S}$ extends to a section of $SN_{\wh{X} }\wh{\S}$, or equivalently that
\begin{equation*}
    e\big(\wh{f}(C \cap \wh{S}),s^Z \cup\wh{s}_1 \cup \cdots \cup \wh{s}_n\big) = 0.
\end{equation*}
By \cref{lmm: finger puncture boundary lemma}(iii), this holds if and only if for each component $C \subseteq S$, 
\begin{equation*}
     e(f(C),s^Z) + \sum_{i=1}^n \sum_{K \subseteq f(C) \cap \del D_i} \fr_{K}(\wh{s}_i) = 0,
\end{equation*}
where the second sum is taken over the (possibly empty) set of components of $f(C) \cap \del D_i$.
\end{itemize}

For each double point $p_i$, there are two choices of spanning annulus $A_i$ up to isotopy. If the Hopf link $f(S) \cap \del D_i$ has components $K \sqcup K'$, then these annuli have Seifert framings specified by
\begin{equation*}
    \fr_{K}(\wh{s}_i) = \fr_{K'}(\wh{s}_i) \in \{ \pm 1\}.
\end{equation*}
Hence there exist valid choices of $A_i$ if and only if we can choose some $\varepsilon_i = -\fr_K(\wh{s}_i) \in \{\pm 1\}$ for each $i = 1,\ldots,n$ such that
\begin{align*}
    e(f(C),s^Z) &= \sum_{i=1}^n \sum_{K \subseteq f(C) \cap \del D_i} \varepsilon_i \\&= \sum_{i=1}^n \cP^C_i \varepsilon_i.
\end{align*}
This completes the proof of equivalence.
\end{proof}

\subsection{Conditions for a spanning surface to be almost-extendable after homotopy}\label{sec: finger-puncture stuff}

The second condition of \cref{lmm: combinatorial condition} is very difficult to check in practice. 
However, if $f$ is allowed to be modified by a regular homotopy, the situation simplifies.

\begin{proposition}\label{lmm: homotopy puncture}
    Let $Z \subset \del X$ be a spanning surface for $\del \S$ with associated Seifert section $s^Z$. Then the following are equivalent.
    \begin{enumerate}[label=(\roman*)]
        \item There exists an immersion $g\colon S \immerse X$ homotopic rel.\ $\del S$ to $f$ such that $Z$ is almost-extendable over $g(S)$.
        \item The class $[Z \cup \S] = 0\in H_2(X;\cyc{2})$ and $e(\S,s^Z) \equiv 2\self(\S) \bmod{4}$.
    \end{enumerate}
    In this case, $g$ can be taken to be regularly homotopic to $f$. If $S = S_0 \sqcup S_1 \sqcup S_2$ is a disjoint union of three non-empty compact surfaces, then the homotopy may be taken to be a sequence of finger moves between $S_0$ and $S_1$, between $S_0$ and $S_2$, and between $S_1$ and $S_2$.
\end{proposition}
\begin{proof}
We first show that (i) $\Rightarrow$ (ii). Suppose we have chosen some immersion $g \colon S \immerse X$ homotopic to $f$ rel.\ $\del S$ such that $Z$ is almost-extendable over $g(S)$. Then $g$ satisfies the conditions of \cref{lmm: combinatorial condition}; that is, $[Z \cup \S] = 0\in H_2(X;\cyc{2})$, and there exists a choice of $\varepsilon_1,\ldots,\varepsilon_n \in \{\pm 1\}$ such that for each component $C \subseteq S$,
    \begin{equation*}\label{eq: double dagger}
    e(g(C), s^Z) =  \sum_{i=1}^n\cP^C_i\varepsilon_i.\tag{$*$}
\end{equation*}
Summing over all components $C$ in \eqref{eq: double dagger}, we see that $e(g(S), s^Z) = \sum_i 2\varepsilon_i$ and hence that
\begin{equation*}
    e(g(S),s^Z) \equiv 2 \self(g(S)) \mod{4}.
\end{equation*}
As remarked at the end of \cref{sec: regular homotopies}, the quantity $ e(g(S),s^Z) - 2 \self(g(S)) \bmod{4}$ is an invariant of homotopy rel.\ boundary, and so
\begin{equation*}
    e(\S,s^Z) - 2\self(\S) \equiv e(g(S),s^Z) - 2\self(g(S)) \equiv 0 \mod{4}.
\end{equation*}
So (ii) follows as required.

We now show that (ii) $\Rightarrow$ (i). It suffices to show that the assumption $e(\S,s^Z) \equiv 2\self(\S) \bmod{4}$ guarantees that we can find a suitable immersion $g\colon S \immerse X$ and signs $\varepsilon_i \in \{\pm 1\}$ as in \cref{lmm: combinatorial condition}.

Choose some collections of components $S_0, S_1, S_2 \subseteq S$, all non-empty, such that $S = S_0 \cup S_1 \cup S_2$. We do not in general assume that the collections are pairwise disjoint; e.g.\ if $S$ is connected, then the choice $S=S_0=S_1=S_2$ is forced. We show that there exist finger moves taking $f$ to some $g\colon S \immerse X$ such that suitable $\varepsilon_i$ can be chosen. Moreover, all finger moves will be constructed between components such that there exist $i \neq j$ where the first component is in $S_i$ and the second is in $S_j$. This proves the final remark in the case that $S_0$, $S_1$, and $S_2$ are pairwise disjoint.

Suppose some $g\colon S \immerse X$ which is regularly homotopic to $f$ via a suitable sequence of finger moves has been chosen. Let $n \coloneq \self(g(S))$, and let $g$ have double points $p_1,\ldots,p_n$. As before, for a component $C \subseteq S$ and index $i=1,\ldots,n$, we write $\cP^C_i \coloneq  \#\big( C \cap g\inv(\{p_i\})\big)$. Write also $\cP^C \coloneq \cP^C_1 + \cdots + \cP^C_n$ for the total number of preimages of double points in $C$. 

By performing more finger moves if necessary, we may assume that $\cP^C \geq \abs{e(g(C), s^Z)}$ for each component $C \subseteq S$.
This is because performing finger moves does not affect relative normal Euler numbers, but performing a finger move between $C$ and another component increases $\cP^C$ by 2; performing a finger move between $C$ and itself increases $\cP^C$ by 4.

We now show that we assume that $\cP^C \equiv e(g(C),s^Z) \bmod{4}$ for all components $C \subseteq S$.
By \cref{lmm: even euler number facts} and the fact that $[Z \cup g(S)] = [Z \cup \S] = 0 \in H_2(X;\cyc{2})$, we see that for each component $C \subseteq S$,
    \begin{equation*}
       \cP^C \bmod{2}= [g(C)] \cdot [g(S \setminus C)]  = e(g(C), s^Z) \bmod{2}.
    \end{equation*}
Since $g$ differs from $f$ by regular homotopy,
\begin{equation*}
    \sum_C e(g(C),s^Z) \equiv e(g(S),s^Z) \equiv e(\S,s^Z) \equiv 2\self(\S) \mod{4},
\end{equation*}
where the final congruence follows by assumption. Moreover, since $g$ only differs from $f$ by finger moves,
\begin{equation*}
    \sum_C\cP^C \equiv 2 \self(g(S)) \equiv 2\self(\S) \mod{4},
\end{equation*}
and so $\sum_C e(g(C),s^Z) \equiv \sum_C\cP^C \bmod{4}$.

Suppose there is a component $C \subseteq S$ such that $\cP^C \not\equiv e(g(C), s^Z) \bmod{4}$. Then there must a second component $C' \subseteq S$ with $\cP^{C'} \not\equiv e(g(C'), s^Z) \bmod{4}$. If there are $i \neq j \in \{0,1,2\}$ such that $C \subseteq S_i$ and $C' \subseteq S_j$, we perform a finger move between $C$ and $C'$. This does not affect $e(g(C), s^Z)$ or $e(g(C'), s^Z)$, but increases both $\cP^C$ and $\cP^{C'}$ by 2.
If not, suppose $C,C' \subset S_i$, and choose $j \neq i$ and a component $D \subseteq S_j$. Then perform a finger move between both $C$ and $D$, and $C'$ and $D$. This again does not affect any relative normal Euler numbers, but increases $\cP^C$ and $\cP^{C'}$ by 2, and $\cP^D$ by 4.
In this way, we can assume that $\cP^C \equiv e(g(C),s^Z) \bmod{4}$ for all components $C \subseteq S$.

We next describe how to assign a unit $\varepsilon^C_i \in \{ \pm 1\}$ to all components $C \subseteq S$ and all $i=1,\ldots,n$, such that 
    \begin{equation*}
        \sum_ {i=1}^n \cP^C_i \varepsilon^C_i = e(g(C),s^Z).
    \end{equation*}
If we can later arrange that the value of $\varepsilon_i^{C}$ is independent of the component $C \subseteq S$, then we can apply \cref{lmm: combinatorial condition} to show that $Z$ is almost-extendable over $g$, proving (i).

Fix some component $C \subseteq S$, and let $k \in \Z$ be such that
\begin{equation*}
    \cP^C = \sum_ {i=1}^n \cP^C_i = e(g(C),s^Z) + 4k.
\end{equation*}
Since $\cP^C_i \in \{0,1,2\}$ for each $i$ and $\cP^C \geq \abs{e(g(C),s^Z)}$ for each component $C \subseteq S$, we can find a subset $\cI \subseteq \{1,\ldots,n\}$ such that $\sum_{i \in \cI} \cP^C_i = 2k$. Then
\begin{equation*}
    \sum_{i \notin \cI} \cP^C_i - \sum_{i \in \cI} \cP^C_i = e(g(C),s^Z).
\end{equation*}
Thus, we let $\varepsilon_i^C = -1$ if $i \in \cI$ and let $\varepsilon_i^C = +1$ if $i \notin \cI$, to get that $\sum_i \cP^C_i \varepsilon^C_i = e(g(C),s^Z)$.

We now wish to perform a sequence of moves on the immersion $g$ and the choices of signs $\varepsilon^C_i$, which do not affect $e(g(C),s^Z)$ or the sum $\sum_i \cP^C_i \varepsilon^C_i$ for any component $C$, but arrange that for any two components $C, C' \subseteq S$ and any $i=1,\ldots,n$, we have that $\varepsilon_i^{C} = \varepsilon_i^{C'}$. To do this, we allow three types of moves. 

The first move is to find components $C$ and $D$, possibly equal, and perform a finger move between them. This introduces two new double points $p_{n+1}$ and $p_{n+2}$. We set $\varepsilon^F_{n+1} = +1$ and $\varepsilon^F_{n+2}=-1$ for all components $F \subset S$. Performing finger moves does not affect relative normal Euler numbers, and the choice of signs ensure that $\sum_i \cP^C_i \varepsilon^C_i$ and $\sum_i \cP^D_i \varepsilon^D_i$ are unchanged. Note that in order to perform this move, there must be $i \neq j$ such that $C \subseteq S_i$ and $D \subseteq S_j$.

The second move is to find two double points $p_i$ and $p_j$ and a component $C$ such that $\varepsilon_i^C = -\varepsilon_j^C$ and $\cP^C_i = \cP^C_j$. We can then swap the signs of $\varepsilon^C_i$ and $\varepsilon^C_j$. This clearly does not affect either $e(g(C),s^Z)$ or the sum $\sum_i \cP^C_i \varepsilon^C_i$ for any component $C$.

The third and final move is to find a double point $p_i$ and a component $C$ such that $\cP^C_i = 0$. We can then replace $\varepsilon^C_i$ with $-\varepsilon^C_i$. This again affects neither $e(g(C),s^Z)$ nor the sum $\sum_i \cP^C_i \varepsilon^C_i$.

To keep track of the effects of these moves, we assign each double point $p_i$ (or more precisely, each index $i=1,\ldots,n$) to one of four types (I)--(IV) based on their interactions with the components of $S$. These types are described below.
    \begin{enumerate}[label=(\Roman*)]
        \item For any two components $C,C' \subseteq S$, $\varepsilon^C_i = \varepsilon^{C'}_i$.
        \item There exist two components $C,C' \subset S$ such that $\varepsilon^C_i \neq \varepsilon^{C'}_i$, but in all such cases $\cP^C_i \cP^{C'}_i = 0$.
        \item There exist two components $C,C' \subset S$ such that $\varepsilon^C_i \neq \varepsilon^{C'}_i$ and $\cP^C_i = \cP^{C'}_i = 1$, but in all such cases there exist $k \neq \ell \in \{0,1,2\}$ such that $C \subseteq S_k$ and $C' \subseteq S_\ell$.
        \item There exist two components $C,C' \subset S$ such that $\varepsilon^C_i \neq \varepsilon^{C'}_i$ and $\cP^C_i = \cP^{C'}_i = 1$, and for any $k \in \{0,1,2\}$, either $C, C' \subset S_k$ or $ C,C' \not\subset S_k$.
    \end{enumerate}
    These should be thought of as follows. If $p_i$ is type (I), then the signs $\varepsilon_i^C$ are independent of the choice of component $C \subseteq S$. 
    Thus we aim to convert all double points to being of type (I). Note that both double points introduced when performing a finger move are of type (I).
    
    If $p_i$ is of type (II), the sign $\varepsilon^C_i$ is the same for all components $C$ which contain a preimage of $p_i$, although it may differ for components $C$ which do not contain a preimage of $p_i$. As such, we may convert all double points of type (II) to type (I) moves of the third type; that is, by flipping signs of some integers $\varepsilon^C_i$ with $\cP^C_i = 0$. Thus we henceforth assume all double points are of types (I), (III), or (IV), and that there are no double points of type (II).
    
    If $p_i$ is of type (III), the signs $\varepsilon^C_i$ do depend on the component $C$, but whenever there are two components $C, C' \subset S$ with $\varepsilon^C_i \neq \varepsilon^{C'}_i$, they lie in different choices of $S_0$, $S_1$, or $S_2$. So in particular we are free to perform finger moves between $C$ and $C'$.
    
    If $p_i$ is of type (IV), there are choices of components $C,C' \subset S$ such that $\varepsilon^C_i \neq \varepsilon^{C'}_i$ and we are not free to perform finger moves between $C$ and $C'$.

    The procedure for converting double points of types (III) and (IV) to type (I) is relatively simple, though cumbersome to describe, so we recommend the reader consult the diagrammatic summary in \cref{fig: regular homotopy 1}. This diagram is to be read as follows. Each shaded region corresponds to a component of $S$, possibly with multiple shaded regions corresponding to the same component. Components are approximately aligned in three columns, corresponding from left-to-right to $S_0$, $S_1$, and $S_2$. Within a component $C$, a node with a sign corresponds to a preimage in $C$ of a double point $p_i$. The indicated sign is the sign of $\varepsilon^C_i$. Two nodes are joined by an arc if they correspond to the same double point. Thus double points of type (I) are represented by arcs joining two nodes of the same sign; double points of type (III) are arcs connecting nodes of different signs which cross from one column into another; double points of type (IV) are arcs connecting nodes of different signs which stay within a single column.

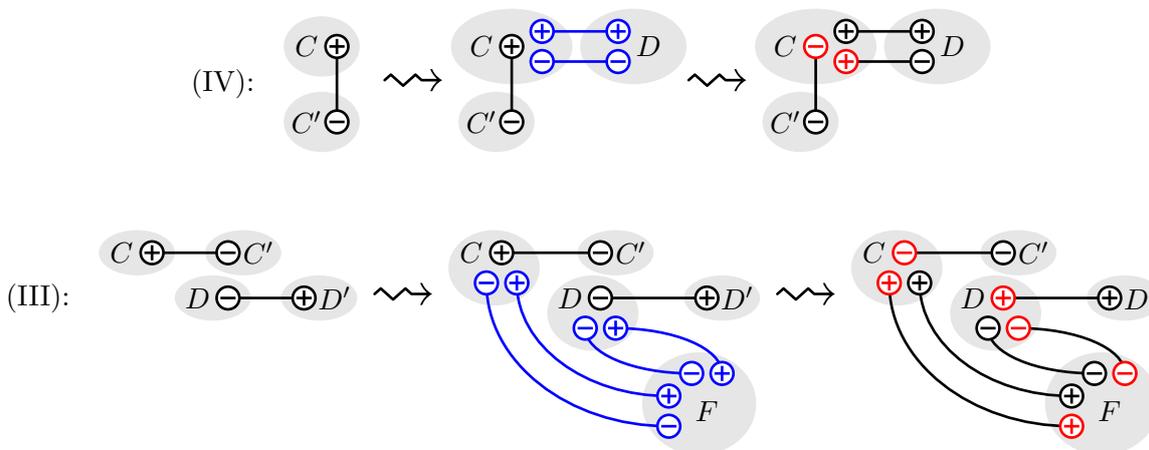
\begin{figure}[h]
    \begin{center}
    \newcommand{\pspt}[3][black]{
    \draw[line width = 1pt, #1, fill=white] (#2,#3) circle(0.15);
    \draw[line width = 1pt, #1] (#2,#3+0.1) -- (#2,#3-0.1);
    \draw[line width = 1pt, #1] (#2+0.1,#3) -- (#2-0.1,#3);
}

\newcommand{\ngpt}[3][black]{
    \draw[line width = 1pt, #1, fill=white] (#2,#3) circle(0.15);
    \draw[line width = 1pt, #1] (#2+0.1,#3) -- (#2-0.1,#3);
}

\begin{tikzpicture}[scale=1]
\node at (-1,0.5) {(IV):};

\fill[gray!20, dashed] (0.3,0) ellipse (0.5 and 0.4);
\fill[gray!20, dashed] (0.3,1) ellipse (0.5 and 0.4);

\fill[gray!20, dashed] (2.8,1) ellipse (0.8 and 0.5);
\fill[gray!20, dashed] (2.6,0) ellipse (0.5 and 0.4);
\fill[gray!20, dashed] (4.4,1) ellipse (0.7 and 0.5);

\fill[gray!20, dashed] (6.8,1) ellipse (0.8 and 0.5);
\fill[gray!20, dashed] (6.6,0) ellipse (0.5 and 0.4);
\fill[gray!20, dashed] (8.4,1) ellipse (0.7 and 0.5);

\draw[line width = 1pt, black] (0.5,1) -- (0.5,0);
\pspt{0.5}{1}
\ngpt{0.5}{0}

\node at (0.1,1) {$C$};
\node at (0.1,0) {$C'$};

\node at (1.5,0.5) {\Huge$\leadsto$};

\draw[line width = 1pt, black] (2.8,1) -- (2.8,0);
\draw[line width = 1pt, blue] (3.2,1.2) -- (4.2,1.2);
\draw[line width = 1pt, blue] (3.2,0.8) -- (4.2,0.8);
\pspt{2.8}{1}
\ngpt{2.8}{0}
\pspt[blue]{3.2}{1.2}
\ngpt[blue]{3.2}{0.8}
\pspt[blue]{4.2}{1.2}
\ngpt[blue]{4.2}{0.8}

\node at (2.4,1) {$C$};
\node at (2.4,0) {$C'$};
\node at (4.6,1) {$D$};

\node at (5.5,0.5) {\Huge$\leadsto$};

\draw[line width = 1pt, black] (6.8,1) -- (6.8,0);
\draw[line width = 1pt, black] (7.2,1.2) -- (8.2,1.2);
\draw[line width = 1pt, black] (7.2,0.8) -- (8.2,0.8);
\ngpt[red]{6.8}{1}
\ngpt{6.8}{0}
\pspt{7.2}{1.2}
\pspt[red]{7.2}{0.8}
\pspt{8.2}{1.2}
\ngpt{8.2}{0.8}

\node at (6.4,1) {$C$};
\node at (6.4,0) {$C'$};
\node at (8.6,1) {$D$};

\end{tikzpicture}

\vspace{1cm}

\begin{tikzpicture}[scale=1]
\node at (-1.5,0.4) {(III):};

\fill[gray!20, dashed] (-0.2,1) ellipse (0.5 and 0.3);
\fill[gray!20, dashed] (1.2,1) ellipse (0.5 and 0.3);
\fill[gray!20, dashed] (0.8,0.4) ellipse (0.5 and 0.3);
\fill[gray!20, dashed] (2.2,0.4) ellipse (0.5 and 0.3);

\fill[gray!20, dashed] (4.5,0.8) ellipse (0.6 and 0.5);
\fill[gray!20, dashed] (6.1,1) ellipse (0.5 and 0.3);
\fill[gray!20, dashed] (5.8,0.2) ellipse (0.6 and 0.5);
\fill[gray!20, dashed] (7.5,0.4) ellipse (0.5 and 0.3);
\fill[gray!20, dashed] (7.2,-1) ellipse (0.75 and 0.68);

\fill[gray!20, dashed] (9.8,0.8) ellipse (0.6 and 0.5);
\fill[gray!20, dashed] (11.4,1) ellipse (0.5 and 0.3);
\fill[gray!20, dashed] (11.1,0.2) ellipse (0.6 and 0.5);
\fill[gray!20, dashed] (12.8,0.4) ellipse (0.5 and 0.3);
\fill[gray!20, dashed] (12.5,-1) ellipse (0.75 and 0.68);

\draw[line width = 1pt, black] (0,1) -- (1,1);
\draw[line width = 1pt, black] (1,0.4) -- (2,0.4);
\pspt{0}{1}
\ngpt{1}{1}
\ngpt{1}{0.4}
\pspt{2}{0.4}
\node at (-0.4,1) {$C$};
\node at (1.4,1) {$C'$};
\node at (0.6,0.4) {$D$};
\node at (2.4,0.4) {$D'$};

\node at (3.3,0.4) {\Huge$\leadsto$};

\draw[line width = 1pt, black] (4.6,1) -- (5.9,1);
\draw[line width = 1pt, black] (5.9,0.4) -- (7.3,0.4);

\draw[line width = 1pt, blue] (5.7,0) arc(180:270:1.4 and 0.6);
\draw[line width = 1pt, blue] (6.1,0) arc(90:0:1.4 and 0.6);

\draw[line width = 1pt, blue] (4.8,0.6) arc(180:270:2 and 1.5);
\draw[line width = 1pt, blue] (4.4,0.6) arc(180:270:2.4 and 1.9);

\pspt{4.6}{1}
\ngpt[blue]{4.4}{0.6}
\pspt[blue]{4.8}{0.6}
\ngpt{5.9}{1}
\ngpt{5.9}{0.4}
\ngpt[blue]{5.7}{0}
\pspt[blue]{6.1}{0}
\pspt{7.3}{0.4}

\ngpt[blue]{7.1}{-0.6}
\pspt[blue]{7.5}{-0.6}
\pspt[blue]{6.8}{-0.9}
\ngpt[blue]{6.8}{-1.3}

\node at (4.2,1) {$C$};
\node at (6.3,1) {$C'$};
\node at (5.5,0.4) {$D$};
\node at (7.7,0.4) {$D'$};
\node at (7.3,-1.1) {$F$};

\node at (8.6,0.4) {\Huge$\leadsto$};

\draw[line width = 1pt, black] (9.9,1) -- (11.2,1);
\draw[line width = 1pt, black] (11.2,0.4) -- (12.6,0.4);

\draw[line width = 1pt, black] (11,0) arc(180:270:1.4 and 0.6);
\draw[line width = 1pt, black] (11.4,0) arc(90:0:1.4 and 0.6);

\draw[line width = 1pt, black] (10.1,0.6) arc(180:270:2 and 1.5);
\draw[line width = 1pt, black] (9.7,0.6) arc(180:270:2.4 and 1.9);

\ngpt[red]{9.9}{1}
\pspt[red]{9.7}{0.6}
\pspt{10.1}{0.6}
\ngpt{11.2}{1}
\pspt[red]{11.2}{0.4}
\ngpt{11}{0}
\ngpt[red]{11.4}{0}
\pspt{12.6}{0.4}

\ngpt{12.4}{-0.6}
\ngpt[red]{12.8}{-0.6}
\pspt{12.1}{-0.9}
\pspt[red]{12.1}{-1.3}

\node at (9.5,1) {$C$};
\node at (11.6,1) {$C'$};
\node at (10.8,0.4) {$D$};
\node at (13,0.4) {$D'$};
\node at (12.6,-1.1) {$F$};

\end{tikzpicture}
    \caption{A diagrammatic representation of how to convert all double points into type (I). Top: example converting from a double point of type (IV) to one of type (III), as well as two of type (I). Bottom: example converting from a pair of double points of type (III) to six of type (I).}
    \label{fig: regular homotopy 1}
    \end{center}
\end{figure}

    The first two moves outlined above are shown on in the diagram as follows. Performing a finger move looks like adding a pair of arcs between two components in different columns, and setting both nodes on one arc to positive and both nodes on the other arc to negative. These new arcs must cross horizontally from one column into another. When we perform this move, the introduced arcs are marked in blue in the diagram. 
    The second move looks like swapping two different signs of nodes in a component. When we perform this move, the swapped signs are marked in red in the diagram. The third move cannot be represented in the diagram, but we will also not require it again, because we only use it to implicitly convert all double points of type (II) into type (I).

    We first describe how to arrange that there are no double points of type (IV); see the top of \cref{fig: regular homotopy 1}.
    Suppose $p_i$ is of type (IV), and let $C, C' \subset S$ be components such that $\varepsilon^C_i \neq \varepsilon^{C'}_i$ and $\cP_i^C = \cP_i^{C'} =1$. Without loss of generality, suppose that $C,C' \subset S_0$. By performing a finger move if necessary, we may assume there is some double point $p_j$ and a component $D \subset S_1$ such that $\cP^C_j = \cP^D_j = 1$ and $\varepsilon^C_j = -\varepsilon^C_i$. By flipping the signs of $\varepsilon^C_i$ and $\varepsilon^C_j$, we may arrange that $p_i$ is of type (I), at the expense of making $p_j$ of type (III). Hence we may assume there are no double points of type (IV).

    We now describe how to arrange that there are no double points of type (III), assuming that there are none of type (IV); see the bottom of \cref{fig: regular homotopy 1}.
    Suppose that $p_i$ is of type (III), and let $C, C' \subset S$ be components such that $\varepsilon^C_i \neq \varepsilon^{C'}_i$ and $\cP_i^C = \cP_i^{C'} =1 $. Without loss of generality, suppose that $C \subseteq S_0$ and $C' \subseteq S_1$. Consider the sum
    \begin{equation*}
        e(g(S),s^Z) = \sum_{i=1}^n \sum_{C \subseteq S} \cP^C_i\varepsilon^C_i,
    \end{equation*}
    which was arranged to be true earlier by the choices of signs $\varepsilon_i^C$.
    Any double point of type (I) contributes $\pm 2$ to this sum, while a double point of type (III) contributes 0. So if there are $m$ double points of type (III),
    \begin{equation*}
        2n \equiv e(g(S),s^Z) \equiv 2(n-m) \mod{4},
    \end{equation*}
    where we recall that the first equivalence is by assumption. Therefore $m$ must be even. In particular, there must exist another double point $p_j$ of type (III), say with $\cP^D_j = \cP^{D'}_j = 1$ and $\varepsilon^D_j = -\varepsilon^C_i$, for two components $D, D' \subset S$. Find some $k \in \{0,1,2\}$ such that $D \subseteq S_k$. Then we can find $\ell \in \{1,2\}$, $\ell \neq k$, and a component $F \subseteq S_\ell$. By performing finger moves between $C$ and $F$ and between $D$ and $F$ if necessary, we can assume there exist double points $p_s$, $p_t$ of type (I) such that 
    \begin{equation*}
    \cP^C_s, \cP^F_s, \cP^D_t, \cP^F_t = 1, \text{ and }\varepsilon^C_s = \varepsilon^F_s = -\varepsilon^D_t=-\varepsilon^F_t = -\varepsilon^C_i.
    \end{equation*}
    By flipping the signs of $\varepsilon^C_s$, $\varepsilon^F_s$, $\varepsilon^D_t$, $\varepsilon^F_t$, $\varepsilon^C_i$, and $\varepsilon^D_j$, we convert both $p_i$ and $p_j$ to type (I), while keeping $p_s$ and $p_t$ as type (I). Hence we can assume there are no double points of type (III).

    Hence all double points may be assumed to be of type (I), and we are done by \cref{lmm: combinatorial condition} as described above. 
\end{proof}

If $S = S_0 \sqcup S_1 \sqcup S_2$ with each $S_i$ non-empty, then the regular homotopy from $f$ to $g$ described in the proof of \cref{lmm: homotopy puncture} restricts to an isotopy of $\restr{f}{S_i}$ for each $i=0,1,2$. However, if we can only assume that $S = S_0 \sqcup S_1$, such as when $S$ is the union of two properly embedded connected surfaces which may intersect, 
the conversion from double points of type (III) to type (I) or (II) outlined above requires self-finger moves of either $S_0$ or $S_1$. In fact, an extra condition is required to be able to make $Z$ almost-extendable over $\S$ after just isotopies of $\restr{f}{S_0}$ and $\restr{f}{S_1}$. 

\begin{proposition}\label{prop: original finger puncture}
    Let $Z \subset \del X$ be a spanning surface for $\del \S$ with associated Seifert section $s^Z$. Suppose that $S = S_0 \sqcup S_1$ is a disjoint union of two non-empty compact surfaces. Write $\S_i = f(S_i)$ and $T= \self(\S_0) + \self(\S_1)$.  
    Then the following are equivalent.
    \begin{enumerate}[label=(\roman*)]
        \item There exists an immersion $g\colon S \immerse X$ regularly homotopic to $f$ such that the homotopy restricts to an isotopy on both $S_0$ and $S_1$, and $Z$ is almost-extendable over $g(S)$.
        \item The class $[Z \cup \S] = 0\in H_2(X;\cyc{2})$ and
        \begin{equation*}
            e(\S_0,s^Z)-e(\S_1,s^Z) \in \{-2T, -2T + 4,\ldots,2T\}.
        \end{equation*}
    \end{enumerate}
\end{proposition}
\begin{proof}
    There are two additional things to prove in order to adapt the proof of \cref{lmm: homotopy puncture}: first, that condition $e(\S_0,s^Z)-e(\S_1,s^Z) \in \{-2T, -2T + 4,\ldots,2T\}$ is necessary for (i) to hold; and second, that it is enough to provide an alternative method of converting double points of type (III) into double points of type (I).

    We first show the condition is necessary, and hence that (i) $\Rightarrow$ (ii). By \cref{lmm: even euler number facts}, we have that $2e(\S_1,s^Z) \equiv 2[\S_0] \cdot [\S_1] \bmod{4}$. Subtracting this from $e(\S,s^Z) \equiv 2\self(\S) \bmod{4}$, which is necessary by \cref{lmm: homotopy puncture}, gives the necessary condition
    \begin{equation*}
         e(\S_0,s^Z) - e(\S_1,s^Z) \equiv 2\self(\S) - 2[\S_0] \cdot [\S_1] \equiv  2T \mod{4}.
    \end{equation*}
    
    It remains to show that (i) implies that $\abs{ e(\S_0,s^Z) - e(\S_1,s^Z)} \leq 2T$. Suppose a suitable immersion $g \colon S \immerse X$ and values $\varepsilon_i \in \{ \pm 1\}$ have been found, so that for each component $C \subseteq S$, 
    \begin{equation*}
        e(g(C),s^Z) = \sum_{i=1}^n \cP^C_i \varepsilon_i.
    \end{equation*}
    Then since $g$ is regularly homotopic to $f$,
    \begin{align*}
        e(\S_0,s^Z) - e(\S_1,s^Z)&= e(g(S_0),s^Z) - e(g(S_1),s^Z) \\
        &= \sum_{i=1}^n \left( \sum_{C \subseteq S_0} \cP^C_i\varepsilon^i -  \sum_{C \subseteq S_1} \cP^C_i\varepsilon^i \right)\\
        &= \sum_{\substack{1 \leq i \leq n, \\g\inv(p_i) \subset S_0}} 2\varepsilon_i - \sum_{\substack{1 \leq i \leq n, \\g\inv(p_i) \subset S_1}} 2\varepsilon_i,
    \end{align*}
    where the last equality follows because double points $p_i$ with preimages in both $S_0$ and $S_1$ contribute 0 to the sum. But since the homotopy from $f$ to $g$ is assumed to be an isotopy on each $S_i$, this gives
    \begin{equation*}
e(\S_0,s^Z) - e(\S_1,s^Z) = \sum_{\substack{1 \leq i \leq n, \\f\inv(p_i) \subset S_0}} 2\varepsilon_i - \sum_{\substack{1 \leq i \leq n, \\f\inv(p_i) \subset S_1}} 2\varepsilon_i.
    \end{equation*}
    The triangle inequality then gives that
    \begin{equation*}
        \abs{e(\S_0,s^Z) - e(\S_1,s^Z)} \leq 2\self(\S_0) + 2\self(\S_1) = 2T.
    \end{equation*}
    Hence (i) $\Rightarrow$ (ii).

    We now provide an alternative method of removing double points of type (III), under the assumption that there are no double points of type (II) or (IV). This will prove that (ii) $\Rightarrow$ (i). Again, the procedure is best described diagrammatically as in \cref{fig: finger puncture}. This diagram is to be read in the same way as \cref{fig: regular homotopy 1}; note, however, that the grey regions corresponding to the different components of $S$ are now only arranged into two columns, corresponding to $S_0$ and $S_1$ respectively.

    \begin{figure}[t]
    \vskip 0.0in
    \begin{center}
    \newcommand{\pspt}[3][black]{
    \draw[line width = 1pt, #1, fill=white] (#2,#3) circle(0.15);
    \draw[line width = 1pt, #1] (#2,#3+0.1) -- (#2,#3-0.1);
    \draw[line width = 1pt, #1] (#2+0.1,#3) -- (#2-0.1,#3);
}

\newcommand{\ngpt}[3][black]{
    \draw[line width = 1pt, #1, fill=white] (#2,#3) circle(0.15);
    \draw[line width = 1pt, #1] (#2+0.1,#3) -- (#2-0.1,#3);
}

\begin{tikzpicture}[scale=1]

\node at (-1.1,0.7) {(a)};

\fill[gray!20, dashed] (0.2,1.1) ellipse (0.5 and 0.3);
\fill[gray!20, dashed] (1.6,1.1) ellipse (0.5 and 0.3);
\fill[gray!20, dashed] (0.2,0.3) ellipse (0.5 and 0.3);
\fill[gray!20, dashed] (1.6,0.3) ellipse (0.5 and 0.3);

\fill[gray!20, dashed] (3.8,1.4) ellipse (0.5 and 0.3);
\fill[gray!20, dashed] (5.5,1.2) ellipse (0.6 and 0.5);
\fill[gray!20, dashed] (3.9,0.2) ellipse (0.6 and 0.5);
\fill[gray!20, dashed] (5.6,0) ellipse (0.5 and 0.3);

\fill[gray!20, dashed] (7.8,1.4) ellipse (0.5 and 0.3);
\fill[gray!20, dashed] (9.5,1.2) ellipse (0.6 and 0.5);
\fill[gray!20, dashed] (7.9,0.2) ellipse (0.6 and 0.5);
\fill[gray!20, dashed] (9.6,0) ellipse (0.5 and 0.3);

\draw[line width = 1pt, black] (0.4,1.1) -- (1.4,1.1);
\draw[line width = 1pt, black] (0.4,0.3) -- (1.4,0.3);

\pspt{0.4}{1.1}
\ngpt{1.4}{1.1}
\ngpt{0.4}{0.3}
\pspt{1.4}{0.3}

\node at (0,1.1) {$C$};
\node at (1.8,1.1) {$C'$};
\node at (0,0.3) {$D$};
\node at (1.8,0.3) {$D'$};

\node at (2.7,0.7) {\Huge$\leadsto$};

\draw[line width = 1pt, black] (4,1.4) -- (5.4,1.4);
\draw[line width = 1pt, black] (4,0) -- (5.4,0);
\draw[line width = 1pt, blue] (5.2,1) arc(90:180:1.4 and 0.6);
\draw[line width = 1pt, blue] (5.6,1) arc(0:-90:1.4 and 0.6);

\pspt{4}{1.4}
\ngpt{5.4}{1.4}
\pspt[blue]{5.2}{1}
\ngpt[blue]{5.6}{1}
\ngpt{4}{0}
\pspt[blue]{3.8}{0.4}
\ngpt[blue]{4.2}{0.4}
\pspt{5.4}{0}

\node at (3.6,1.4) {$C$};
\node at (5.8,1.4) {$C'$};
\node at (3.6,0) {$D$};
\node at (5.8,0) {$D'$};

\node at (6.7,0.7) {\Huge$\leadsto$};

\draw[line width = 1pt, black] (8,1.4) -- (9.4,1.4);
\draw[line width = 1pt, black] (8,0) -- (9.4,0);
\draw[line width = 1pt, black] (9.2,1) arc(90:180:1.4 and 0.6);
\draw[line width = 1pt, black] (9.6,1) arc(0:-90:1.4 and 0.6);

\pspt{8}{1.4}
\pspt[red]{9.4}{1.4}
\ngpt[red]{9.2}{1}
\ngpt{9.6}{1}
\pspt[red]{8}{0}
\ngpt[red]{7.8}{0.4}
\ngpt{8.2}{0.4}
\pspt{9.4}{0}

\node at (7.6,1.4) {$C$};
\node at (9.8,1.4) {$C'$};
\node at (7.6,0) {$D$};
\node at (9.8,0) {$D'$};

\end{tikzpicture}

\vspace{1cm}

\begin{tikzpicture}[scale=1]

\node at (-5.3,0) {(b)};

\fill[gray!20, dashed] (-3.8,1) ellipse (0.3 and 0.5);
\fill[gray!20, dashed] (-2.6,1) ellipse (0.3 and 0.5);
\fill[gray!20, dashed] (-4,0) ellipse (0.5 and 0.3);
\fill[gray!20, dashed] (-2.4,0) ellipse (0.5 and 0.3);
\fill[gray!20, dashed] (-4,-0.8) ellipse (0.5 and 0.3);
\fill[gray!20, dashed] (-4,-1.8) ellipse (0.5 and 0.3);

\fill[gray!20, dashed] (0,1) ellipse (0.3 and 0.5);
\fill[gray!20, dashed] (1.4,0.9) ellipse (0.5 and 0.6);
\fill[gray!20, dashed] (-0.2,0) ellipse (0.5 and 0.3);
\fill[gray!20, dashed] (1.3,-0.2) ellipse (0.6 and 0.5);
\fill[gray!20, dashed] (0.1,-1.2) ellipse (0.65 and 0.47);
\fill[gray!20, dashed] (0.1,-2.2) ellipse (0.65 and 0.47);

\fill[gray!20, dashed] (4.9,1) ellipse (0.3 and 0.5);
\fill[gray!20, dashed] (6.3,0.9) ellipse (0.5 and 0.6);
\fill[gray!20, dashed] (4.7,0) ellipse (0.5 and 0.3);
\fill[gray!20, dashed] (6.2,-0.2) ellipse (0.6 and 0.5);
\fill[gray!20, dashed] (5,-1.2) ellipse (0.65 and 0.47);
\fill[gray!20, dashed] (5,-2.2) ellipse (0.65 and 0.47);

\draw[line width = 1pt, black] (-3.8,0.8) -- (-2.6,0.8);
\draw[line width = 1pt, black] (-3.8,0) -- (-2.6,0);
\draw[line width = 1pt, black] (-3.8,-0.8) -- (-3.8,-1.8);

\pspt{-3.8}{0.8}
\ngpt{-2.6}{0.8}
\pspt{-3.8}{0}
\ngpt{-2.6}{0}
\ngpt{-3.8}{-0.8}
\ngpt{-3.8}{-1.8}

\node at (-3.8,1.2) {$C$};
\node at (-2.6,1.2) {$C'$};
\node at (-4.2,0) {$D$};
\node at (-2.2,0) {$D'$};
\node at (-4.2,-0.8) {$G$};
\node at (-4.2,-1.8) {$F$};

\node at (-1.3,0) {\Huge$\leadsto$};

\draw[line width = 1pt, black] (0,0.8) -- (1.2,0.8);
\draw[line width = 1pt, black] (0,0) -- (1.2,0);
\draw[line width = 1pt, black] (0,-1.2) -- (0,-2.2);

\draw[line width = 1pt, blue] (1,-0.4) arc(0:-90:0.6 and 0.6);
\draw[line width = 1pt, blue] (1.4,-0.4) arc(0:-90:1 and 1);

\draw[line width = 1pt, blue] (1.6, 1) arc(90:0:1.1) arc(0:-90:2.3); 
\draw[line width = 1pt, blue] (1.6, 0.6) arc(90:0:0.7) arc(0:-90:1.9); 

\pspt{0}{0.8}
\ngpt{1.2}{0.8}
\ngpt[blue]{1.6}{1}
\pspt[blue]{1.6}{0.6}
\pspt{0}{0}
\ngpt{1.2}{0}
\pspt[blue]{1}{-0.4}
\ngpt[blue]{1.4}{-0.4}
\ngpt{0}{-1.2}
\pspt[blue]{0.4}{-1}
\ngpt[blue]{0.4}{-1.4}
\ngpt{0}{-2.2}
\pspt[blue]{0.4}{-2}
\ngpt[blue]{0.4}{-2.4}

\node at (0,1.2) {$C$};
\node at (1.2,1.2) {$C'$};
\node at (-0.4,0) {$D$};
\node at (1.6,0) {$D'$};
\node at (-0.4,-1.2) {$G$};
\node at (-0.4,-2.2) {$F$};

\node at (3.5,0) {\Huge$\leadsto$};

\draw[line width = 1pt, black] (4.9,0.8) -- (6.1,0.8);
\draw[line width = 1pt, black] (4.9,0) -- (6.1,0);
\draw[line width = 1pt, black] (4.9,-1.2) -- (4.9,-2.2);

\draw[line width = 1pt, black] (5.9,-0.4) arc(0:-90:0.6 and 0.6);
\draw[line width = 1pt, black] (6.3,-0.4) arc(0:-90:1 and 1);

\draw[line width = 1pt, black] (6.5, 1) arc(90:0:1.1) arc(0:-90:2.3); 
\draw[line width = 1pt, black] (6.5, 0.6) arc(90:0:0.7) arc(0:-90:1.9); 

\pspt{4.9}{0.8}
\pspt[red]{6.1}{0.8}
\ngpt{6.5}{1}
\ngpt[red]{6.5}{0.6}
\pspt{4.9}{0}
\pspt[red]{6.1}{0}
\ngpt[red]{5.9}{-0.4}
\ngpt{6.3}{-0.4}
\pspt[red]{4.9}{-1.2}
\ngpt[red]{5.3}{-1}
\ngpt{5.3}{-1.4}
\pspt[red]{4.9}{-2.2}
\ngpt[red]{5.3}{-2}
\ngpt{5.3}{-2.4}

\node at (4.9,1.2) {$C$};
\node at (6.1,1.2) {$C'$};
\node at (4.5,0) {$D$};
\node at (6.5,0) {$D'$};
\node at (4.5,-1.2) {$G$};
\node at (4.5,-2.2) {$F$};

\end{tikzpicture}

\vspace{1cm}

\begin{tikzpicture}[scale=1]

\node at (-5.3,0) {(c)};

\fill[gray!20, dashed] (-3.8,1) ellipse (0.3 and 0.5);
\fill[gray!20, dashed] (-2.6,1) ellipse (0.3 and 0.5);
\fill[gray!20, dashed] (-4,0) ellipse (0.5 and 0.3);
\fill[gray!20, dashed] (-2.4,0) ellipse (0.5 and 0.3);
\fill[gray!20, dashed] (-2.4,-0.8) ellipse (0.5 and 0.3);
\fill[gray!20, dashed] (-2.4,-1.8) ellipse (0.5 and 0.3);

\fill[gray!20, dashed] (2.3,1) ellipse (0.3 and 0.5);
\fill[gray!20, dashed] (0.9,0.9) ellipse (0.5 and 0.6);
\fill[gray!20, dashed] (2.5,0) ellipse (0.5 and 0.3);
\fill[gray!20, dashed] (1,-0.2) ellipse (0.6 and 0.5);
\fill[gray!20, dashed] (2.2,-1.2) ellipse (0.65 and 0.47);
\fill[gray!20, dashed] (2.2,-2.2) ellipse (0.65 and 0.47);

\fill[gray!20, dashed] (7.2,1) ellipse (0.3 and 0.5);
\fill[gray!20, dashed] (5.8,0.9) ellipse (0.5 and 0.6);
\fill[gray!20, dashed] (7.4,0) ellipse (0.5 and 0.3);
\fill[gray!20, dashed] (5.9,-0.2) ellipse (0.6 and 0.5);
\fill[gray!20, dashed] (7.1,-1.2) ellipse (0.65 and 0.47);
\fill[gray!20, dashed] (7.1,-2.2) ellipse (0.65 and 0.47);

\draw[line width = 1pt, black] (-3.8,0.8) -- (-2.6,0.8);
\draw[line width = 1pt, black] (-3.8,0) -- (-2.6,0);
\draw[line width = 1pt, black] (-2.6,-0.8) -- (-2.6,-1.8);

\pspt{-3.8}{0.8}
\ngpt{-2.6}{0.8}
\pspt{-3.8}{0}
\ngpt{-2.6}{0}
\pspt{-2.6}{-0.8}
\pspt{-2.6}{-1.8}

\node at (-3.8,1.2) {$C$};
\node at (-2.6,1.2) {$C'$};
\node at (-4.2,0) {$D$};
\node at (-2.2,0) {$D'$};
\node at (-2.2,-0.8) {$G$};
\node at (-2.2,-1.8) {$F$};

\node at (-1.1,0) {\Huge$\leadsto$};

\draw[line width = 1pt, black] (1.1,0.8) -- (2.3,0.8);
\draw[line width = 1pt, black] (1.1,0) -- (2.3,0);
\draw[line width = 1pt, black] (2.3,-1.2) -- (2.3,-2.2);

\draw[line width = 1pt, blue] (1.3,-0.4) arc(180:270:0.6 and 0.6);
\draw[line width = 1pt, blue] (0.9,-0.4) arc(180:270:1 and 1);

\draw[line width = 1pt, blue] (0.7, 1) arc(90:180:1.1) arc(180:270:2.3); 
\draw[line width = 1pt, blue] (0.7, 0.6) arc(90:180:0.7) arc(180:270:1.9); 

\pspt{1.1}{0.8}
\ngpt[blue]{0.7}{1}
\pspt[blue]{0.7}{0.6}
\ngpt{2.3}{0.8}
\pspt{1.1}{0}
\pspt[blue]{0.9}{-0.4}
\ngpt[blue]{1.3}{-0.4}
\ngpt{2.3}{0}
\pspt{2.3}{-1.2}
\pspt[blue]{1.9}{-1}
\ngpt[blue]{1.9}{-1.4}
\pspt{2.3}{-2.2}
\pspt[blue]{1.9}{-2}
\ngpt[blue]{1.9}{-2.4}

\node at (1.1,1.2) {$C$};
\node at (2.3,1.2) {$C'$};
\node at (0.7,0) {$D$};
\node at (2.7,0) {$D'$};
\node at (2.7,-1.2) {$G$};
\node at (2.7,-2.2) {$F$};

\node at (3.8,0) {\Huge$\leadsto$};

\draw[line width = 1pt, black] (6,0.8) -- (7.2,0.8);
\draw[line width = 1pt, black] (6,0) -- (7.2,0);
\draw[line width = 1pt, black] (7.2,-1.2) -- (7.2,-2.2);

\draw[line width = 1pt, black] (6.2,-0.4) arc(180:270:0.6 and 0.6);
\draw[line width = 1pt, black] (5.8,-0.4) arc(180:270:1 and 1);

\draw[line width = 1pt, black] (5.6, 1) arc(90:180:1.1) arc(180:270:2.3); 
\draw[line width = 1pt, black] (5.6, 0.6) arc(90:180:0.7) arc(180:270:1.9); 

\ngpt[red]{6}{0.8}
\pspt[red]{5.6}{1}
\pspt{5.6}{0.6}
\ngpt{7.2}{0.8}
\ngpt[red]{6}{0}
\pspt{5.8}{-0.4}
\pspt[red]{6.2}{-0.4}
\ngpt{7.2}{0}
\ngpt[red]{7.2}{-1.2}
\pspt{6.8}{-1}
\pspt[red]{6.8}{-1.4}
\ngpt[red]{7.2}{-2.2}
\pspt{6.8}{-2}
\pspt[red]{6.8}{-2.4}

\node at (6,1.2) {$C$};
\node at (7.2,1.2) {$C'$};
\node at (5.6,0) {$D$};
\node at (7.6,0) {$D'$};
\node at (7.6,-1.2) {$G$};
\node at (7.6,-2.2) {$F$};

\end{tikzpicture}
    \caption{A revised schema for converting double points of type (III) to type (I). The three rows (a)--(c) refer to different possible configurations of double points of type (III).}
    \label{fig: finger puncture}
    \end{center}
\end{figure}
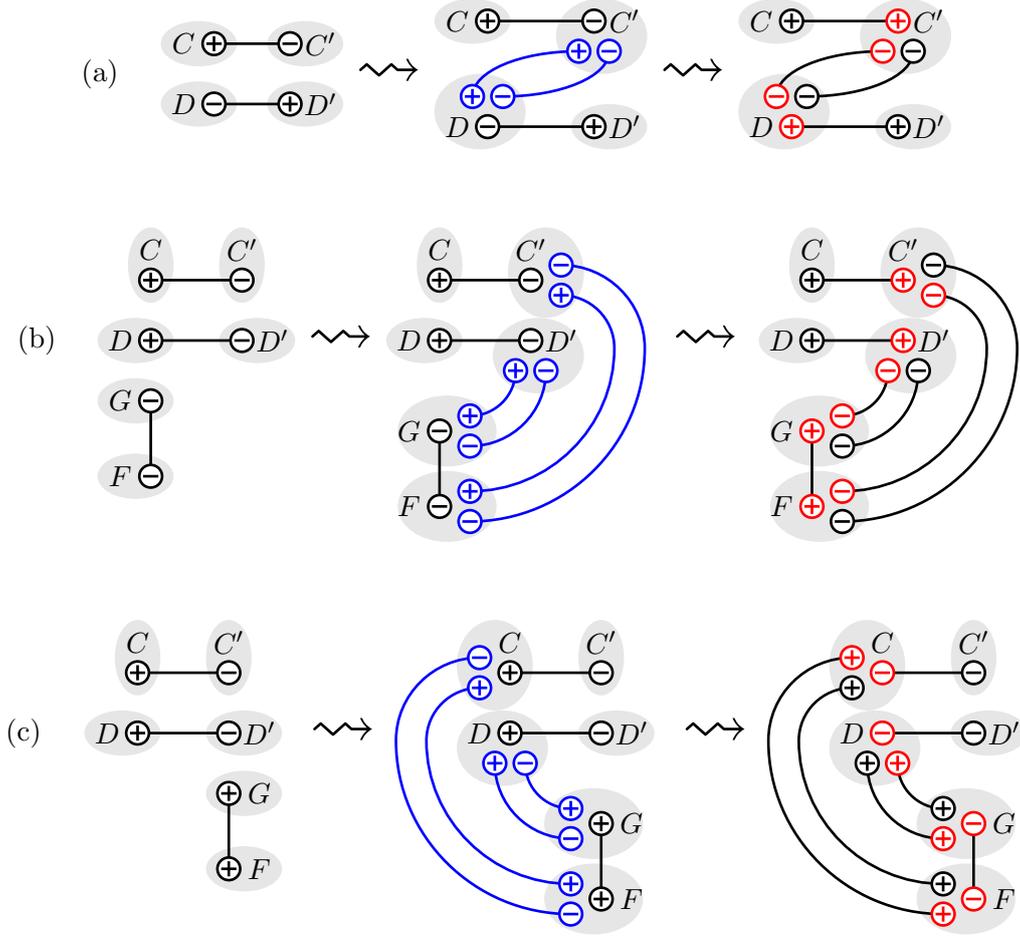

    Suppose that $p_i$ is of type (III), with $\cP^C_i = \cP^{C'}_i = 1$ and $\varepsilon^C_i =-\varepsilon^{C'}_i$. Without loss of generality, suppose that $C \subseteq S_0$ and $C' \subseteq S_1$. As before, there must exist another double point $p_j$ of type (III), say with $\cP^D_j = \cP^{D'}_j = 1$. We can order $D$ and $D'$ such that $D \subseteq S_0$ and $D' \subseteq S_1$. We have three cases to consider, each shown in \cref{fig: finger puncture}.
    \pagebreak

    \begin{enumerate}[label=(\alph*)]
        \item Suppose we can choose $p_j$ such that $\varepsilon^C_i =- \varepsilon^D_j$. Then we are in the situation of \cref{fig: finger puncture}(a). By performing a finger move between $C'$ and $D$ if necessary, we may assume that there is a double point $p_\ell$ such that
        \begin{equation*}
            \cP^{C'}_\ell = \cP^D_\ell = 1 \text{ and } \varepsilon^{C'}_\ell = - \varepsilon^{C'}_i = - \varepsilon^D_i = \varepsilon^D_\ell.
        \end{equation*}
        Then by flipping the signs of $\varepsilon^{C'}_\ell$, $\varepsilon^{C'}_i$, $\varepsilon^D_j$, and  $\varepsilon^D_\ell$, we convert both $p_i$ and $p_j$ to type (I), while keeping $p_\ell$ type (I).
        \item Suppose we are not in case (a), but that there is a double point $p_\ell$ of $g(S_0)$ and a component $F\subseteq S_0$ with $\cP^F_\ell \geq 1$ and $\varepsilon^F_\ell  = -\varepsilon^C_i$. Let $G \subseteq S_0$ be the other component of $S_0$ such that $\cP^G_\ell \geq 1$, noting that possibly $F=G$. Then we are in the situation of \cref{fig: finger puncture}(b). Since $p_\ell$ cannot be of type (IV) by assumption, or type (III) since it is a point of intersection between components in $S_0$, it must be of type (I). So $\varepsilon^F_\ell = \varepsilon^G_\ell$.
        
        By performing finger moves if necessary, we may assume there are double points $p_s$, $p_t$ of type (I) such that
        \begin{equation*}
            \cP^F_s = \cP^{C'}_s = \cP^G_t = \cP^{D'}_t = 1 \text{ and } \varepsilon^F_s = \varepsilon^{C'}_s = \varepsilon^G_t = \varepsilon^{D'}_t = -\varepsilon^F_\ell.
        \end{equation*}
        By flipping the signs of $\varepsilon^{C'}_i$, $\varepsilon^{C'}_s$, $\varepsilon^{D'}_j$, $\varepsilon^{D'}_t$, $\varepsilon^F_\ell$, $\varepsilon^F_s$, $\varepsilon^G_\ell$, and $\varepsilon^G_t$, we convert both $p_i$ and $p_j$ to type (I), while keeping each of $p_\ell$, $p_s$ and $p_t$ as type (I).

        Note that if $F=G$, flipping the signs of $\varepsilon^F_\ell=\varepsilon^G_\ell$ and $\varepsilon^F_s = \varepsilon^G_s$ is slightly different to the move described in \cref{lmm: homotopy puncture}, since $\cP^F_\ell = 2$. However, it is still a legal move, since it still does not affect $\sum_i \cP^F_i \varepsilon^F_i$.
        \item Suppose that we are not in case (a) or (b). Then consider the sum
    \begin{equation*}
        e(\S_0,s^Z) - e(\S_1, s^Z) = \sum_{k=1}^n \left( \sum_{C \subseteq S_0}  \cP^C_k  \varepsilon^C_k - \sum_{C \subseteq S_1}  \cP^C_k  \varepsilon^C_k \right).
    \end{equation*}
    Since we are not in case (a), we may assume that all double points of type (III) contribute $2\varepsilon^C_i$ to the sum. A double point of type (I) contributes 0 to the sum if it is not a double point of $g(S_0)$ or $g(S_1)$. If it is a double point of $g(S_0)$, it must contribute $2\varepsilon^C_i$ to the sum, since we are not in case (b). If all double points of $g(S_1)$ contributed $2\varepsilon^C_i$ to the sum, then we would have that
    \begin{align*}
        \abs{ e(\S_0,s^Z) - e(\S_1, s^Z)} &= 2 \cdot \#\{\text{double points of type (III)}\} + 2\self(g(S_0)) + 2\self(g(S_1))\\
        & \geq 4 + 2T,
     \end{align*}
    which contradicts the assumption that $\abs{ e(\S_0,s^Z) - e(\S_1, s^Z)} \leq 2T$. So there must be a double point $p_\ell$ of $g(S_1)$ and a component $F \subseteq S_1$ with $\cP^F_\ell \geq 1$ and $\varepsilon^F_\ell = \varepsilon^C_i$. This case proceeds exactly as case (b), but reversing the roles of $S_0$ and $S_1$. See \cref{fig: finger puncture}(c).
    \end{enumerate}

    We can therefore arrange that there are no double points of type (III). Hence all double points are of type (I), and we are done.
\end{proof}

We can then deduce \cref{cor: puncturing embeddings} from \cref{prop: original finger puncture}, since $\self(\S_0)=\self(\S_1)=0$ when $\S_0$ and $\S_1$ are embedded in $X$. We restate the conclusion for the readers convenience.

\begin{namedtheorem}[\cref{cor: puncturing embeddings}]
    Let $X$ be an orientable 4-manifold, and let $\S_0,\S_1 \subset X$ be two properly embedded compact surfaces. Let $Z \subset \del X$ be a spanning surface for $\del \S_0 \cup \del \S_1$ with associated Seifert section $s^Z$. Then there are surfaces $\S'_0$ and $\S'_1$ isotopic to $\S_0$ and $\S_1$ respectively such that $Z$ is almost-extendable over $\S'_0 \cup \S'_1$ if and only if $[\S_0 \cup Z \cup \S_1] = 0 \in H_2(X;\cyc{2})$ and $e(\S_0,s^Z) = e(\S_1,s^Z)$. 
\end{namedtheorem}

 \section{Cobordisms of properly embedded surfaces}\label{sec: cobordisms}

In this section we will prove our main results on cobordism.
In \cref{sec: oriented cobordisms}, we prove the oriented Theorems~\ref{thm: oriented cobordism} and \ref{thm: oriented general cobordism} using the techniques developed in \cref{sec: cs homotopy}. These are very similar to classical results in \cite{ThomQuelqueProprietes}, though relax the conditions that the ambient manifold needs to be closed. We prove Theorem~\ref{mainthm: general cobordism}, Theorem~\ref{mainthm: cobordism} and Corollary~\ref{maincor: cobordism rel boundary} in Sections~\ref{sec: unoriented cobordisms}, \ref{sec: theorem E}, and \ref{sec: cobordisms rel boundary} respectively. Throughout, we write $\pr_X \colon X \times I \to X$ for the projection map onto the ambient manifold $X$.

\subsection{Existence of oriented cobordisms}\label{sec: oriented cobordisms}

We can use \cref{prop: oriented codimension 2} to directly construct cobordisms between $\Z$-homologous oriented properly embedded codimension 2 submanifolds. This applies in all dimensions.

Recall that we say a cobordism $Z \subset \del X \times I$ from $\del \S_0$ to $\del \S_1$ extends to a cobordism $Y \subset X \times I$ from $\S_0$ to $\S_1$ if $Y \cap (\del X \times I) = Z$. If $Y$ and $Z$ are oriented cobordisms, we require that $Z$ has the orientation restricted from $\del Y$. We first prove \cref{thm: oriented general cobordism}, since it is the most different to results already found in the literature.

\begin{namedtheorem}[\cref{thm: oriented general cobordism}]
    Let $X$ be a connected oriented $(n+2)$-manifold and let $\S_0, \S_1 \subset X$ be oriented properly embedded compact $n$-manifolds. 
    Let $Z \subset \del X \times I$ be an oriented cobordism from $\del \S_0$ to $\del \S_1$. Then $Z$ extends to an oriented cobordism from $\S_0$ to $\S_1$ if and only if $$[-\S_0 \cup -\pr_{X}(Z) \cup \S_1] = 0 \in H_n(X;\Z).$$
\end{namedtheorem}

\begin{proof}
    Write $B \coloneq -\S_0 \stimes \{0\} \cup -Z \cup \S_1 \stimes \{1\} \subset \del (X \stimes I).$
    
    The forward direction is clear, since if $B$ is the boundary of some oriented cobordism in $X \times I$, then $[B] = 0 \in H_n(X \times I;\Z)$, which is equivalent to $[\pr_X(B)] = 0 \in H_n(X;\Z)$.

    To show the reverse, suppose that $[B] = 0 \in H_n(X \times I;\Z)$. By considering the exact sequence of the pair $(X \times I, \del (X \times I) )$, there exists some $\beta \in H_{n+1}(X \times I, \del(X \times I);\Z)$ such that $\del \beta = [B] \in H_n(\del (X \times I);\Z)$. 
    Since $Z \subset \del X \times I$ and $-\S_0 \stimes \{0\} \cup \S_1 \stimes\{1\} \subset X \times \del I$ both meet $\del X \times \del I$ transversely, we can smooth the corners of  $X \times I$ so that $B$ is smoothly embedded in $\del (X \times I)$. We can then apply \cref{prop: oriented codimension 2} with $M = X \times I$ and $A = \del (X \times I)$ to find an oriented cobordism $Y$ from $\S_0$ to $\S_1$ extending $Z$.
\end{proof}

Similar methods can also be used to prove \cref{thm: oriented cobordism}, where the cobordism $Z$ on the boundary is not prescribed.

\begin{namedtheorem}[\cref{thm: oriented cobordism}]
    Let $X$ be a connected oriented $(n+2)$-manifold and let $\S_0, \S_1 \subset X$ be oriented properly embedded compact $n$-manifolds. Then there is an oriented cobordism from $\S_0$ to $\S_1$ if and only if $[\S_0] = [\S_1] \in H_2(X, \del X;\Z).$
\end{namedtheorem}
\begin{proof}
    The forward direction is clear. The reverse direction follows almost exactly as in \cref{thm: oriented general cobordism}, except that \cref{prop: oriented codimension 2} is applied with with $M = X \times I$, $A = X \times \{0,1\}$, and $B = -\S_0 \stimes \{0\} \cup \S_1 \stimes \{1\}$.
\end{proof}

In the unoriented case, this approach fails even for $n=2$, since we do not have a suitable analogue of \cref{prop: oriented codimension 2} for finding compact 3-manifolds representing $\cyc{2}$-homology classes in 5-manifolds.

\subsection{Proof of Theorem \ref{mainthm: general cobordism}}\label{sec: unoriented cobordisms}
Instead of using the classical approaches of \cref{sec: oriented cobordisms}, we build unoriented cobordisms between surfaces by ``stretching out" spanning 3-manifolds away from the intersection points between the surfaces. More precisely, suppose we have properly embedded compact surfaces $\S_0,\S_1 \subset X$, and a spanning manifold $Z \subset \del X$ for $\del \S_0 \sqcup \del \S_1$.
Suppose that $Z$ is almost-extendable over $\S_0 \cup \S_1$, in the sense of \cref{def: almost-extendable}. That is, write  $\wh\S \subset \wh{X}$ for the image of the punctured embedding of the immersion given by the union of the two embeddings $\S_0\hookrightarrow X$ and $\S_1 \hookrightarrow X$; then we are supposing that we can find a spanning 3-manifold $\wh{Y} \subset \wh{X}$ for $\wh\S$, which meets each component of $\del\wh{X} \setminus \del X$ in an annulus. If this is the case, $\wh{Y}$ can be stretched out into $\wh{X} \times I$, and these annuli stretch out into the trace of an isotopy in $\del\wh{X} \setminus \del X$ from  $\S_0 \cap (\del\wh{X} \setminus \del X)$ to $\S_1\cap(\del\wh{X} \setminus \del X)$. This trace can be ``filled in" to give a cobordism in $X \times I$ from $\S_0$ to $\S_1$ by the following lemma.

\begin{lemma}\label{lmm: local around double point}
    Let $X = D^4$, let $\S_0 = D^2 \times \{(0,0)\}$, and let $\S_1 = \{(0,0)\} \times D^2$. Let $A \subset \del D^4$ be an annulus spanning the Hopf link $\del \S_0 \sqcup \del \S_1$. Fix an identification $\varphi \colon A \to S^1 \times I$ such that $\varphi\inv(S^1 \times \{i\}) = \del \S_i$ for $i=0,1$. Let $\pr_I \colon S^1 \times I \to I$ be the projection map, and let
    \begin{equation*}
        Z \coloneq \{(a,\pr_I\varphi(a)) \mid a \in A \} \subset \del D^4 \times I.
    \end{equation*}
    Then there is a cobordism from $\S_0$ to $\S_1$ extending $Z$.
\end{lemma}
\begin{proof}
    The map $ S^1 \times I \to \del D^4 \times I$ given by $(x,t) \mapsto (\varphi\inv(x,t),t)$ is an isotopy from $\del \S_0$ to $\del\S_1$ with trace $Z$. This can be extended radially to a map $F \colon D^2 \times I \to D^4 \times I$ given by
    \begin{equation*}
        F(\lambda \cdot x,t) = ( \lambda \varphi\inv(x,t),t) \quad \text{for $x \in S^1$, $\lambda \in [0,1]$, $t \in I$.} 
    \end{equation*}
    Then $F$ is an isotopy from $\S_0$ to $\S_1$ which fixes the intersection point at the origin.
    Hence $F(D^2 \times I) \subset D^4 \times I$ is a cobordism from $\S_0$ to $\S_1$ extending $Z$.
\end{proof}

We have now developed all of the necessary technical tools to prove Theorem \ref{mainthm: general cobordism}.

\pagebreak

\begin{namedtheorem}[Theorem \ref{mainthm: general cobordism}]
    Let $X$ be a connected orientable 4-manifold, and let $\S_0, \S_1 \subset X$ be properly embedded compact surfaces. 
    Let $Z \subset \del X \times I$ be a cobordism from $\del \S_0$ to $\del \S_1$.
    Then $Z$ extends to a cobordism from $\S_0$ to $\S_1$ if and only if $[\S_0 \cup \pr_X(Z) \cup \S_1] = 0 \in H_2(X;\cyc{2})$ and $e(\S_1,s_1) = e(\S_0,s_0) + e(Z,s_0\cup s_1)$ after an arbitrary choice of framing $s_i$ of $\del \S_i$ for $i=0,1$.
\end{namedtheorem}
\begin{proof}
    We first prove the forward direction. Let $Y \subset X \times I$ be a cobordism from $\S_0$ to $\S_1$ extending $Z$. The homological condition is necessary, since
    \begin{equation*}
        [\S_0 \cup \pr_{X}(Z) \cup \S_1] = [\pr_X(\del Y)] =  0 \in H_2(X;\cyc{2}).
    \end{equation*}
    Recall that, after fixing an orientation on $X$,
    \begin{equation*}
        \del (X \times I) = - X \stimes \{0\} \cup -\del X \stimes I \cup  X \stimes \{1\}.
    \end{equation*} 
    So \cref{def: euler number} of the relative Euler numbers by counting intersection points with a push-off shows that 
    \begin{equation*}
    e(\del Y) = -e(\S_0, s_0) - e(Z, s_0 \cup s_1) + e(\S_1,s_1),
    \end{equation*}
    for any choices of framings $s_i$ of the link $\del \S_i$.
    But $e(\del Y) = 0$ by \cref{lmm: euler numbers of boundaries}, hence the forward direction follows.

    We now prove the reverse direction. Note that it suffices to consider the case where $\S_0$ and $\S_1$ are both non-empty, since the empty submanifold is cobordant to a trivial 2-sphere, which is non-empty. We may also assume that $\S_0$ and $\S_1$ only intersect transversely. In particular, this means that $\del \S_0$ and $\del \S_1$ are disjoint, and that $\S_0 \cup \S_1$ is the image of a generic proper immersion $f \colon S_0 \sqcup S_1 \immerse X$ of a compact surface such that $f(S_i)= \S_i$ for $i=0,1$.
    
    We first consider the case where $\S_0$ and $\S_1$ are closed and $Z = \varnothing$. We apply \cref{cor: puncturing embeddings}. Let $\wh{\S} \subset \wh{X}$ be the image of the punctured embedding of $f$, and let $D_1,\ldots,D_n \subset X$ be the 4-balls such that $\wh{X} = \bar{X \setminus (D_1 \cup \cdots \cup D_n)}$. \cref{cor: puncturing embeddings} says that, potentially after isotopies of $\S_0$ and $\S_1$, there is a spanning manifold $\wh{Y}$ for $\wh{\S}$ such that $A_j \coloneq \wh{Y} \cap \del D_j$ is an annulus for each $j=1,\ldots,n$. 
    
    Write $\wh{\S}_i \coloneq \S_i \cap \wh{\S}$ for $i=0,1$.
    Fix identifications $\varphi_j \colon A_j \to S^1 \times I$ such that
    \begin{equation*}
        \varphi_j\inv(S^1 \times \{i\}) = A_j \cap \wh{\S}_i\quad \text{for $i=0,1$ and $j=1,\ldots,n$.}
    \end{equation*}   
    Then by Tietze's extension theorem and smooth approximation, there is a map $F \colon \wh{Y} \to I$ such that 
    \begin{equation*}
        F\inv(\{i\}) = \wh{\S}_i \text{ for $i=0,1$} \quad\text{and}\quad\restr{F}{A_j} = \pr_I \varphi_j \colon A_j \to I \text{ for $j=1,\ldots,n$.}
    \end{equation*}
    where $\pr_I \colon S^1 \times I \to I$ is the projection map.
    Let $G \colon \wh{Y} \to \wh{X} \times I$ be the embedding given by $G(y) = (y,F(y))$, so that the image $G(\wh{Y}) \subset \wh{X} \times I$ is a cobordism from $\wh{\S}_0$ to $\wh{\S}_1$. 

    But by applying \cref{lmm: local around double point} to each $D_j$, we can find cobordisms $Y_j \subset D_j \times I$ from $\S_0 \cap D_j$ to $\S_1 \cap D_j$ extending $G(\wh{Y} \cap D_j) = \{(a, \pr_I \varphi_j a) \mid a \in A_j\}$ for each $j=1,\ldots,n$.
    Then
    \begin{equation*}
    Y \coloneq G(\wh{Y}) \cup \bigcup_{j=1}^n Y_j \subset X \times I
    \end{equation*}
    is a cobordism from $\S_0$ to $\S_1$. We have therefore shown the result in the case that $Z = \varnothing$.

    Now consider the general case, where $\del \S_0, \del \S_1 \subset \del X$ and $Z \subset \del X \times I$ may be non-empty. Let
    \begin{equation*}
        \S' \coloneq \S_0 \stimes \{0\} \cup Z \cup \S_1 \stimes \{1\} \subset \del (X \times I).
    \end{equation*}
    We again smooth the corners of $X \times I$ so that $\S'$ is a closed surface smoothly embedded in $\del (X \times I)$. Consider the long exact sequence of the pair $(X \times I, \del (X \times I) )$. Since $[\S'] = 0 \in H_2(X \times I;\cyc{2})$ by assumption, exactness says that there exists a class $\alpha \in H_3(X \times I, \del (X \times I);\cyc{2})$ such that $\del \alpha = [\S] \in H_2(\del (X \times I);\cyc{2} )$. But there is a suspension isomorphism
    \begin{equation*}
        H_2(X, \del X;\cyc{2}) \xrightarrow{\cong} H_3(X \times I, \del (X \times I);\cyc{2})
    \end{equation*}
    given by $[\sigma] \mapsto [\sigma \times I]$, so by \cref{lmm: embedding surfaces in 4manifolds}, there is a properly embedded compact surface $M \subset X$ such that $[\S'] = [\del (M \times I)] \in H_2(\del (X \times I);\cyc{2} )$. Furthermore, $e(\del (M \times I) ) = 0$ by \cref{lmm: euler numbers of boundaries}. Since $e(\S')=0$ by assumption, the closed case done above gives a cobordism $Y' \subset \del (X \times I) \times I$ from $\S'$ to $\del (M \times I)$.

    Finally, let $(C,\varphi)$ be a collar of $\del (X \times I)$, and fix an identification $\theta \colon X \times I \to \bar{(X \times I) \setminus C}$ such that $\theta(x) = \varphi(x,1)$ for $x \in \del (X \times I)$. Then
    \begin{equation*}
        Y \coloneq \varphi(Y') \mathop{\cup}_{\theta(\del (M \times I) )} \theta(M \times I)
    \end{equation*}
    is the required cobordism from $\S_0$ to $\S_1$ extending $Z$.
\end{proof}

\subsection{Proof of Theorem \ref{mainthm: cobordism}}\label{sec: theorem E}

In this subsection we prove Theorem \ref{mainthm: cobordism}. We first state a lemma.

\begin{lemma}\label{lmm: assoc section extends}
    Let $Y$ be a closed orientable 3-manifold and let $L_0, L_1 \subset Y$ be two disjoint embedded closed 1-manifolds. Let $Z \subset Y \times I$ be a cobordism from $L_0$ to $L_1$. Let $\pr_Y \colon Z \to Y$ and $\pr_I \colon Z \to I$ be the restrictions of projection maps from $Y \times I$. Assume that ${\pr_{Y}}$ is an embedding near $\del Z$, and hence determines an associated section $s^Z \colon L_0 \cup L_1 \to SN_{Y}{(L_0 \cup L_1)}$ by taking the normal direction into $\pr_Y(Z)$.
    If ${\pr_{Y}} \colon Z \immerse Y$ is an immersion, then $e(Z,s^Z) = 0$.
\end{lemma}
\begin{proof}
    If $\pr_{Y}$ is an immersion, then it has a normal bundle $N(\pr_{Y})$, giving a quotient map
    \begin{equation*}
         q\colon N_{Y \times I} Z = \frac{\pr_{Y}^* T Y \oplus \pr_I^* T I}{TZ} \to \frac{\pr_{Y}^* T Y}{TZ} = N(\pr_{Y}).
     \end{equation*}
    This is surjective, so in particular of constant rank, giving a short exact sequence
\[\begin{tikzcd}[sep = small]
	0 & K & {N_{Y \times I} Z} & {N(\pr_{Y})} & 0
	\arrow[from=1-1, to=1-2]
	\arrow[from=1-2, to=1-3]
	\arrow["q", from=1-3, to=1-4]
	\arrow[from=1-4, to=1-5]
\end{tikzcd}\]
of vector bundles over $Z$.
Since orientation characters are additive under extension and both $Y$ and $I$ are orientable, $N_{Y \times I}Z$ and $N(\pr_{Y})$ both have the same orientation character as $TZ$. Thus the orientation character of $K$ is trivial, and since $K$ is a line bundle, $K$ is trivial.

After identifying $\del Z$ with its image $L_0 \cup L_1$, the bundle $\restr{N(\pr_{Y})}{\del Z}$ describes the directions normal to both $L_0\cup L_1$ and $\pr_{Y}(Z)$. Thus $\restr{K}{\del Z}$ describes the direction normal to $L_0\cup L_1$ but tangent to $\pr_{Y}(Z)$. Hence $\restr{K}{\del Z}$ is the line subbundle spanned by $s^Z$. Since $\restr{K}{\del Z}$ extends to the trivial line bundle $K$ over $Z$, the section $s^Z$ extends to a non-vanishing section of $N_{Y \times I}Z$ over all of $Z$.
\end{proof}

We now deduce Theorem \ref{mainthm: cobordism}.

\begin{namedtheorem}[Theorem \ref{mainthm: cobordism}]
    Let $X$ be a connected orientable 4-manifold and let $\S_0, \S_1 \subset X$ be properly embedded compact surfaces. Then $\S_0$ and $\S_1$ are cobordant if and only if $[\S_0] = [\S_1] \in H_2(X, \del X;\cyc{2})$ and either $\del X \neq \varnothing$ or $e(\S_0) = e(\S_1)$.
\end{namedtheorem}

\begin{proof}
    If $\del X = \varnothing$, then the statement of Theorem \ref{mainthm: cobordism} is exactly Theorem \ref{mainthm: general cobordism}. So suppose that $\del X \neq\varnothing$. Then the forward direction is clear by the same argument as Theorem \ref{mainthm: general cobordism}.
    
    So we prove the reverse direction, and suppose that $[\S_0] =[\S_1] \in H_2(X, \del X;\cyc{2})$ and $\del X \neq \varnothing$. Let $L \coloneq \del \S_0 \cup \del \S_1$, and consider the long exact sequence of the triple $(X, \del X, L)$. Since $[\S_0]$ and $[\S_1]$ are homologous, by exactness there is some $\alpha \in H_2(\del X, L;\cyc{2})$ whose image in $H_2(X,L;\cyc{2})$ under inclusion is precisely $[\S_0]+[\S_1]$. By applying \cref{thm: relative spanning manifolds} to $L$, there exists some spanning manifold $Z' \subset \del X$ for $L$ such that $[Z'] = \alpha \in H_2(\del X, L;\cyc{2})$. In particular then $[\S_0 \cup Z' \cup \S_1] = 0 \in H_2(X,L;\cyc{2})$. Since $H_2(L;\cyc{2}) = 0$, this implies that $[\S_0 \cup Z' \cup \S_1] = 0 \in H_2(X;\cyc{2})$ by the long exact sequence of the pair $(X,L)$.

    Let $f \colon Z' \to I$ be any map with $f\inv(\{i\}) = \del \S_i$ for $i=0,1$, and let 
    \begin{equation*}
        Z = \big\{(z,f(z) )\mid z \in Z'\big\} \subset \del X \times I.
    \end{equation*}
    Note that $Z$ is a cobordism from $\del \S_0$ to $\del \S_1$, and the projection $\del X \times I \to \del X$ restricts to an embedding on $Z$. Let $s$ be the Seifert section associated to $Z'$. Then by \cref{lmm: assoc section extends}, $e(Z,s) = 0$. But by \cref{lmm: even euler number facts}, $e(\S_0,s) \equiv e(\S_1,s) \bmod{2}$, and hence
    \begin{equation*}
        e(\S_0 \cup Z \cup \S_1)\equiv-e(\S_0, s) -e(Z,s) + e(\S_1,s) \equiv 0\mod{2}.
    \end{equation*}
    Since an unknotted $\R\P^2 \subset D^4$ has $e(\R\P^2) = \pm 2$ by \cite{WhitneyDiffMan}, we can let $M \subset \del X \times I$ be a disjoint union of copies of $\R\P^2$, disjoint from $Z$, such that $e(M) = -e(\S_0 \cup Z \cup \S_1)$. Hence by Theorem \ref{mainthm: general cobordism}, there is a cobordism $Y$ from $\S_0$ to $\S_1$ extending $Z \cup M$.
\end{proof}

\subsection{Proof of Corollary~\ref{maincor: cobordism rel boundary}} \label{sec: cobordisms rel boundary}

We now consider the situation where $\del \S_0=\del \S_1$, and prove Corollary~\ref{maincor: cobordism rel boundary}. Recall that we call a cobordism from $\S_0$ to $\S_1$ a cobordism rel.\ boundary if it extends $\del \S_0 \times I$.

\begin{namedtheorem}[Corollary~\ref{maincor: cobordism rel boundary}]
    Let $X$ be an orientable 4-manifold and let $\S_0,\S_1 \subset X$ be properly embedded compact surfaces such that $\del \S_0 = \del \S_1$. Then $\S_0$ and $\S_1$ are  cobordant rel.\ boundary if and only if $[\S_0 \cup \S_1] = 0 \in H_2(X;\cyc{2})$ and $e(\S_0,s)  = e(\S_1,s)$ after an arbitrary choice of framing~$s$~of~$\del \S_0$.
\end{namedtheorem}
\begin{proof}
    Write $L \coloneq \del \S_0=\del \S_1$. We wish to apply \cref{lmm: assoc section extends} with $L_i = \del \S_i$ for $i=0,1$, but they are not disjoint links. We therefore first perturb $\S_1$ to a surface $\S'_1$ which is transverse to $\S_1$. We do this as follows. Fix a section $s'\colon \S_1 \to DN_X\S_1$ which is transverse to the zero section, such that $\restr{s'}{L} \colon L \to \restr{SN_X\S_1}{L}$ represents the framing $s$. Fix a tubular neighbourhood $(\nu \S_1, \varphi)$ of $\S_1$, and let $\S'_1 \coloneq \varphi s (\S_1)$ be the push-off of $\S_1$ in the direction of $s'$. Then $L$ and $L'\coloneq \del \S'_1$ are disjoint, and hence we may assume that $\S'_1$ is transverse to both $\S_0$ and $\S_1$.

    Let $g \colon \S_1 \times I \to X \times I$ be the isotopy given by $g(x,t) = (\varphi(t \cdot s(x)),t)$, which takes $\S_1$ to $\S'_1$. Let $Z \coloneq g(L \times I) \subset \del X \times I$, so that $Z$ is a cobordism from $L$ to $L'$.  Extend $g$ to an ambient isotopy $G \colon X \times I \to X \times I$. Then there is a cobordism rel.\ boundary $Y$ from $\S_0$ to $\S_1$ if and only if there is a cobordism $G(Y)$ from $\S_0$ to $\S'_1$ extending $Z$.    

    Let $s^Z$ be the Seifert section associated to $Z$. Then by \cref{lmm: assoc section extends}, $e(Z,s^Z) = 0$. Since $\restr{s^Z}{L} = s$, we have that $e(\S_0,s^Z) = e(\S_0,s)$. Also $e(\S'_1,s^Z) = e(\S_1,s)$, since both are given by counting the intersections between $\S'_1$ and $\S_1$ as in \cref{def: euler number}, with signs coming from a choice of orientation of $X$. Thus by Theorem \ref{mainthm: general cobordism}, there is a cobordism from $\S_0$ to $\S'_1$ extending $Z$ if and only if
    \begin{equation*}
        [\S_0 \cup \pr_X(Z) \cup \S'_1] = [\S_0 \cup \S_1] = 0 \in H_2(X;\cyc{2}),
    \end{equation*}
    and $e(\S_0,s) = e(\S_1,s)$. This proves the result.
\end{proof}

 \section{Ambient surgery}\label{sec: Ambient surgery}

For the remainder of this paper, our focus will shift from constructing cobordisms to constructing concordances. Our primary technique for proving the existence of a concordance between two connected surfaces in a simply-connected 4-manifold $X$ will be to surger an arbitrary connected cobordism $Y \subset X \times I$ between them into a concordance.
This approach dates back to Kervaire, who used it to show that any two embedded $2k$-spheres in $S^{2k+2}$ are concordant \cite{Kervaire}. Our approach will be heavily based on the approach of Sunukjian in \cite{Sunukjian} and the exposition by Klug--Miller in \cite{KlugMiller1}.

In this section, we review ambient surgery on embedded submanifolds, recalling or proving several key lemmas. All material found here is based on similar material in \cite{Sunukjian} and \cite{KlugMiller1}, but is recorded here for ease of referencing. Some results are also extended or have details added. In places, our notation differs; in particular, what we refer to as $k$-surgery (referencing the dimension of the $k$-sphere) is referred to as $(k+1)$-surgery (referencing the dimension of the $(k+1)$-handle) by both Sunukjian and Klug--Miller. We make this change to agree with the standard convention in surgery theory (c.f.\ \cite{RanickiSurgery}, \cite{LuckMacko}).

\subsection{Ambient surgery}
Let $X$ be a connected $n$-manifold, let $Y \subset X$ be a properly embedded compact $m$-manifold, and let $K \subset Y$ be an embedded $k$-sphere with $N_YK$ trivial.
A $(k+1)$-disc $\Delta \subset X$ is said to \textit{span $K$ into $X \setminus Y$} if it is embedded in $X$, meets $Y$ transversely, and $\del \Delta = \Delta \cap Y = K$.
Following \cite{KlugMiller1}, a framing $\varphi$ of $N_YK$ is called \textit{$\Delta$-admissible} if it extends to a framing of some trivial subbundle $\xi \subseteq N_X\Delta$ of rank $m-k$. Then a choice of tubular neighbourhood $(\nu \Delta,\eta)$ determines the submanifold with corners
\begin{equation*}
    Y' \coloneq \bar{Y \setminus \eta(D N_YK)} \mathop{\cup}_{\eta(SN_YK)}  \eta(S \xi ) \subseteq X.
\end{equation*}
After smoothing corners, we say that $Y'$ is the result of \textit{ambient $k$-surgery over $\Delta$ with framing $\varphi$}. Note that $Y'$ is a proper embedding of the result of abstract surgery along $K$ with framing $\varphi$, and is determined by $\Delta$ and $\varphi$ up to isotopy.

The same framing $\varphi$ of $N_YK$ may be $\Delta$-admissible for some choices of $(k+1)$-disc $\Delta$ but not others. An abstract surgery along a framed $k$-sphere can be realised ambiently in $X$ if and only if the framing is $\Delta$-admissible for some choice of disc $\Delta$ spanning $K$ into $X \setminus Y$. 

Ambient surgery can also be viewed in terms of cobordism.
Let $Y \times I \cup h^{k+1}$ be the trace of some abstract $k$-surgery on $Y$, where $h^{k+1}$ is the $(m+1)$-dimensional $(k+1)$-handle determining the surgery. This surgery can be realised ambiently if and only if the embedding $Y \times I \subset X \times I$ can be extended to an embedding of the trace of surgery with $h^{k+1} \subset X \times \{1\}$. It can be realised over $\Delta$ if and only if the image of core of $h^{k+1}$ can be taken to be $\Delta \times \{1\}$.

By considering handle decompositions and the ``rising water principle" \cite[\textsection6.2]{GompfStipsicz}, one can see that the fundamental group of the complement $\pi_1(X \setminus Y)$ is only affected by ambient surgery with either very small or very large codimension. The proof of this is identical to that of Proposition 4.1 in \cite{Sunukjian} for the case $n=m+2$.

\begin{lemma}\label{lmm: k-surgery fundamental group}
    Let $2 \leq m-k  \leq n-3$. If $Y'$ is obtained by ambient $k$-surgery on $Y \subset X$, then $\pi_1(X \setminus Y) \cong \pi_1(X \setminus Y')$.
\end{lemma}



We will mostly consider the case where $Y$ has codimension 2; that is, $n=m+2$. Then only ambient 0-surgery (that is, ambient connected sum) and ambient $(m-1)$-surgery affect $\pi_1(X \setminus Y)$. 
If $X$ is simply-connected, then $\pi_1(X \setminus Y)$ is normally generated by meridians to $Y$. By performing ambient 0-surgery along some arc, the meridians at the endpoints of the arc get identified in $\pi_1(X \setminus Y)$. After choosing a set of meridians normally generating $\pi_1(X \setminus Y)$ and identifying them in this way, we get the following result.

\begin{lemma}[{\cite[Proposition 4.3]{Sunukjian}}]\label{lmm: 1-surgery fundamental group}
    Suppose that $n=m+2$ and that $X$ is simply-connected. Then there exists a properly embedded compact $m$-manifold $Y' \subset X$ obtained from $Y$ by performing finitely many ambient 0-surgeries to $Y$, such that $\pi_1(X \setminus Y')$ is cyclic and generated by a meridian to $Y'$. We can take $Y'$ to be orientable if and only if $Y$ is.
\end{lemma}


The condition that $\pi_1(X \setminus Y)$ is cyclic and generated by a meridian to $Y$ will be especially useful, since it guarantees that every curve $\gamma \colon S^1 \to Y$ has a null-homotopic push-off in $X \setminus Y$. This is because, after choosing an arbitrary push-off, it can be altered by meridians to $Y$ until it is null-homotopic, but up to isotopy this is equivalent to simply changing the choice of push-off.

\subsection{Ambient integral Dehn surgery}\label{sec: ambient integral dehn}

We specialise to the case $n=5$, $m=3$, $k=1$. That is, we let $X$ be a 5-manifold and let $Y \subset X$ be a properly embedded compact 3-manifold. This is the case when $Y$ is a cobordism between compact surfaces properly embedded in a 4-manifold. We want to perform ambient 1-surgery (i.e.\ ambient integral Dehn surgery) along a knot $K \subset Y$ with trivial normal bundle. 

In this case, there is a $\cyc{2}$-valued obstruction for a given abstract surgery to be realisable ambiently over a given 2-disc.

\begin{lemma}[{\cite[Lemma 5.4]{KlugMiller1}}]\label{lmm: klug-miller}
    Fix a framing $\varphi$ of $N_YK$ and a 2-disc $\Delta$ spanning $K$ into $X \setminus Y$. Then the $\Delta$-admissible framings of $N_YK$ are exactly the result of either the odd integers or the even integers acting on $\varphi$.
\end{lemma}
Here the action of $\Z$ on the framings of $N_YK$ is given by  $\star$ as in \cref{sec: knot framings}, after identifying framings of $N_YK$ with isotopy classes of sections in the usual way. This is only defined up to sign without a choice of orientation on $N_YK$, but it is still possible to define the orbit of $\varphi$ under $2\Z$.
\smallbreak
\begin{proof}
    Note that $\restr{TX}{K} = N_\Delta K \oplus TK \oplus N_Y K \oplus \zeta$, where $\zeta$ is the line subbundle normal to both $Y$ and $\Delta$. See \cref{fig: subbundle diagram} for a schematic.
    
\begin{figure}[t]
    \begin{center}
    \begin{tikzpicture}[scale=1.2]

\pgfmathsetmacro{\radius}{1.5}

\fill[fill = blue!20] (0,0) -- (\radius,0) arc(0:90:\radius) -- (0,0);
\fill[fill = red!20] (0,0) -- (0,\radius) arc(90:180:\radius) -- (0,0);
\fill[fill = yellow!30] (0,0) -- (-\radius,0) arc(180:270:\radius) -- (0,0);
\fill[fill = green!30] (0,0) -- (0,-\radius) arc(-90:0:\radius) -- (0,0);

\draw[->,line width = 0.7pt] (0,0) -- (1.2*\radius,0);
\draw[->,line width = 0.7pt] (0,0) -- (-1.2*\radius,0);
\draw[->>,line width = 0.7pt] (0,0) -- (0,1.2*\radius);
\draw[->,line width = 0.7pt] (0,0) -- (0,-1.2*\radius);

\node[anchor=south] at (0,1.3*\radius) {$N_YK$};
\node[anchor=north] at (0,-1.3*\radius) {$N_\Delta K$};
\node[anchor=west] at (1.3*\radius,0) {$TK$};
\node[anchor=east] at (-1.3*\radius,0) {$\zeta$};

\node[anchor=south west] at (0.1,0.1) {$\restr{TY}{K}$};
\node[anchor=north west] at (0.1,-0.1) {$\restr{T\Delta}{K}$};
\node[anchor=south east] at (-0.1,0.1) {$\restr{N_X\Delta}{K}$};
\node[anchor=north east] at (-0.1,-0.1) {$\restr{N_XY}{K}$};

\end{tikzpicture}
    \caption{Schematic of the bundle $\restr{TX}{K}$. The four smaller subbundles $N_YK$, $TK$, $N_\Delta K$, and $\zeta$ represented by arrows are all normal to each other. The number of heads on the arrow indicates the rank. The subbundle spanned by a pair of adjacent subbundles is indicated between them.}
    \label{fig: subbundle diagram}
    \end{center}
    \vskip -0.0in
\end{figure}
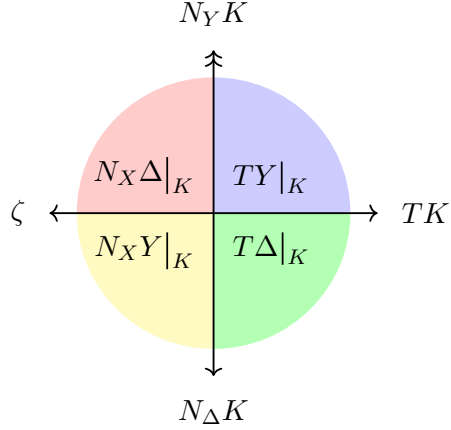
   Since both $N_X\Delta$ and $N_YK$ are orientable and all orientable vector bundles over $S^1$ are trivial, $\zeta$ must be trivial. Hence the framing $\varphi$ of $N_YK$ induces a framing of $\restr{N_X \Delta}{K} =  N_YK \oplus \zeta$, which extends over $\Delta$ if and only if it is the restriction of the unique framing of $N_X \Delta$.
    
    Recall that the framings of a trivial rank $r$ bundle over $S^1$ form a $\pi_1(\SO(r))$-torsor. Any Lie subgroup inclusion $\SO(2) \hookrightarrow \SO(3)$ induces the quotient map
    \begin{equation*}
        \pi_1(\SO(2)) \cong \Z \to \cyc{2} \cong \pi_1(\SO(3) ),
    \end{equation*}
    so two framings of the rank 2 bundle $N_YK$ induce the same framing of the rank 3 bundle $\restr{N_X \Delta}{K}$ if and only if they differ by the action of an even integer.
\end{proof}

\begin{remark}
    For more general $m$ and $k$ with $n=m+2$, the $\Delta$-admissible framings are one orbit class of framings under the action of the kernel
\begin{equation*}
    \ker \big( \pi_{k-1}(\SO(m-k+1)) \to \pi_{k-1}(\SO(m-k+2)) \big)
\end{equation*}
where the map is induced by Lie subgroup inclusion. For $m \geq 2k$, this map can be seen to be an isomorphism by considering the fibration $\SO(m-k+1) \hookrightarrow \SO(m-k+2) \to S^{m-k+1}$; so in this case, each choice of disc $\Delta$ spanning $K$ into $X \setminus Y$ has a unique framing that is $\Delta$-admissible.
\end{remark}

We will find it useful to view \cref{lmm: klug-miller} from the perspective of stable framings. Let $\xi \colon E \to X$ be a trivial vector bundle, and let $\varepsilon\colon \R \times X \to X$ be the trivial line bundle. We say that two framings of $\xi$ are \textit{stably equivalent} if they induce the same framing on $\xi \oplus \ell \varepsilon$ for sufficiently large $\ell \geq 0$, and a \textit{stable framing} of $\xi$ is a stable equivalence of framings. For $r \geq 3$, the map $\pi_1(\SO(r) ) \to \pi_1(\SO(r+1) )$ induced by Lie group inclusion is an isomorphism between cyclic groups of order 2, so there are two stable framings of a trivial vector bundle over $S^1$. So if $\rank(\xi) \geq 3$ the framings and stable framings of $\xi$ are in bijection; if $\rank(\xi) = 2$, a stable framing of $\xi$ can be thought of as a framing ``mod 2". \cref{lmm: klug-miller} says that, for a fixed 2-disc $\Delta$ spanning $K$ into $X \setminus Y$, the $\Delta$-admissible framings of $N_YK$ form one stable equivalence class, since $\varphi$ and $1 \star \varphi$ are representatives of the two distinct stable framings for any framing $\varphi$ of $N_YK$.

The following lemma is a simple but useful consequence of naturality of Stiefel--Whitney classes, and the fact that an orientable vector bundle of rank at least 3 on a closed surface is trivial if and only if its second Stiefel--Whitney class vanishes.

\begin{lemma}\label{lmm: Stiefel Whitney killing lemma}
    Let $M$ be a smooth orientable $n$-manifold with $n \geq 3$, let $\S$ be a closed oriented surface, and let $f \colon \S \to M$ be a continuous map. Then $ \left\langle w_2(M), f_*[\S] \right\rangle = 0$ if and only if $f^*TM$ is a trivial bundle.
\end{lemma}
\begin{proof}
  Consider the following commutative diagram with exact rows, coming from naturality of the universal coefficient theorem.
\[\begin{tikzcd}
	{H^2(M;\cyc{2})} & {\Hom(H_2(M;\Z),\cyc{2})} & 0 \\
	{H^2(\S;\cyc{2})} & {\Hom(H_2(\S;\Z),\cyc{2})} & 0
	\arrow[from=1-1, to=1-2]
	\arrow["{f^*}", from=1-1, to=2-1]
	\arrow[from=1-2, to=1-3]
	\arrow["{f^*}", from=1-2, to=2-2]
	\arrow["\cong", from=2-1, to=2-2]
	\arrow[from=2-2, to=2-3]
\end{tikzcd}\]
    The bottom horizontal arrow is an isomorphism as it is a surjection and both groups are cyclic of order 2.
    We consider what it means for the image of $w_2(M) \in H^2(M;\cyc{2})$ to be zero under the diagonal of the square.
    \begin{itemize}
        \item Going anticlockwise around the square, $w_2(M)$ maps to $f^*w_2(M) = w_2(f^*TM)$, which is zero if and only if $f^*TM$ is trivial.
        \item Going clockwise around the square, $w_2(M)$ maps to
    \begin{equation*}
        f^* \left\langle w_2(M), - \right\rangle = \left\langle w_2(M), f_*(-) \right\rangle \colon H_2(\S;\Z) \to \cyc{2}
    \end{equation*}
    is the zero map if and only if it is zero on the fundamental class $[\S]$.
    \end{itemize}
    By commutativity of the square, these two conditions must coincide.
\end{proof}

\cref{lmm: Stiefel Whitney killing lemma} allows Sunukjian to give a condition in terms of the action of $w_2(X \setminus Y)$ on $H_2(X \setminus Y;\Z)$ which guarantees that all abstract 1-surgeries on $Y$ can be performed ambiently in $X$ \cite[Theorem 6.1]{Sunukjian}. We repeat the proof here for the reader's convenience.

\begin{proposition}\label{lmm: ambient surgery with non-trivial sphere}
    Suppose that $X$ is orientable, that $\pi_1(X \setminus Y)$ is cyclic and generated by a meridian to $Y$, and that there exists some $\sigma \in H_2(X \setminus Y;\Z)$ with $ \left\langle w_2(X \setminus Y), \sigma \right\rangle = 1 \in \cyc{2}$.
    Then any abstract 1-surgery on $Y$ can be performed ambiently.
\end{proposition}

\begin{proof}
    Since $\pi_1(X \setminus Y)$ is cyclic, the homology group $H_2(\pi_1(X \setminus Y);\Z) =0$, and so the Hurewicz map $\pi_2(X \setminus Y) \to H_2(X \setminus Y;\Z)$ is surjective \cite[\textsection I.6 \& \textsection II.5]{BrownGroupCohomology}. Hence every class in $H_2(X \setminus Y;\Z)$ is represented by the image of a 2-sphere under a continuous map, which we may assume to be a smooth embedding since $\dim(X) \geq 5$.
    
    Let $\sigma \in H_2(X \setminus Y;\Z)$ be a class such that $\left\langle w_2(X \setminus Y), \sigma \right\rangle =1 $, and let $\S \subset X \setminus Y$ be an oriented embedded 2-sphere with $[\S] = \sigma \in H_2(X \setminus Y;\Z)$. By \cref{lmm: Stiefel Whitney killing lemma}, the bundle $\restr{T(X \setminus Y)}{\S}$ must be non-trivial. Since $TS^2$ is stably trivial, the classification of vector bundles on closed surfaces implies that  $N_{X \setminus Y}\S$ must be the unique non-trivial orientable bundle of rank 3 over $\S$.

    Fix any embedded circle $K \subset Y$ with $N_YK$ trivial, and any framing $\varphi$ of $N_YK$. As $\pi_1(X \setminus Y) $ is generated by a meridian to $Y$, $K$ has null-homotopic push-off. Thus there exists a 2-disc $\Delta$ spanning $K$ into $X \setminus Y$. If $\varphi$ is $\Delta$-admissible, we are done. If not, tube $\Delta$ to $\S$, to form a new disc $\Delta'$ spanning $K$ into $X \setminus Y$. Since $\S$ has non-trivial normal bundle, the framings of $N_YK$ which extend over a subbundle of $N_X\Delta$ cannot extend over a subbundle of $N_X\Delta'$, and vice versa. So $\varphi$ must be $\Delta'$-admissible. Thus we can perform surgery along $K$ with framing $\varphi$, either over $\Delta$ or over $\Delta'$.
\end{proof}

Recall that by \cref{lmm: 1-surgery fundamental group}, when $X$ is simply-connected we can always perform ambient 0-surgery to assume that $\pi_1(X \setminus Y)$ is cyclic and generated by a meridian to $Y$. Hence after performing appropriate 0-surgeries, the only obstruction to being able to perform ambient integral Dehn surgery arises when the action of $w_2(X \setminus Y)$ on $H_2(X \setminus Y;\Z)$ is trivial. In this case, we will use $\Pin^-$-structures to control which abstract surgeries can be realised ambiently.
 \section{Pin$^\pm$-structures and Pin$^\pm$-surgery}\label{sec: Spin}
In this section, we extend many of the results about $\Spin$-surgery found in \cite{Sunukjian} to its unoriented analogue, $\Pin^\pm$-surgery. We begin with some important background. \cref{subsec: spin via framing} gives several equivalent definitions of a $\Spin$-structure on an oriented vector bundle, including a less-used one which is most convenient for our purposes. Sections \ref{sec: pin}--\ref{sec: pin cobordism} discuss $\Pin^\pm$-structures, $\Pin^\pm$-surgery, and $\Pin^\pm$-cobordism respectively. Then in \cref{sec: ambient dehn}, we give conditions involving $\Pin^\pm$-structures for when integral Dehn surgery can be realised ambiently. This will be the key result in our proof of Theorem \ref{mainthm: general concordance}.

\subsection{$\Spin$-structures on vector bundles}\label{subsec: spin via framing}
Let $X$ be a topological space. Recall that for $n \geq 2$, $\Spin(n)$ is the unique connected double cover of $\SO(n)$, and $\Spin(1)$ is the discrete double cover of $\SO(1) = \{1\}$. Let $\xi\colon  E \to X$ be an oriented rank $n$ vector bundle, and let $f_{\xi}\colon X \to \BSO(n)$ be its classifying map. A \textit{$\Spin$-structure} on $\xi$ is a homotopy class of lifts
\[\begin{tikzcd}
	& {\BSpin(n)} \\
	X & {\BSO(n).}
	\arrow[from=1-2, to=2-2]
	\arrow[dashed, from=2-1, to=1-2]
	\arrow["{f_\xi}"', from=2-1, to=2-2]
\end{tikzcd}\]
The oriented bundle $\xi$ admits a $\Spin$-structure if and only if $w_2(\xi)=0$. If $f\colon Y \to X$ is any continuous map and $\mathfrak{s}$ is a $\Spin$-structure on $\xi$, we write $f^* \mathfrak{s}$ for the $\Spin$-structure on $f^* \xi$ given by composing $f$ with a lift $X \to \BSpin(n)$ representing $\mathfrak{s}$.

There are many equivalent definitions of $\Spin$-structures; see Chapter IV of \cite{Kirby} and \cite{ChenZinger} for surveys and further details. We will require an equivalent definition in terms of compatible stable framings over loops.

\begin{definition}[{\cite[Definition 1.3]{ChenZinger}}] \label{def: spin}
    A \textit{$\Spin$-structure} on $\xi$ is a choice of stable framing of $\gamma^* \xi$ for each continuous map $\gamma \colon S^1 \to X$, such that for any compact surface $\Sigma$ and any continuous map $f \colon \Sigma \to X$, the stable framing of $\restr[*]{f}{\del \S} \xi$ extends to a stable framing of $f^* \xi$.
\end{definition}

If $X$ has a CW-structure, this definition of a $\Spin$-structure may be seen as analogous to the well-known definition of a $\Spin$-structure on $\xi$ as a stable framing over the 1-skeleton of $X$ which extends over the 2-skeleton. One key consequence of \cref{def: spin} is that $\Spin$-structures on $\xi$ and $\xi \oplus \varepsilon$ correspond naturally, where $\varepsilon$ is the trivial line bundle over $X$. 

Henceforth, we will use \cref{def: spin} as our primary definition of a $\Spin$-structure.

If $\xi \colon E \to X$ admits a $\Spin$-structure, then the set of $\Spin$-structures on $\xi$ is an $H^1(X;\cyc{2})$-torsor. The action of $a \in H^1(X;\cyc{2})$ is given by changing the stable framing of $\gamma^* \xi$ for a map $\gamma\colon S^1 \to X$ if and only if $\langle a,\gamma_*[S^1] \rangle = 1 \in \cyc{2}$. See \textsection4.2 of \cite{ChenZinger} for a slightly different description of the action in terms of \cref{def: spin} as well as proof that the action is free and transitive, or \textsection1 of \cite{KirbyTaylor} for a more classical viewpoint. 

This means that a $\Spin$-structure is determined by the stable framings of $\xi$ on embedded circles representing a basis of $H_1(X;\cyc{2})$, as we spell out in the next lemma.

\begin{lemma}\label{lmm: spin from basis}
    Let $\xi \colon E \to X$ be an oriented rank $n$ vector bundle which admits a $\Spin$-structure. Let $\Gamma$ be a collection of embeddings $\gamma\colon S^1 \to X$ such that the tuple $\big(\gamma_*[S^1] \in H_1(X;\cyc{2}) \big)_{\gamma \in \Gamma}$ is a basis for $H_1(X;\cyc{2})$. Then the following data are equivalent:
    \begin{enumerate}[label=(\roman*)]
        \item A $\Spin$-structure on $\xi$; and
        \item A choice of stable framing of $\gamma^*\xi$ for each $\gamma \in \Gamma$.
    \end{enumerate}
\end{lemma}

\begin{proof}
    The implication (i)$\Rightarrow$(ii) is immediate by \cref{def: spin}. The fact that action of $H^1(X;\cyc{2})$ on $\Spin$-structures is both free and transitive implies that two $\Spin$-structures on $\xi$ determine the same stable framings of $\big\{ \gamma^* \xi \big\}_{\gamma \in \Gamma}$ if and only if they are equivalent $\Spin$-structures. 
\end{proof}

We will use \cref{lmm: spin from basis}(ii)$\Rightarrow$(i) to construct $\Spin$-structures with particular properties by first choosing the appropriate stable framings over a basis of $H_1(X;\cyc{2})$, then extending them to a $\Spin$-structure over the whole vector bundle.

\subsection{$\Pin^\pm$-structures}\label{sec: pin}
Let $\xi \colon E \to X$ be a rank $n$ vector bundle, not necessarily orientable. 
 Write $\det \xi \coloneq \bigwedge^{n} \xi$ for the line bundle given by the top exterior power of $\xi$. A \textit{$\Pin^-$-structure} on $\xi$ is a $\Spin$-structure on $\xi \oplus \det \xi$. A \textit{$\Pin^+$-structure} on $\xi$ is a $\Spin$-structure on $\xi \oplus 3 \det \xi$. We refer the reader to \cite{KirbyTaylor} and \cite{ChenZinger} for more details and equivalent definitions.
We will often talk about a $\Pin^\pm$-structure on $\xi$ being a $\Spin$-structure on $\xi \oplus \ell \det \xi$, without choosing a sign or specifying that $\ell = 2 \pm 1$.

 \begin{remark}
     To make sense of this definition, we must show that for $\ell \geq 1$ odd, $\xi \oplus \ell \det \xi$ is orientable and has a preferred orientation \cite[\textsection2.1]{ChenZinger}.
     To see that it is orientable, it is enough to note that the line bundle $\det \xi$ has the same orientation character as $\xi$, so $\xi \oplus \ell \det \xi$ is orientable for any $\ell \geq 1$ odd. 
      To see that it has a preferred orientation, choose an orientation of the fibre $\restr{\xi}{x}$ for each $x \in X$. This determines an orientation of $\restr{(\ell\det\xi)}{x}$ for each $x \in  X$, and for $\ell$ odd, reversing the orientation of one reverses the orientation of the other. Hence the fibre $\restr{(\xi \oplus \ell\det \xi)}{x}$ has a preferred orientation for each $x \in X$. To check that these combine to give a preferred global orientation of $\xi \oplus \ell \det \xi$, note that for any open $U \subseteq X$ such that $\restr{\xi}{U}$ is orientable, the preferred orientations of the fibres over each $x \in U$ combine to give the orientation of $\restr{(\xi \oplus \ell \det \xi)}{U}$ induced by a choice of orientation of $\restr{\xi}{U}$. There is a cover of $X$ by such open neighbourhoods since $\xi$ is locally trivial, so the preferred orientations of each fibre must combine to give a global orientation.
 \end{remark}

\begin{remark} There is also a definition of a $\Pin^\pm$-structure in terms of lifts of classifying maps \cite[\textsection1]{KirbyTaylor}. For each $n \geq 1$, there exist two Lie groups $\Pin^\pm(n)$ fitting into a central extension
    \begin{equation*}
        1 \to \cyc{2} \to \Pin^\pm(n) \to \O(n) \to 1.
    \end{equation*}
A $\Pin^\pm$-structure on a $\xi$ is equivalent to a homotopy class of lifts
\[\begin{tikzcd}
	& {\BPin^\pm(n)} \\
	X & {\BO(n)}
	\arrow[from=1-2, to=2-2]
	\arrow[dashed, from=2-1, to=1-2]
	\arrow["{f_\xi}"', from=2-1, to=2-2]
\end{tikzcd}\]
    of the classifying map $f_\xi \colon X \to \BO(n)$.

 One can show that $\xi$ admits a $\Pin^+$-structure if and only if $w_2(\xi) = 0$, and admits a $\Pin^-$-structure if and only if $w_2(\xi)+ w_1^2(\xi) = 0$. In either case, if $\xi$ admits a $\Pin^\pm$-structure, the set of $\Pin^\pm$-structures is naturally a $H^1(X;\cyc{2})$-torsor. If $\xi$ is oriented, $\det (\xi)$ is trivial, so there is a natural bijection between the $\Spin$-structures, $\Pin^-$-structures, and $\Pin^+$-structures on $\xi$.
\end{remark}

We note the following for ease of referencing.

\begin{lemma}[{\cite[Lemma 1.6]{KirbyTaylor}}]\label{lmm: inducing pin}
    Let $\xi\colon E \to X$ be a vector bundle, and let $\varepsilon$ be a trivial bundle on $X$ with a choice of stable framing. Then a choice of $\Pin^\pm$-structure on one of $\xi$ or $\varepsilon \oplus \xi$ uniquely determines a $\Pin^\pm$-structure on the other.
\end{lemma}
This follows by observing that, since $\varepsilon$ is trivial,
\begin{equation*}
        (\varepsilon \oplus \xi) \oplus \ell \det (\varepsilon \oplus \xi) = \varepsilon \oplus (\xi \oplus \ell \det \xi),
    \end{equation*}
and that $\Spin$-structures correspond naturally under stabilisation. The same result is true if $\varepsilon$ is stably trivial and stably framed.

We finish this subsection by discussing $\Pin^\pm$-structures on smooth manifolds. For a smooth manifold $Y$, a \textit{$\Pin^\pm$-structure} on $Y$ is a $\Pin^\pm$-structure on the tangent bundle $TY$. A \textit{$\Pin^\pm$-manifold} is a smooth manifold with a choice of $\Pin^\pm$-structure. The following facts can all be found in \cite{KirbyTaylor}, but are summarised for convenience.

\begin{lemma}\label{lmm: pin manifold facts}
    Let $Y$ be a smooth manifold. Let $Z \subseteq Y$ be a submanifold, and choose a stable framing of $N_YZ$. Then a choice of $\Pin^\pm$-structure on $Y$ determines a $\Pin^\pm$-structure on (i) $Z$; (ii) $\del Y$; and (iii) $Y \times I$.
\end{lemma}
\begin{proof}
    For (i), note that a $\Pin^\pm$-structure on $Y$ restricts to one on $\restr{TY}{Z}$. When taken with a stable framing of $N_YZ$, this determines a $\Pin^\pm$-structure on $TZ$ by \cref{lmm: inducing pin}. 
    Since $N_Y\del Y$ is the trivial line bundle over $\del Y$, it is naturally framed, and so (ii) follows from (i) by setting $Z = \del Y$. 
    For (iii), let $\pr_Y \colon Y \times I \to Y$ be the projection map, and note that there is a natural isomorphism
\begin{equation*}
    \pr_Y^*(\varepsilon \oplus TY) \cong T(Y \times I),
\end{equation*}
where $\varepsilon$ is the trivial line bundle on $Y$. So the result again follows by \cref{lmm: inducing pin}.
\end{proof}

\subsection{$\Pin^\pm$-surgery}\label{sec: pin surgery time}

Recall that an abstract $k$-surgery on $Y$ is specified by an embedded $k$-sphere $K \subset Y$ and a framing $\varphi$ of $N_YK$. Equivalently, it is specified by a surgery trace $Y \times I \cup h^{k+1}$, where $h^{k+1}$ is a $(k+1)$-handle attached along $K \times \{1\}$. If $Y$ is a $\Pin^\pm$-manifold, an abstract surgery is called a \textit{$\Pin^\pm$-surgery} if the corresponding $\Pin^\pm$-structure on $Y \times I$ as in \cref{lmm: pin manifold facts}(iii) extends over the surgery trace; equivalently, if the surgery trace is a $\Pin^\pm$-manifold with $Y \times \{0\}$ as a $\Pin^\pm$-submanifold of the boundary.

Another viewpoint on this is as follows. The framing $\varphi$ of $N_YK$ and a $\Pin^\pm$-structure on $Y$ induce a $\Pin^\pm$-structure on $K$ by \cref{lmm: pin manifold facts}(i). The core of the attaching handle has a unique $\Pin^\pm$-structure since it is contractible, so determines a unique $\Pin^\pm$-structure on $K$.
The abstract surgery along $K$ with framing $\varphi$ is a $\Pin^\pm$-surgery if and only if these two $\Pin^\pm$-structures on $K$ agree. This second viewpoint has the advantage that it highlights that whether an abstract surgery is a $\Pin^\pm$-surgery or not depends only on the stable equivalence class of $\varphi$.

For $k \neq 1$, the $k$-sphere has a unique $\Pin^\pm$-structure, and hence every abstract $k$-surgery is a $\Pin^\pm$-surgery. For $k=1$, there are two $\Pin^\pm$-structures on the circle, corresponding to the two stable framings of $TS^1$, so there is a $\cyc{2}$-valued obstruction to an abstract 1-surgery being a $\Pin^\pm$-surgery. 
The result of $\Pin^\pm$-surgery inherits a unique $\Pin^\pm$-structure for $k \geq1$.

\subsection{$\Pin^\pm$-cobordism}\label{sec: pin cobordism}
We say that closed $n$-dimensional $\Pin^\pm$-manifolds $M_0$, $M_1$ are \textit{$\Pin^\pm$-cobordant} if there exists an $(n+1)$-dimensional $\Pin^\pm$-manifold $W$
and a $\Pin^\pm$-diffeomorphism $\del W \to M_0 \sqcup M_1$. 
Equivalently, we say that $M_0$ and $M_1$ are $\Pin^\pm$-cobordant if there exists a finite sequence of $\Pin^\pm$-surgeries on $M_0$ yielding $M_1$. The $\Pin^\pm$-cobordism classes of closed $n$-manifolds form a group under disjoint union, denoted $\Omega^{\Pin^\pm}_n$.

In low dimensions, it becomes more convenient to work with $\Pin^-$-structures than $\Pin^+$-structures, since all compact surfaces and 3-manifold admit $\Pin^-$-structures, while not all admit $\Pin^+$-structures. Moreover, $\Pin^-$-structures on compact surfaces and 3-manifolds behave well under cobordism.

\begin{proposition}[{\cites{BrownInvariant, KirbyTaylor}}]\label{prop: Pin diffeomorphism}
    All compact surfaces admit a $\Pin^-$-structure. Two closed connected $\Pin^-$-surfaces are $\Pin^-$-diffeomorphic if and only if they are abstractly diffeomorphic and $\Pin^-$-cobordant. All compact 3-manifolds admit a $\Pin^-$-structure, and $\Omega_3^{\Pin^-} =0$.
\end{proposition}

Since $\Omega_3^{\Pin^-} =0$, all compact 3-manifolds with chosen $\Pin^-$-structures and $\Pin^-$-diffeomorphic boundaries are related by $\Pin^-$-surgeries. The following lemma shows that these can all be chosen to be 1-surgeries. This was originally shown for closed oriented 3-manifolds by Kaplan using Kirby calculus \cite{Kaplan}; a Morse-theoretic is proof for all compact oriented 3-manifolds is given for Lemma 5.4 of \cite{Sunukjian}. 

\begin{proposition}\label{lmm: pin cobordism with boundary}
    Let $M_0$, $M_1$ be compact connected 3-dimensions $\Pin^-$-manifolds. Suppose that $\del M_0$ and $\del M_1$ are $\Pin^-$-diffeomorphic, and that $M_0$ and $M_1$ are either both orientable or both non-orientable. Then there is a sequence of $\Pin^-$-surgeries taking $M_0$ to $M_1$, all of which can be taken to be 1-surgeries.
\end{proposition}
\begin{proof}
    The case that $M_0$ and $M_1$ are orientable is Lemma 5.4 of \cite{Sunukjian}, so we assume that $M_0$ and $M_1$ are both non-orientable.
    Fix a $\Pin^-$-diffeomorphism $\psi \colon \del M_0 \to \del M_1$, and let $N$ be the closed 3-manifold obtained by gluing 
    \begin{equation*}
        N \coloneq M_0 \times \{0\} \mathop{\cup}_{\id \times \{0\} } \del M_0 \times [0,1] \mathop{\cup}_{\psi \times \{1\} } M_1 \times \{1\}.
    \end{equation*}
    Note that if $M_0$ and $M_1$ are closed, then $N$ is the disjoint union $M_0 \sqcup M_1$. 
    Since $\psi$ preserves the $\Pin^-$-structure on the boundaries, $N$ inherits a $\Pin^-$-structure. Then since $\Omega_3^{\Pin^-} = 0$, we can find a compact connected 4-dimensional $\Pin^-$-manifold $W$ such that $\del W = N$ as $\Pin^-$-manifolds. Take a Morse function $f \colon W \to [0,1]$ such that $f\inv(\{i\}) = M_i \times \{i\} \subseteq \del W$ for $i=0,1$, and $f(\del M_0 \times \{t\}) = t$ for $t \in [0,1]$. This gives a handle decomposition of $W$ rel.\ $M_0 \times \{0\}$.

    
    We explain how to modify this handle decomposition of $W$ rel.\ $M_0 \times \{0\}$ into one with only 2-handles. First, since $W$ is connected, every 0-handle has a cancelling 1-handle, and every 4-handle has a cancelling 3-handle. Hence we can arrange that there are no 0-handles or 4-handles.
    For each 1-handle in $W$, find an embedded circle $K$ which is the union of the core of the 1-handle and an arc connecting the two points in the attaching sphere. 
    We can choose $K$ such that $N_WK$ is trivial, by taking the band sum of $K$ with an orientation-reversing loop in $M_0 \times \{0\}$ if necessary. Perform 1-surgery along $K$ to obtain a new 4-manifold $W'$, which has handle decomposition with one fewer 1-handle than $W$ (see \cite[\textsection5.2]{GompfStipsicz} -- in terms of Kirby calculus, we are replacing a dotted circle with a 0-framed circle). This surgery can always be taken to be a $\Pin^-$-surgery, since $N_WK$ has two framings, one of which yields $\Pin^-$-surgery while the other does not.
    
    Repeating this for every 1-handle, we may arrange the handle decomposition of $W$ rel.\ $M_0 \times \{0\}$ has no 1-handles. Turning the handle decomposition upside down, we may also arrange that it also has no 3-handles. Reinterpreting the handle decomposition of $W$ as a sequence of $\Pin^-$-surgeries on $M_0$ yields the result. 
\end{proof}

\subsection{Ambient integral Dehn $\Pin^\pm$-surgery}\label{sec: ambient dehn}
We return to ambient integral Dehn surgery as in \cref{sec: ambient integral dehn}. 
Let $X$ be a 5-manifold and let $Y \subset X$ be a properly embedded compact 3-manifold.

Recall that for an embedded circle $K \subset Y$ and 2-disc $\Delta$ spanning $K$ into $X \setminus Y$, we say that a framing $\varphi$ of $N_YK$ is $\Delta$-admissible if abstract surgery along $K$ with framing $\varphi$ can be realised ambiently over $\Delta$. \cref{lmm: klug-miller} says that exactly one of the two stable framings of $N_YK$ is given by $\Delta$-admissible framings. For a fixed $\Pin^\pm$-structure on $Y$, this $\Delta$-admissible stable framing either corresponds to the stable framing inducing $\Pin^\pm$-surgery on $K$, or the other stable framing. This is summarised in the following lemma, analogous to Lemma 5.2 of \cite{Sunukjian}.

\begin{lemma}\label{lmm: ambient pin surgery}
    Let $K \subset Y$ be an embedded circle with $N_YK$ trivial, and let $\Delta$ be a 2-disc spanning $K$ into $X \setminus Y$. Fix a $\Pin^\pm$-structure on $Y$. Then exactly one of the following is true.
    \begin{enumerate}[label=(\roman*)]
        \item A framing $\varphi$ of $N_YK$ is $\Delta$-admissible if and only if abstract surgery along $K$ with framing $\varphi$ is a $\Pin^\pm$-surgery.
        \item For any framing $\varphi$ of $N_YK$, exactly one of the following is true: either $\varphi$ is $\Delta$-admissible; or abstract surgery along $K$ with framing $\varphi$ is a $\Pin^\pm$-surgery.
    \end{enumerate}
\end{lemma}

A $\Pin^-$-structure on $Y \subset X$ can be thought of in terms of framings of $TX$ over loops in $Y$. This is a generalisation of the ideas in \cite{BlanloeilSaeki}, but for oriented 5-manifolds other than $D^5$.

\begin{lemma}\label{lmm: embedded pin structures}
    A $\Pin^-$-structure on $Y$ is equivalent to a choice of framing of $\gamma^* TX$ for each continuous map $\gamma \colon S^1 \to Y$, such that for any compact surface $\S$ and any continuous map $f\colon \S \to Y$, the framing of $\restr[*]{f}{\del \S} TX$ extends to a framing of $f^*TX$.
\end{lemma}
\begin{proof}
    We first show that over loops and compact surfaces with non-empty boundary, we can identify the vector bundles $TY \oplus \det TY \oplus \varepsilon^Y$ and $TX$, where $\varepsilon^Y$ is the trivial line bundle over $Y$. To this end, let $M$ be either $S^1$ or a compact surface with non-empty boundary, and let $g\colon M \to Y$ be a continuous map. Since $M$ has the homotopy-type of a 1-complex, the rank 2 bundle $g^*N_XY$ admits a section. Hence it splits as $g^* N_XY = \lambda \oplus \varepsilon^M$, where $\varepsilon^M$ is a trivial line bundle over $M$ and $\lambda$ is some other line bundle.
    Since $TX$ is orientable, the orientation characters of $N_XY$ and $TY$ agree. Hence the orientations characters of $\det(g^* N_XY)$ and $\det(g^*TY)$ agree, and in particular agree with the orientation character of $\lambda$. Then since line bundles are determined by their orientation characters, we can identify
    \begin{equation*}
        \lambda = \det(g^* N_XY) = \det (g^* TY) = g^*(\det TY).
    \end{equation*}
    Therefore, we can identify
    \begin{equation*}
        g^* TX = g^*N_XY \oplus g^*TY = \lambda \oplus \varepsilon^M \oplus g^*TY = g^*(TY \oplus \det TY \oplus \varepsilon^Y) ,
    \end{equation*}
    as required. Finally, since $\rank(g^*TX)=5 \geq 3$ and $M$ has the homotopy-type of a 1-complex, there is a bijection between the framings and the stable framings of $g^*TX$. Therefore a stable framing of $g^*(TY \oplus \det TY)$ determines a stable framing of $g^* TX$, and hence a framing of $g^*TX$; and vice versa.

    The lemma then follows by considering the definitions. A $\Pin^-$-structure on $Y$ is a choice of stable framings of $\gamma^*(TY \oplus\det TY)$ for each map $ \gamma \colon S^1 \to Y$, which are compatible in the sense of \cref{def: spin}. By the argument above, this is equivalent to a choice of stable framings of $\gamma^*TX$ for each $\gamma \colon S^1 \to Y$, which are compatible in the sense of the statement of the lemma.
\end{proof}

\begin{corollary}\label{lmm: pin from basis}
    Let $\gamma_1,\ldots,\gamma_n\colon S^1 \to X$ be a collection of embeddings such that the tuple $\big([\gamma_i(S^1)] \in H_1(Y;\cyc{2}) \big)_{1 \leq i \leq n}$ is a basis for $H_1(Y;\cyc{2})$. Then the following data are equivalent:
    \begin{enumerate}[label=(\roman*)]
        \item A $\Pin^-$-structure on $Y$; and
        \item A choice of framing of $\gamma_i^*TX$ for each $i=1,\ldots,n$.
    \end{enumerate}
\end{corollary}
\begin{proof}
    This follows immediately from \cref{lmm: spin from basis} and \cref{lmm: embedded pin structures}.
\end{proof}

Now assume that $\pi_1(X \setminus Y)$ is cyclic and generated by a meridian to $Y$, so that every embedded circle $K \subset Y$ has a null-homotopic push-off. In exactly the cases where \cref{lmm: ambient surgery with non-trivial sphere} does not apply, we construct a $\Pin^-$-structure on $Y$ such that all abstract $\Pin^-$-surgeries on $Y$ can be realised ambiently. This is analogous to Proposition 5.1 of \cite{Sunukjian}, although the proof given below differs substantially from the argument given there (see \cref{rmk: proof of 5.1}). The same result holds for $\Pin^+$-structures under the assumption that $Y$ admits a $\Pin^+$-structure. We do not require this result however, so we omit it.

\begin{proposition}\label{prop: main sunukjian step}
    Suppose that $X$ is orientable, that $\pi_1(X \setminus Y)$ is cyclic and generated by a meridian to $Y$, and that $\left\langle w_2(X \setminus Y), \sigma \right\rangle = 0 \text{ for all }\sigma \in H_2(X \setminus Y;\Z).$
    Then $Y$ admits a $\Pin^-$-structure such that any abstract 1-surgery which is a $\Pin^-$-surgery can be performed ambiently.
\end{proposition}

\begin{proof}
   Let us first unpack what is required for such a $\Pin^-$-structure on $Y$ to exist. 
   Fix an embedded circle $K \subset Y$ and a 2-disc $\Delta$ spanning $K$ into $X \setminus Y$, which exists since $K$ has a null-homotopic push-off by assumption. The disc $\Delta$ determines a unique framing of $\restr{TX}{K}$ which extends over $\Delta$. 
   By \cref{lmm: embedded pin structures}, a $\Pin^-$-structure on $Y$ also determines a framing of $\restr{TX}{K}$. If these framings agree, all ambient 1-surgeries over $\Delta$ are $\Pin^-$-surgeries; otherwise, none are. 

   Therefore, given a $\Pin^-$-structure on $Y$, all abstract 1-surgeries which are $\Pin^-$-surgeries can be performed ambiently in $X$ if the following holds: for each embedded curve $K \subset Y$, there is a 2-disc $\Delta$ spanning $K$ into $X \setminus Y$ such that the framing of $\restr{TX}{K}$ determined by the $\Pin^-$-structure extends over $\Delta$. (This is more than is strictly necessary, since we need only consider $K$ such that $N_YK$ is trivial).

   We now construct such a $\Pin^-$-structure using \cref{lmm: pin from basis}. Let $K_1,\ldots,K_n \subset Y$ be embedded circles such that the tuple $\big( [K_i] \in H_1(Y;\cyc{2}) \big)_{1 \leq i \leq n}$ is a basis for $H_1(Y;\cyc{2})$. For each $i=1,\ldots,n$, let $\Delta_i$ be a 2-disc spanning $K_i$ into $X \setminus Y$, which exists since $\pi_1(X \setminus Y)$ is generated by a meridian to $Y$. By \cref{lmm: spin from basis}, give $Y$ the $\Pin^-$-structure such that the determined framing of $\restr{TX}{K_i}$ extends over $\Delta_i$ for all $i=1,\ldots,n$. For ease of referencing, we call this $\Pin^-$-structure $\mathfrak{p}$.

   It remains to show that for all other embedded circles $K \subset Y$, there is a disc $\Delta$ spanning $K$ into $X \setminus Y$ such that the framing of $\restr{TX}{K}$ determined by $\mathfrak{p}$ extends over $\Delta$. To this end, fix an embedded circle $K \subset Y$ and an orientation of $K$.
   We will construct the following three things:
   \begin{itemize}
       \item an embedding $h \colon D^2 \to X$ such that $h(D^2)$ spans $K$ into $X \setminus Y$;
       \item an abstract oriented surface $\S$ with $\del \S \cong S^1$; and
       \item a continuous map $f \colon \S \to X$ such that $f(\del \S) = K$.
   \end{itemize}
   These will satisfy two conditions: first, that the framing of $\restr{TX}{K}  =\restr[*]{f}{\del \S}TX$ determined by $\mathfrak{p}$ extends over $\S$; and second, that the map
   \begin{equation*}
       f \mathop{\cup}_K h \colon \S \mathop{\cup}_{S^1} D^2 \to X
   \end{equation*}
   from the closed oriented surface $\S \cup D^2$ can be homotoped to an embedding in $X \setminus Y$. Then \cref{lmm: Stiefel Whitney killing lemma} and the assumption that $w_2(X \setminus Y)$ acts trivially on $H_2(X \setminus Y;\Z)$ will force $(f \cup h)^* TX$
   to be trivial. Then since the framing of
   $\restr{TX}{K} = \restr[*]{h}{S^1}TX$ determined by $\mathfrak{p}$ extends over $\S$, it must also extend over $\Delta \coloneq h(D^2)$ as required.

    Let us first construct $\S$ and the continuous map $f\colon \S \to X$. We will do this in several steps. First, we will find a compact orientable surface $\wh{\S}$ with at least 3 boundary components, and a continuous map $f \colon \wh{\S} \to Y$. We will then cap off all but one boundary component of $\wh\S$ with discs to obtain $\S$, and extend $f$ over these discs to build the required map $\S \to X$.

   We construct $f \colon \wh{\S} \to Y$ as follows.
   Without loss of generality, suppose $K_1,\ldots,K_n$ are ordered such that, for some $0 \leq m \leq n$,
   \begin{equation*}
       [K] = [K_1] + \cdots + [K_m] \in H_1(Y; \cyc{2}).
   \end{equation*}
   Then, after choosing an orientation of $K_i$ for each $i=1,\ldots,m$, there exists some $\beta \in H_1(Y;\Z)$ such that 
   \begin{equation*}
       [K] + [K_1] + \cdots + [K_m] + 2 \beta = 0 \in H_1(Y;\Z).
   \end{equation*}
   Let $L \subset Y$ be an oriented embedded curve such that $\beta = [L] \in H_1(Y;\Z)$, which exists since $\dim Y=3$. Then there exists an oriented surface $\wh{\S}$ with $m+3$ boundary components, which we label $C$, $C_1,\ldots,C_m$, $Q_1$, and $Q_2$, as well as a continuous map $f \colon \wh{\S} \to Y$ such that the restriction of $f$ to each boundary component of $\wh\S$ is an embedding, and
   \begin{equation*}
       f(C) = K,\quad f(C_i) = K_i\, \text{ for $i=1,\ldots,m$, }\quad \text{and} \quad f(Q_1)=f(Q_2) = L,
   \end{equation*}
   as oriented embedded circles.

   We can now define the oriented surface $\S$ to be the oriented surface formed by capping off the boundary components $C_1,\ldots,C_m,Q_1$, and $Q_2$ with discs, so that $\del \S = C$.

    Next, we describe how to extend $f$ to a map $f \colon \S \to X$. Choose a disc $\Delta_L$ spanning $L$ into $X \setminus Y$, which again exists since $\pi_1(X \setminus Y)$ is generated by a meridian to $Y$. 
    Then extend $f \colon \wh\S \to Y$ to a map $f \colon \S \to X$ such that the disc capping off $C_i$ gets mapped to $\Delta_i $ for $i=1,\ldots,m$, and the two discs capping off $Q_1$ and $Q_2$ both get mapped to $\Delta_L$. Thus $f\colon \S \to X$ is a continuous map from an oriented surface with $\del \S = C \cong S^1$ such that $f(\del \S) =K$, as was required.

    Next, we show that the framing of $\restr{TX}{K}  =\restr[*]{f}{\del \S}TX$ determined by the $\Pin^-$-structure $\mathfrak{p}$ extends over $\S$. By \cref{lmm: embedded pin structures}, the framing of $\restr[*]{f}{\del \wh{\S}} TX$ determined by $\mathfrak{p}$ extends over $\wh{\S}$, since the image $f(\wh{\S})$ lies entirely in $Y$. By construction, it extends further over the discs capping off $C_1,\ldots,C_m$. Suppose that it does not extend over all of $\S$; that is, the framing of $\restr{TX}{L} =  \restr[*]{f}{Q_j} TX$ given by the $\Pin^-$-structure $\mathfrak{p}$ does not extend over the discs capping off $Q_j$ for $j=1,2$. Then we could change the framing of $\restr{TX}{L}$ we are considering. This would change the framing of $\restr[*]{f}{Q_j}TX$ for both $j=1$ and $j=2$, and hence not affect the fact that the framing of $\restr[*]{f}{\del \wh\S} TX$ under consideration extends over $\wh{\S}$. We may therefore assume that the framing of $\restr[*]{f}{\del \wh\S}TX$ under consideration extends over $\S$. In particular, the framing of $\restr[*]{f}{\del \S}TX$ determined by $\mathfrak{p}$ extends over $\S$, as required.

    It remains to construct the embedding $h \colon D^2 \to X$ such that $h(D^2)$ spans $K$ into $X \setminus Y$, and such that map $f \cup h \colon \S \cup D^2 \to X$ can be homotoped off $Y$. We first show that $f\colon \S \to X$ can be homotoped off $Y$, and use this to construct the embedding $h$.
    
    Fix a tubular neighbourhood $(\nu Y, \varphi)$ of $Y$. By taking $\nu Y$ sufficiently small, we may assume that it intersects each of the discs $\Delta_1,\ldots,\Delta_m$, and $\Delta_L$ only in a collar of their boundary. Then the identification $\varphi\colon DN_XY \to \nu Y$ specifies a section of $\restr[*]{f}{C_i}SN_XY$ for each for $i=1,\ldots,n$ given by the normal direction into $\Delta_i$, as well as a section of $\restr[*]{f}{Q_j} SN_XY$ for $j=1,2$ given by the direction into $\Delta_L$. These combine to give a section of $\restr[*]{f}{\del \wh{\S} \setminus C}SN_XY$, which can be extended to a section
    \begin{equation*}
        s\colon \wh{\S} \to \restr[*]{f}{\wh\S}SN_XY,
    \end{equation*}
    since it has not been prescribed on the boundary component $C$. So $\restr{f}{\wh\S}$ can be pushed off $Y$ in the direction of $s$. But by construction, this push-off is into the discs $\Delta_1,\ldots,\Delta_m$, and $\Delta_L$ which only meet $Y$ transversely in their boundaries, so it extends to a push-off of $f\colon \S \to X$ off $Y$. 
    For ease of reference, we call this push-off $f^s_+\colon \S \to \bar{X \setminus \nu Y}$.
   
   We now construct the embedding $h\colon D^2 \to X$. Let
   \begin{equation*}
        K^s_+ \coloneq f^s_+(\del \S) \subset S \nu Y
   \end{equation*}
   be the embedded circle given by pushing $K$ off $Y$ in the direction of $s$. Since $K^s_+$ is bounded by $f^s_+(\S)$, the circle $K^s_+$ is null-homologous in $X \setminus Y$. But since $\pi_1(X \setminus Y)$ is cyclic and hence abelian, $K^s_+$ is in fact null-homotopic in $X \setminus Y$. Thus we can find a null-homotopy of $K^s_+$ in $X \setminus Y$, given by a map $D^2 \to X \setminus Y$ which restricts to $K^s_+$ on the boundary. We may extend this null-homotopy linearly through $\nu Y$ so that it restricts on the boundary $K \subset Y$, and perturb it to give a smooth embedding $h\colon D^2 \to X$. Then $h(D^2)$ spans $K$ into $X \setminus Y$ as required. 

    The final thing that remains to check is that the push-off $f^s_+ \colon \S \to X$ extends to a push-off of $f \cup h \colon \S \cup D^2 \to X$. But this is true by the construction of $h$, since $f^s_+$ pushes $\del \S$ off in the inwards direction of $h(D^2)$.

    We have therefore constructed both the embedding $h \colon D^2 \to X$ and the map $f \colon \S \to X$ satisfying the conditions required.
   
    The proof the concludes as outlined at the start. Since $\S \cup h(D^2)$ is a closed orientable surface and the map $f \cup h$ can be homotoped into $X \setminus Y$, \cref{lmm: Stiefel Whitney killing lemma} and the assumption that $\langle w_2(X \setminus Y),\sigma \rangle = 0$ for all $\sigma \in H_2(X \setminus Y;\Z)$ implies that $(f\cup h)^*TX$ is trivial. Since we have shown that the framing of $\restr[*]{f}{\del \S}$ determined by the $\Pin^-$-structure $\mathfrak{p}$ on $Y$ extends over $\S$, this implies that the framing of $\restr[*]{h}{S^1}TX$ extends over $D^2$. That is, the framing of $\restr{TX}{K}$ determined by $\mathfrak{p}$ extends over the disc $h(D^2)$ which spans $K$ into $X \setminus Y$. Therefore any abstract $\Pin^-$-surgery on $K$ can be realised over $h(D^2)$. 

    Since $K$ was arbitrary, we have shown that any abstract 1-surgery which is a $\Pin^-$-surgery with respect to $\mathfrak{p}$ can be realised ambiently in $X$.
\end{proof}

\begin{remark}
    The statement of Proposition 5.1 of \cite{Sunukjian} omits the assumption that $\pi_1(X \setminus Y)$ is cyclic and generated by a meridian to $Y$, or any possible weaker assumption. Without such an assumption, one must check that the chosen stable framing of $N_XY$ over each embedded circle extends over any map from any compact surface, not just orientable surfaces. This is the claim made in point (3) of the proof, but in fact the proof only verifies it for orientable surfaces. This has no effect on the rest of the proof of Theorem 6.1 of \cite{Sunukjian}, since Proposition 5.1 is only applied in cases where the missing assumption is satisfied.
\end{remark}

\begin{remark}\label{rmk: proof of 5.1}
    There is also an error in point (1) of the proof of Proposition 5.1 of \cite{Sunukjian}. The proof claims that, for any two discs $\Delta_1$ and $\Delta_2$ spanning some embedded circle $K \subset Y$ into $X \setminus Y$, the 2-sphere $\Delta_1 \cup \Delta_2$ can be pushed off $Y$. This is used to argue, similarly to the proof of \cref{prop: main sunukjian step} above, that the unique framings of some subbundles of $TX$ on $\Delta_1$ and $\Delta_2$ restrict to the same framing on $K$. This is true whenever $H_2(X;\Z) = 0$ by an argument involving Alexander duality, and in particular is true in Kervaire's case of $X= D^5$. In general however, it is not always possible to push $\Delta_1 \cup \Delta_2$ off $Y$. This is fixed by arguing in the proof of \cref{prop: main sunukjian step} above, carefully constructing a disc $\Delta$ spanning each embedded curve $K$ into $X \setminus Y$.
    
    We give an explicit counterexample. Let $X = S^2 \times S^2 \times I$ and $Y = S^2 \times \{\pt\} \times I$. Then the complement $X \setminus Y \cong S^2 \times \mathring{D}^2 \times I$, and so $\pi_1( X \setminus Y) = \{1\}$ and $w_2(X \setminus Y) = 0$. Let $K \subset Y$ be any embedded circle, and let $s,s' \colon K \to \restr{SN_XY}{K}$ be two sections of the circle normal bundle. After choosing a tubular neighbourhood $(\nu Y, \varphi)$ of $Y$, let $K_+^s \coloneq \varphi s(K)$ and $K_+^{s'} \coloneq \varphi s'(K)$ be the push-offs of $K$ in the directions of $s$ and $s'$, respectively. Since $X \setminus Y$ is simply-connected, we can find two embedded discs in $X \setminus Y$ with boundaries $K_+^s$ and $K_+^{s'}$ respectively. By extending these discs by an annulus through $\nu Y$ to have boundary $K$, we find two 2-discs $\Delta, \Delta'$ spanning $K$ into $X \setminus Y$. 
    
    The algebraic intersection number $[Y] \cdot [\Delta \cup \Delta'] \in \Z$ is exactly the difference in the framings $s$ and $s'$. To see this, let $\S \subset X$ be the embedded sphere given by perturbing $\Delta \cup \Delta'$ in the direction of $s$ near $K$. Then $\S$ intersects $Y$ only in transverse isolated points, and the signed count of the intersection points agrees with the number of times $s'$ twists positively around $s$.

    In particular, if $s$ and $s'$ are not isotopic, then $\Delta \cup \Delta'$ cannot be pushed off of $Y$.
\end{remark}

Combining \cref{lmm: ambient surgery with non-trivial sphere} and \cref{prop: main sunukjian step} shows that we can always use $\Pin^-$-structures to control ambient surgeries.

\begin{corollary}\label{cor: pin surgery}
    Suppose that $X$ is orientable, and that $\pi_1(X \setminus Y)$ is cyclic and generated by a meridian to $Y$.
    Then $Y$ admits a $\Pin^-$-structure such that any abstract 1-surgery which is a $\Pin^-$-surgery can be performed ambiently.
\end{corollary}

\begin{proof}
    If $w_2(X \setminus Y)$ acts trivially on $H_2(X \setminus Y;\Z)$, this is \cref{prop: main sunukjian step}. If it acts non-trivially, all abstract 1-surgeries can be realised ambiently, so we may choose any $\Pin^-$-structure on $Y$.
\end{proof}

 \section{Concordance of surfaces in simply-connected 4-manifolds}\label{sec: final concordance}
In this final section, we prove Theorem \ref{mainthm: general concordance} as an application of \cref{cor: pin surgery}.

\begin{namedtheorem}[Theorem \ref{mainthm: general concordance}]
    Let $X$ be a simply-connected 4-manifold and let $\S_0,\S_1 \subset X$ be properly embedded compact connected surfaces. Let $Z \subset \del X \times I$ be a concordance from $\del \S_0$ to $\del \S_1$. Then $Z$ extends to a concordance from $\S_0$ to $\S_1$ if and only if $\S_0 \cong \S_1$ and either:
    \begin{enumerate}[label=(\roman*)]
        \item $\S_0$ and $\S_1$ are orientable, and $Z$ extends to an orientable cobordism from $\S_0$ to $\S_1$; or
        \item $\S_0$ and $\S_1$ are non-orientable, and $Z$ extends to a cobordism from $\S_0$ to $\S_1$.
    \end{enumerate}
\end{namedtheorem}

\begin{proof}
    The forwards direction is immediate, since a concordance is a special type of cobordism. So we prove the reverse. Let $Y \subset X \times I$ be a cobordism from $\S_0$ to $\S_1$ extending $Z$, orientable if $\S_0$ and $\S_1$ are. By tubing together components, we may assume that $Y$ is connected. We will show that if $X$ is simply-connected, then there is a sequence of surgeries taking $Y$ to $\S_0 \times I$ which can be realised ambiently, giving a concordance from $\S_0$ to $\S_1$.
    
    By \cref{lmm: 1-surgery fundamental group}, we can perform ambient 0-surgery on $Y$ to assume that $\pi_1(X \times I \setminus Y)$ is cyclic and generated by a meridian to $Y$ without affecting the orientability of $Y$. Then by \cref{cor: pin surgery}, we can give $Y$ a $\Pin^-$-structure such that all abstract 1-surgeries which are $\Pin^-$-surgeries can be realised ambiently.
    
    In order to apply \cref{lmm: pin cobordism with boundary} to find a sequence of 1-surgeries taking $Y$ to $\S_0 \times I$, we must show that there is a $\Pin^-$-structure on $\S_0 \times I$ such that $\del Y$ and $\del(\S_0 \times I)$ are $\Pin^-$-diffeomorphic. We consider two cases.
    
    \begin{itemize}
        \item If $\S_0$ and $\S_1$ are closed, give $\S_0 \times \{0\}$ the $\Pin^-$-structure restricted from the one on $\del Y$. Extend this to a $\Pin^-$-structure on $\S_0 \times I$ by \cref{lmm: pin manifold facts}(iii). Then both $\S_1 \times \{1\} \subset \del Y$ and $\S_0 \times \{1\} \subset \S_0 \times I$ are $\Pin^-$-cobordant to $\S_0 \times \{0\}$, so they are $\Pin^-$-diffeomorphic by \cref{prop: Pin diffeomorphism}. In particular, $\del Y$ and $\del (\S_0 \times I)$ are $\Pin^-$-diffeomorphic.

        \item If $\S_0$ and $\S_1$ have non-empty boundary, then $\del Y$ is connected. Then for any $\Pin^-$-structure on $\S_0\times I$, the $\Pin^-$-structures on $\del Y$ and $\del (\S_0 \times I)$ are both null-cobordant. So again, \cref{prop: Pin diffeomorphism} says that $\del Y$ and $\del (\S_0 \times I)$ are $\Pin^-$-diffeomorphic. 
    \end{itemize}

    Thus \cref{lmm: pin cobordism with boundary} gives a sequence of $\Pin^-$-surgeries taking $Y$ to $\S_0 \times I$. Each these are 1-surgeries, and we can perform each of these surgeries ambiently in $X \times I$. In fact, we can perform them all simultaneously, since 2-discs are generically disjoint in $X \times I \setminus Y$ and 1-surgery does not change the fundamental group by \cref{lmm: k-surgery fundamental group}. 

    Performing these surgeries ambiently in $X \times I$ yields a properly embedded 3-manifold in $X \times I$ with boundary $\del Y$, diffeomorphic to $\S_0 \times I$. This is the required concordance from $\S_0$ to $\S_1$.
\end{proof}

\bibliography{refs}

\end{document}